\documentclass[11pt,reqno,a4paper]{amsart}
\usepackage[T1]{fontenc}
\usepackage{lmodern}
\usepackage{mathtools,amssymb,tikz-cd}
\usepackage[margin=1.05in]{geometry}
\usepackage{microtype}
\usepackage{enumitem}
\usepackage{needspace}
\usepackage[colorlinks=true,linkcolor=blue,citecolor=blue,urlcolor=blue]{hyperref}
\numberwithin{equation}{section}
\newtheorem{theorem}{Theorem}[section]
\newtheorem{proposition}[theorem]{Proposition}
\newtheorem{lemma}[theorem]{Lemma}
\newtheorem{corollary}[theorem]{Corollary}
\theoremstyle{definition}
\newtheorem{definition}[theorem]{Definition}
\newtheorem{notation}[theorem]{Setting}
\theoremstyle{remark}

\DeclareMathOperator{\Pic}{Pic}
\DeclareMathOperator{\Spec}{Spec}
\DeclareMathOperator{\Spf}{Spf}
\DeclareMathOperator{\Hom}{Hom}

\DeclareMathOperator{\im}{im}
\DeclareMathOperator{\ord}{ord}
\DeclareMathOperator{\divisor}{div}
\DeclareMathOperator{\Supp}{Supp}

\newcommand{\OO}{\mathcal O}
\newcommand{\II}{\mathcal I}
\newcommand{\FF}{\mathcal F}
\newcommand{\JJ}{\mathcal J}
\newcommand{\MM}{\mathcal M}
\newcommand{\LL}{\mathcal L}
\newcommand{\YY}{\mathfrak Y}
\newcommand{\Gm}{\mathbb G_m}
\newcommand{\Ga}{\mathbb G_a}
\newcommand{\Q}{\mathbb Q}
\newcommand{\Z}{\mathbb Z}
\newcommand{\RR}{\mathbb R}
\newcommand{\sheafHom}{\mathcal Hom}
\newcommand{\sheafExt}{\mathcal Ext}

\setlist[enumerate,1]{label=\textup{(\roman*)},leftmargin=*,itemsep=3pt}
\title[Abundance for threefolds in positive characteristic]{Abundance for threefolds in positive characteristic}
\author{Fabio Bernasconi, Jakub Witaszek and Zheng Xu}
\date{}
\hypersetup{pdftitle={Abundance for threefolds in positive characteristic},
  pdfauthor={Zheng Xu}}

\address{Dipartimento di Matematica ``Guido Castelnuovo'', SAPIENZA Universita di Roma, 
Piazzale Aldo Moro 5, I-00185, Roma}
\email{fabio.bernasconi@uniroma1.it}

\address{Northwestern University, Department of Mathematics, Lunt Hall, 2033 Sheridan Road, Evanston, IL 60208, USA}
\email{jakub.witaszek@northwestern.edu}

\address{Department of Mathematics, Southern University of Science and Technology (SUSTech), 1088 Xueyuan Rd, Nanshan District, Shenzhen, Guangdong, China, 518055}
\email{xuz6@sustech.edu.cn}

\begin{document}
\begin{abstract}
We prove the abundance conjecture for threefolds in characteristic $>42$.
\end{abstract}
\maketitle
\setcounter{tocdepth}{1}
\tableofcontents

\section{Introduction}

One of the central remaining problems in the study of positive characteristic birational geometry in dimension three is the abundance conjecture.

\smallskip
\noindent\textbf{Abundance conjecture.}
\emph{Let $(X,B)$ be a projective lc pair of dimension $3$ over a perfect
field $k$ of characteristic $p>0$. If $K_X+B$ is nef, then it is
semi-ample.}

\smallskip
In characteristic $0$, the abundance conjecture for threefolds was proved
in \cite{Miy88,Kaw92,KMM94} (see also \cite{Kol92}). In characteristic $p>3$, the conjecture has been
proven when the numerical dimension $\nu(K_X+B)\ne1$ (see \cite{XZ19,Xu23}). 
In mixed characteristic, the abundance conjecture for
arithmetic threefolds whose closed points have residue characteristic
greater than five was proved in \cite{BBS24}.

In this paper, we prove the abundance conjecture for threefolds in
characteristic $p>42$.

\begin{theorem}\label{thm:abundance}
Let $(X,B)$ be a projective lc pair of dimension $3$ over a perfect field
$k$ of characteristic $p>42$. If $K_X+B$ is nef, then it is semi-ample.
\end{theorem}


\subsection*{Outline of the proof}

The known cases of abundance reduce Theorem~\ref{thm:abundance}
to the numerical dimension $\nu(K_X+B)=1$ and the nef dimension $n(K_X+B)=3$.
By Proposition~\ref{prop:reduction}, after birational
modifications we may assume that $(X,B)$ is $\Q$-factorial lc,
$X\setminus B$ is terminal, and
\begin{equation}\label{eq:intro-abundance-reduction}
 B=\coprod_{i=1}^t S_i,\qquad
 L:=K_X+B\sim_{\Q}\sum_{i=1}^t d_iS_i,\qquad
 d_i>0,\qquad \nu(L)=1.
\end{equation}
Here $L$ is nef and the $S_i$ are prime divisors.
The reduction preserves the Iitaka and numerical dimensions.

The main technical result is
Corollary~\ref{cor:local-canonical-relation}.
It proves semiampleness of a prime effective $\Q$-Cartier
divisor $S$ on a normal projective threefold $X$
provided that, on a neighbourhood $U$ of $S$:
\begin{equation}\label{eq:intro-local-hypotheses}
 \begin{gathered}
 X\text{ is klt},\qquad (X,S)\text{ is lc},\qquad
 (K_X+S)|_U\sim_{\Q}aS|_U,\qquad a\in\Q_{>0},\\
 mS\text{ is Cartier on }X,\qquad
                  \OO_X(mS)|_S\simeq\OO_S
 \end{gathered}
\end{equation}
for some integer $m>0$. Theorem~\ref{thm:abundance} follows readily from this result   (see Lemma~\ref{lem:abundance-component}  and Section~\ref{sec:abundance} for details).


\smallskip
It suffices to show Corollary~\ref{cor:local-canonical-relation}.
By Totaro's criterion \cite{Tot09}, it suffices to find and integer $l>0$ such that
\begin{equation}\label{eq:intro-finite-target}
 E=lS\text{ Cartier},\qquad \OO_X(E)|_E\simeq\OO_E.
\end{equation}
large infinitesimal neighbourhood.

\smallskip
Note that $S$ is a slc Calabi-Yau surface. Denote the formal completion of $X$ along $S$ by $\widehat X_S$. As an application of the bounds on the global indices of log Calabi--Yau surfaces the authors proved in \cite{BWX26}, we can take cyclic covers of degree prime
to $p$ on $\widehat X_S$ whose reduced boundary $A$ satisfies
$\omega_A\simeq\OO_A$ and is Gorenstein slc.
On such a formal cover $\YY$, write
\[
 \II_j=\OO_{\YY}(-jA),\qquad
 \FF_j=\II_j/\II_{j+1},\qquad \MM=\omega_{\YY}(A).
\]

If the cohomology groups $H^1(A,\FF_j)$ vanish,
the required trivialization extends. For instance, it is the case when $A$ is a lc but not klt  surface. In the remaining case,
after a further tame cover there is an integer $c>0$ such that
\[
 \JJ:=\II_c\text{ is invertible},\quad
 H^*(A,\FF_j)=0\ (0<j<c),\quad
 D_\ell:=V(\JJ^\ell),\quad D:=D_1,
\]
and
\begin{equation}\label{eq:intro-cup}
 \begin{gathered}
 \JJ|_D\simeq\OO_D,\qquad
 \Pic^0_{D/k}\text{ is smooth semiabelian},\\
 V:=H^1(D,\OO_D),\qquad \dim V=2,\qquad
 \bigwedge^2V\xrightarrow{\sim}H^2(D,\OO_D)=k.
 \end{gathered}
\end{equation}
We can further assume that the Galois group $G$ of the formal cover acts faithfully on $V$, and each of its characters has order  at most $12$. This case is achieved by careful computations of obstruction classes (see Section \ref{sec:connecting} and \ref{sec:finite}).

\smallskip

\subsection*{Organization}

Section~\ref{sec:divisorial} fixes the notation and reviews
preliminary results. 
Section~\ref{sec:surface-cover} proves a semi-ampleness criterion needed for the slc Calabi--Yau boundary. Sections~\ref{sec:connecting}--\ref{sec:canonical-relation}
prove the local semiampleness criterion, and
Section~\ref{sec:abundance} proves Theorem~\ref{thm:abundance}.
Appendix~\ref{app:surface} treats normal surfaces and their
divisorial sheaves. Appendix~\ref{app:models} treats nonnormal
surfaces and their conductors. 
Appendix~\ref{app:witt} constructs
some cohomology classes which is used for Proposition \ref{prop:comparison-class}. The bounds on the global indices of log Calabi--Yau surfaces used here are proved
in the companion paper \cite{BWX26} that we will publish independently in the nearby future; their application
to the ambient threefold is explained in
Section~\ref{sec:ambient-index}.

\subsection*{Acknowledgements}
We thank Jefferson Baudin, Bhargav Bhatt, J\'anos Koll\'ar, Wenhao Ou, Zsolt Patakfalvi, Burt Totaro, Lei Zhang, and Chenyang Xu for helpful discussions related to the content of the paper.

Fabio Bernasconi was supported by the grant PZ00P2-21610 from SNSF, Jakub Witaszek is supported by NSF CAREER grant 2540921 and NSF grant 2401360. Zheng Xu is supported by start-up and matching funds from the Southern University of Science and Technology (SUSTech), and he would like also thank Zhiyu Tian and Wenhao Ou for support over the past years.

\subsection*{AI disclosure}

We started the project a couple of years before the advent of AI.
The main ideas of using the boundedness of indices of log Calabi--Yau surface pairs, the evaporation technique, $F$-splittings, and the classification of slc surfaces emerged from our collaboration, without any input from AI models.
In particular, this settles the case where the slc Calabi--Yau boundary $A$ has zero irregularity, leaving open the irregular case.

In July 2026, we first used OpenAI's ChatGPT 5.6 to solve a very special case of Corollary \ref{cor:local-canonical-relation}, in which $S$ is a bi-elliptic surface and $\OO_X(S)|_S$ is a $2$-torsion line bundle.
We then obtained Proposition \ref{prop:comparison-class} by working interactively with ChatGPT 5.6.
ChatGPT was also used to help write this draft of the article.

We decided to upload this (very) preliminary version of the article following the release of OpenAI's list of solved problems in mathematics, which claims to solve the abundance conjecture for threefolds in characteristic $p>3$. 
We will continue working on the draft and update it, together with its companion paper, in the near future.



\section{Preliminaries}\label{sec:divisorial}

\subsection{Notation and conventions}

Unless another field is specified, $k$ is algebraically closed of
characteristic $p>42$. 

\noindent $\bullet$ All schemes are excellent and noetherian.
A variety is an integral separated scheme of finite type over a
field. A surface is allowed to be reduced and reducible.
Boundaries in threefold pairs are $\Q$-divisors.

\noindent $\bullet$ For the basic knowledge and notions in birational geometry, such as pairs and singularities, we refer to \cite{KM98}. Boundaries in threefold pairs are $\Q$-divisors.
The nef dimension $n$ is used as in \cite[\S2.1]{Xu23}.

\noindent $\bullet$ A reduced surface $A$ is
demi-normal if it is $S_2$ and has only ordinary double points in
codimension one. It is semi-log canonical (slc) if, in addition,
$K_A$ is $\Q$-Cartier and $(A^\nu,C^\nu)$ is lc, where
$\nu:A^\nu\to A$ is its normalization and $C^\nu$ is the
conductor divisor.
Explicitly, the conductor ideal is
$\operatorname{Ann}_{\OO_A}(\nu_*\OO_{A^\nu}/\OO_A)$,
and it defines a curve $C\subset A$ and its inverse image
$C^\nu\subset A^\nu$.
The normalization of $C^\nu$ has an involution which
interchanges the two branches over the generic double
locus. In the elliptic-conductor case $C^\nu$ is already
smooth, so the involution acts on $C^\nu$ itself.
A superscript $-$ denotes its anti-invariant summand.
We call such a surface Gorenstein when it is
Cohen--Macaulay and its dualizing sheaf is invertible.

\noindent $\bullet$ Standard coefficients are the numbers in
$\{1\}\cup\{1-1/r:r\in\Z_{>0}\}$.

\noindent $\bullet$ For an integral Weil divisor $L$ on a normal scheme, $\OO(L)$
denotes its rank-one reflexive divisorial sheaf. Reflexive powers
and pullbacks of a coherent sheaf $\FF$ are denoted by $\mathcal F^{[r]}$ and
$f^{[*]}\mathcal F=(f^*\mathcal F)^{**}$.
Tensor products with invertible sheaves are sometimes written
without the symbol $\otimes$. On a proper surface, we write
\[
 h^*(\mathcal F)=(h^0(\mathcal F),h^1(\mathcal F),h^2(\mathcal F)),
 \qquad
 H^*(\mathcal F)=0
 \ \Longleftrightarrow\
 H^i(\mathcal F)=0\text{ for every }i.
\]
A sheaf with this vanishing is called acyclic.

\noindent $\bullet$ The symbol $\Pic^0_{Z/k}$ denotes the identity component of the
Picard group scheme, not its reduction.
A semiabelian variety is a smooth connected commutative group
which is an extension of an abelian variety by a torus.
A vector group associated to a finite-dimensional $k$-space $W$
is the additive group of its affine space.
In an exact sequence of group schemes we also write $W$ for
this vector group.

\noindent $\bullet$ All finite groups acting on the covers in this paper have order prime to $p$.
An equivariant sheaf is equipped with a compatible action of
the group on the sheaf, also called a linearization.
For a character $\chi:G\to k^*$ and a representation $W$, set
\[
 W_\chi=\{w\in W:g(w)=\chi(g)w\text{ for all }g\in G\}.
\]
Taking invariants, and taking character spaces for an abelian
group, are exact. The average $|G|^{-1}\sum_{g\in G}g$ is the
projection to invariants.

\noindent $\bullet$ If $T\subset X$ is a proper closed subscheme, $\widehat X_T$
denotes the formal completion of $X$ along $T$.

\noindent $\bullet$ For a square-zero closed immersion $Z\hookrightarrow Z'$ with
ideal $\mathcal K$, the sequence
\[
 1\longrightarrow1+\mathcal K\longrightarrow\OO_{Z'}^*
       \longrightarrow\OO_Z^*\longrightarrow1,\qquad
 1+\mathcal K\simeq(\mathcal K,+),
\]
describes both extension of invertible sheaves and extension of
their trivializations \cite[Tag 0C6Q]{Sta}.
The obstruction to extending an invertible sheaf lies in
$H^2(\mathcal K)$. The obstruction to extending a chosen
trivialization of a fixed invertible sheaf lies in
$H^1(\mathcal K)$.
In particular, the kernel of restriction on Picard groups is
a quotient of $H^1(\mathcal K)$.
Induction shows that the Picard kernel of any finite nilpotent
thickening in characteristic $p$ is $p$-primary.

\noindent $\bullet$ We use Serre duality with its canonical trace
\cite[Tag 0FVU]{Sta}. For a proper Cohen--Macaulay surface
$Z$ and any coherent sheaf $\mathcal H$,
\[
 H^2(Z,\mathcal H)^*\simeq\Hom_Z(\mathcal H,\omega_Z).
\]
This top-degree assertion does not require $\mathcal H$ to be
Cohen--Macaulay.

\subsection{MMP reduction step}\label{sec:mmp-author}
First, we recall the theory of MMP for projective lc threefold pairs
over a perfect field of characteristic $p>3$.

\begin{theorem}[\cite{HNT19}, Theorem 1.1, and \cite{HW22a}]
\label{thm:mmp}
Let $(X,B)$ be a lc threefold pair over a perfect field $k$ of
characteristic $>3$ and $f:X\to Y$ a projective surjective morphism
to a quasi-projective variety. If $K_X+B$ is pseudo-effective
(resp.\ not pseudo-effective) over $Y$, then we can run a
$(K_X+B)$-MMP to get a log minimal model (resp.\ Mori fibre space)
over $Y$.
\end{theorem}

Moreover, we define a partial MMP over algebraically closed fields
of characteristic $p>3$.

\begin{definition}[\cite{Xu23}, Definition 2.23]\label{def:partial-mmp}
Let $(X,B)$ be a $\Q$-factorial projective lc threefold pair over
an algebraically closed field $k$ of characteristic $>3$ and $A$
be an effective $\Q$-divisor such that $(X,B+A)$ is lc and
$K_X+B+A$ is nef. We can run a partial $(K_X+B)$-MMP with scaling
of $A$ as follows.

Let $\lambda_0=\inf\{t:K_X+B+tA\text{ is nef}\}$. If $\lambda_0<1$,
then we stop. Otherwise, there exists a $(K_X+B)$-negative
extremal ray $R_0$ which satisfies $(K_X+B+A)\cdot R_0=0$.
We contract this extremal ray. If the contraction is a Mori fibre
contraction, we stop. Otherwise let $\mu_0:X\dashrightarrow X_1$
be the divisorial contraction or flip. Repeat this process for
$(X_1,\mu_{0*}B)$, $\mu_{0*}A$ and so on.

We call this construction a $(K_X+B)$-MMP which is
$(K_X+B+A)$-trivial.
\end{definition}

\begin{proposition}\label{prop:reduction}
Let $(X,B)$ be a lc projective threefold pair over an algebraically
closed field $k$ of characteristic $>3$. Assume that $K_X+B$ is nef.
Then we have either the nef dimension $n(K_X+B)\leq2$, or there
exist a $\Q$-factorial lc projective threefold pair $(Y,B_Y)$
such that the following conditions hold:
\begin{enumerate}[label=\textup{(\arabic*)}]
\item $B_Y$ is reduced, and $(Y\setminus B_Y)$ is terminal,
\item $K_Y+B_Y\sim D_Y$ for an effective $\Q$-divisor $D_Y$
with $\Supp D_Y=B_Y$,
\item $K_Y+B_Y$ is nef,
\item $\kappa(K_X+B)=\kappa(K_Y+B_Y)$ and
$\nu(K_X+B)=\nu(K_Y+B_Y)$,
\item any connected component $S$ of $B_Y$ is integral.
\end{enumerate}
\end{proposition}

\begin{proof}
By \cite[Theorem 2.27]{Xu23}, replacing $(X,B)$ by a dlt model,
we can assume that $X$ is terminal, and $(X,B)$ is $\Q$-factorial
and dlt. By Definition~\ref{def:partial-mmp} we run a
$K_X$-MMP which is $(K_X+B)$-trivial. Since $X$ is terminal,
it terminates by \cite[Lemma 2.26]{Xu23}. If it terminates with
a Mori fibre space, then we have $n(K_X+B)\leq2$. The assertion
holds. Hence we can assume that the $K_X$-MMP which is
$(K_X+B)$-trivial terminates with a lc pair $(X',B')$ such that
$K_{X'}+(1-\varepsilon)B'$ is nef for any sufficiently small
$\varepsilon>0$. By \cite[Theorem 4.4]{Xu24}, we have
\[
 \kappa(K_X+(1-\varepsilon)B)
 =\kappa(K_{X'}+(1-\varepsilon)B')\geq0
\]
for any sufficiently small rational $\varepsilon>0$.
Hence there exists an effective $\Q$-divisor
$D\sim_{\Q}K_X+B$ such that $\Supp B\subseteq\Supp D$.

Let $g:W\to X$ be a log smooth resolution of $(X,\Supp D)$ with
the reduced exceptional divisor $E$. Set
$B_W:=\Supp(g^{-1}_*D)+E$. Since $(X,B)$ is dlt, we can write
$K_W+B_W\sim_{\Q}D_W$, where
$D_W=g^*D+(K_W+B_W-g^*(K_X+B))$ is an effective $\Q$-divisor with
\[
 \Supp D_W=\Supp(g^{-1}_*D)+E=B_W.
\]
By replacing $W$ by the output of a $(K_W+B_W)$-MMP over $X$ by
Theorem~\ref{thm:mmp} and using again that $(X,B)$ is dlt,
we can assume that
$W\setminus\Supp D_W\cong X\setminus\Supp D$ and
$\Supp D_W=\Supp g^*D$.

Now, since $K_W+B_W$ is pseudo-effective, by
Theorem~\ref{thm:mmp} we can run a $(K_W+B_W)$-MMP
which terminates with a minimal model $h:W\dashrightarrow Y$.
In particular $K_Y+B_Y$ is nef for $B_Y:=h_*B_W$.
Furthermore, $K_Y+B_Y\sim_{\Q}D_Y$ and $\Supp D_Y=B_Y$, where
$D_Y:=h_*D_W$. Hence, (2) and (3) holds. Since $h$ and $g$ are
isomorphisms outside of $D_W$, we get (1). Note that both
$K_X+B$ and $K_Y+B_Y$ are effective and nef $\Q$-divisors.
Applying \cite[Lemma 11.3.3]{Kol92}) to pullbacks of $K_X+B$
and $K_Y+B_Y$ on any common resolution of $X$ and $Y$, we get
\[
 \kappa(K_X+B)=\kappa(K_Y+B_Y),\qquad
 \nu(K_X+B)=\nu(K_Y+B_Y),
\]
that is (4). To get (5), note that we can run
$(K_Y+B_Y-S)$-MMP for any irreducible component of $B_Y$ and
it terminates by \cite[Theorem 4.5]{Xu24} since
$K_Y+B_Y-\varepsilon S$ is effective for any sufficiently
small $\varepsilon\geq0$. Hence we can assume that $(Y,B_Y)$
satisfies (5) by the proof of \cite[Lemma 13.3.2]{Kol92}.
\end{proof}

\Needspace{6\baselineskip}
\subsection{The formal threefold and sheaves on the divisor}

\begin{notation}\label{not:formal}
Let $\YY$ be a normal excellent Cohen--Macaulay formal threefold,
complete along a connected reduced proper effective $\Q$-Cartier
divisor $A$. Assume that
\begin{equation}\label{eq:formal-assumptions}
 \begin{gathered}
 \II_j:=\OO_{\YY}(-jA)\text{ is maximal Cohen--Macaulay}
 \quad(j\in\Z),\\
 \MM:=\omega_{\YY}(A)\text{ is invertible},\quad
 \MM|_A\simeq\OO_A.
 \end{gathered}
\end{equation}
Set $\FF_j=\II_j/\II_{j+1}$. All products of these quotients are
induced by $\II_i\II_j\subseteq\II_{i+j}$. Suppose also that
$NA$ is Cartier and $\II_N|_A\simeq\OO_A$ for some $N>0$.
\end{notation}

The following lemma helps us study the dual of $\FF_j$ and the Cartierness of a multiple of $A$.

\begin{lemma}\label{lem:quotient-duality}
Under Setting~\ref{not:formal}, $A$ is Gorenstein,
$\omega_A\simeq\MM|_A$, and multiplication induces
\begin{equation}\label{eq:quotient-duality}
       \FF_{-j}\xrightarrow{\sim}\sheafHom_A(\FF_j,\OO_A).
\end{equation}
If $\FF_j$ is invertible at $x\in A$, then $jA$ is Cartier
near $x$.
\end{lemma}

\begin{proof}
The depth lemma shows that every $\FF_j$ is Cohen--Macaulay.
Dualizing $0\to\II_{j+1}\to\II_j\to\FF_j\to0$ gives
\[
 \sheafExt^1_{\YY}(\FF_j,\omega_{\YY})
   \simeq\MM\FF_{-j},
 \qquad
 \sheafHom_{\YY}(\II_j,\omega_{\YY})=\MM\II_{1-j}.
\]
The higher Ext groups of the $\II_j$ vanish by maximal
Cohen--Macaulayness and canonical-module duality. Taking $j=0$
gives $\omega_A=\MM|_A$. Finite duality then identifies the
displayed Ext module with $\sheafHom_A(\FF_j,\omega_A)$.
Naturality with respect to multiplication identifies the resulting
pairing with the evaluation pairing in \eqref{eq:quotient-duality}.

If $\FF_j$ is invertible at $x$, choose inverse generators in
$\FF_j$ and $\FF_{-j}$, and lift them to $u\in\II_j$ and
$v\in\II_{-j}$. Then $uv\equiv1\pmod{\II_1}$, so $uv$ is a
unit. For $w\in\II_j$,
\[
                         w=u\frac{vw}{uv}.
\]
Thus $u$ generates $\II_j$ near $x$.
\end{proof}

We next identify exactly the degrees in which sections can occur.
This will later control sections of line bundles on thickenings.

\begin{lemma}\label{lem:section-period}
Under Setting~\ref{not:formal}, let $c$ be the least positive integer for which $cA$ is Cartier and
$\II_c|_A\simeq\OO_A$. Then, for every $j\in\Z$,
\begin{equation}\label{eq:section-period}
 H^0(A,\FF_j)\ne0\ \Longleftrightarrow\ c\mid j
 \ \Longleftrightarrow\ H^2(A,\FF_j)\ne0.
\end{equation}
In either nonzero case the corresponding group is one-dimensional.
\end{lemma}

\begin{proof}
Let $s\in H^0(A,\FF_j)$ be nonzero, and let
$\tau$ be the chosen generating section of $\II_N|_A$.
Since $NA$ is Cartier, multiplication identifies
$\FF_{Nj}$ with $(\II_N|_A)^{\otimes j}$. Thus
\[
 s^N=\lambda\tau^j\in H^0(A,\FF_{Nj}),\qquad
                         \lambda\in H^0(A,\OO_A)=k.
\]
Purity of $\FF_j$ implies that $s$ is nonzero at a generic
point of $A$. Hence $\lambda\ne0$.
Multiplication by the invertible sheaf $(\II_N|_A)^{-j}$ gives
\[
 t:=\lambda^{-1}\tau^{-j}s^{N-1}\in H^0(A,\FF_{-j}),
 \qquad st=1.
\]
For every local section $w$ of $\FF_j$, associativity gives
$w=s(tw)$. Consequently $s$ generates $\FF_j$, and
Lemma~\ref{lem:quotient-duality} makes $jA$ Cartier.
The set of integers $j$ with $jA$ Cartier and
$\II_j|_A\simeq\OO_A$ is a subgroup of $\Z$, hence equals
$c\Z$. The assertion for $H^2$ follows from
\eqref{eq:quotient-duality} and Serre duality.
\end{proof}

\Needspace{6\baselineskip}
\subsection{Passing to one Cartier multiple}

The following setting is  natural in the context of 
the surface calculation in
Section~\ref{sec:surface-cover}. 

\begin{notation}\label{not:cartier-thickenings}
In addition to \eqref{eq:formal-assumptions}, fix $c\in\Z_{>0}$
and assume
\begin{equation}\label{eq:acyclic-layers}
 \JJ:=\II_c\text{ is invertible},\qquad \JJ|_A\simeq\OO_A,
 \qquad H^*(A,\FF_j)=0\quad(0<j<c).
\end{equation}
Set
\[
 D_l=V(\JJ^l),\qquad D=D_1=cA,
 \qquad\OO_{D_l}=\OO_{\YY}/\JJ^l\quad(l\geq1).
\]
We assume that these schemes are proper over $k$.
\end{notation}

\begin{lemma}\label{lem:cohomology-and-picard}
Restriction induces an isomorphism of graded $k$-algebras and an
isomorphism of fppf group functors
\begin{equation}\label{eq:cohomology-picard}
 H^*(D,\OO_D)\xrightarrow{\sim}H^*(A,\OO_A),\qquad
             \Pic_{D/k}\xrightarrow{\sim}\Pic_{A/k}.
\end{equation}
In particular $\JJ|_D$ and $\MM|_D$ are trivial.
\end{lemma}

\begin{proof}
The ideal $\II_1/\II_c$ has successive quotients
$\FF_1,\ldots,\FF_{c-1}$. Its cohomology vanishes by
\eqref{eq:acyclic-layers}. This proves the first assertion,
including compatibility with products. The square-zero units
sequences for the same filtration give the second assertion
\cite[Tag 0C6Q]{Sta}. The vanishings persist after every affine
base change over $k$, so this is an assertion about fppf Picard
functors. Applying it to the two specified trivial restrictions on $A$, we get the latter assertion.
\end{proof}

Although $D=cA$ is Cartier on $\YY$, its dualizing module need
not be invertible, since $\YY$ is only Cohen--Macaulay. The
following map will be useful for dealing with the canonical divisor of $\YY$.

\begin{lemma}\label{lem:dualizing-comparison}
For $l\geq1$ there is a natural homomorphism
\begin{equation}\label{eq:phi}
 \begin{split}
 \phi_l:\MM\JJ^{1-l}|_{D_l}&\longrightarrow\omega_{D_l},\\
 \MM(\II_{c-cl}/\II_c)&\longrightarrow
                         \MM(\II_{1-cl}/\II_1).
 \end{split}
\end{equation}
Its kernel and cokernel are respectively
\begin{equation}\label{eq:phi-errors}
 \MM(\II_1/\II_c),\qquad
                    \MM\JJ^{-l}(\II_1/\II_c).
\end{equation}
They are acyclic, also after tensoring by any power of $\JJ$.
Consequently $\phi_l$ induces cohomology isomorphisms compatible
with multiplication and canonical trace.
\end{lemma}

\begin{proof}
Cartier adjunction gives
$\omega_{D_l}=\MM(\II_{1-cl}/\II_1)$. The inclusion
$\II_{c-cl}\subseteq\II_{1-cl}$ defines \eqref{eq:phi}.
Intersection and quotient of these fractional ideals give
\eqref{eq:phi-errors}. Their filtrations have quotients
$\MM|_A\otimes\JJ^r|_A\otimes\FF_j$ with $0<j<c$;
these are acyclic. The compatibilities follow from the defining
inclusions and functoriality of duality
\cite[Tag 0FVU]{Sta}. For $l=1$, the image is the canonical
trace subsheaf $\omega_A\subseteq\omega_D$. Thus the induced
map on top cohomology is the trace identification, with no scalar
factor $c$.
\end{proof}

\Needspace{6\baselineskip}
\subsection{The semiampleness criterion}

We recall Totaro's semiampleness criterion and then explain how
to apply it to an invariant trivialization on a formal cover.
The exponent of the line bundle and the multiplicity of the
Cartier divisor in this application must agree.

\begin{lemma}[Totaro]\label{lem:principal-parts}
Let $X$ be a normal projective variety over an algebraically closed
field of characteristic $p>0$. If $E$ is an effective Cartier
divisor and $\OO_X(E)|_E\simeq\OO_E$, then $E$ is semiample.
\end{lemma}

\begin{proof}
Apply \cite[Lemma 4.1]{Tot09} to the effective Cartier divisor
$D=E$ and the line bundle $L=\OO_X(E)$.
The order of $L|_D$ in $\Pic(D)$ is $m=1$, and the condition
in that lemma holds with $r=0$. Thus $L$ is semiample.
The cited lemma allows $D$ to be nonreduced.
\end{proof}

\Needspace{12\baselineskip}
\begin{lemma}\label{lem:invariant-descent}
Suppose $h:\YY\to\widehat X_S=\YY/G$ is a finite quasi-\'{e}tale
quotient by a group of order prime to $p$, where $X$ is normal
projective, $S$ is reduced and $\Q$-Cartier, and
$A=h^*S$ is reduced.
Use Setting~\ref{not:cartier-thickenings}. If the character of
$\JJ|_A$ has order $R$ and $\JJ^{Rl}|_{D_{Rl}}$ has a
$G$-invariant trivialization, then $cRlS$ is Cartier and
\begin{equation}\label{eq:invariant-descent}
             \OO_X(cRlS)|_{cRlS}\simeq\OO_{cRlS}.
\end{equation}
Hence $S$ is semiample.
\end{lemma}

\begin{proof}
The reduced character of $\JJ^R$ is trivial. On an affine open
of the quotient, lift its reduced generator to the finite inverse
image and apply $|G|^{-1}\sum_{g\in G}g$. The resulting invariant
section still generates by Nakayama's lemma. Since
\[
 (h_*\OO_{\YY}(-cRA))^G=\OO_{\widehat X_S}(-cRS),
\]
$cRS$ is Cartier on the completion. Faithfully flat descent from
completion gives Cartierness near $S$, and it holds outside $S$
as well. Exactness of invariants identifies the quotient by the
$l$-th power with $\OO_{cRlS}$ and descends the given trivialization.
Apply Lemma~\ref{lem:principal-parts} and we are done.
\end{proof}

\subsection{Sections on thickenings and roots of rational functions}

The following observation will be used in Sections~\ref{sec:finite}
and~\ref{sec:canonical-relation}. 
Let a finite group $G$ act on the data of
Setting~\ref{not:cartier-thickenings}.
For an equivariant line bundle $\mathcal H$ with
$\mathcal H|_D\simeq\OO_D$, we define a filtration of $H^0(D_N,\mathcal H)^G$ by
\[
 \operatorname{Fil}^j
  =H^0(D_N,\mathcal H\JJ^j/\mathcal H\JJ^N)^G
  \subseteq H^0(D_N,\mathcal H)^G
  \quad(0\leq j\leq N).
\]
A section has order $j$ if it belongs to
$\operatorname{Fil}^j\setminus\operatorname{Fil}^{j+1}$.

\begin{lemma}\label{lem:section-filtration}
For $0\leq j<N$, the $j$-th graded piece of this filtration
is nonzero if and only if
$\mathcal H\JJ^j|_{D_{N-j}}$ has a $G$-invariant generating
section. Every nonzero graded piece is one-dimensional.
\end{lemma}

\begin{proof}
The invertibility of $\JJ$ identifies the sheaf
$\mathcal H\JJ^j/\mathcal H\JJ^N$ with
$\mathcal H\JJ^j|_{D_{N-j}}$.
A section of order $j$ restricts to a nonzero element of
$H^0(D,\mathcal H\JJ^j|_D)=k$, and hence generates the
line bundle on $D_{N-j}$ by Nakayama's lemma.
Conversely a generating section gives a section of order $j$.
The restriction map embeds the graded piece into the
one-dimensional space.
\end{proof}

\begin{lemma}\label{lem:field-powers}
Let $X$ be a variety over a perfect field $k$ of
characteristic $p>0$, and let $x\in X(k)$ be a smooth point.
Set $K=k(X)$ and $\widehat K=\operatorname{Frac}\widehat\OO_{X,x}$.
For every integer $h\geq0$,
\[
                         K\cap\widehat K^{p^h}=K^{p^h}.
\]
\end{lemma}

\begin{proof}
For $h=1$, choose \'{e}tale parameters at $x$.
Their differentials form a basis of $\Omega_{K/k}$, and the
dual derivations extend continuously to $\widehat\OO_{X,x}$
and then to $\widehat K$. If $f\in K$ is a $p$-th power
in $\widehat K$, all these derivations annihilate $f$.
Thus $df=0$, so $f\in K^p$ by \cite[Tag 031W]{Sta}.
The reverse inclusion is immediate.
The root in $K$ agrees with the root in $\widehat K$ by
injectivity of Frobenius. Induction gives the assertion for $h>1$.
\end{proof}

\subsection{Normalized tame cyclic covers}

We record the cyclic-cover results in the form used below.
The normalization is taken in a finite product of fields.
One may then choose a connected component.

\begin{lemma}[Ambro]\label{lem:normalized-cyclic-algebra}
Let $X$ be a normal variety over $k$, let $q>0$ be
prime to $p$, and let $f\in k(X)^*$. Set
$H=q^{-1}\divisor_X(f)$.
The normalization $Y$ of $X$ in $k(X)[z]/(z^q-f)$ is finite,
and its algebra is
\begin{equation}\label{eq:normalized-cyclic-algebra}
 \pi_*\OO_Y
   =\bigoplus_{i=0}^{q-1}\OO_X(\lfloor iH\rfloor)z^i .
\end{equation}
Multiplication is induced by $z^q=f$ in the generic algebra.
For a prime divisor $P\subset X$, the ramification index at
every prime divisor above $P$ is
\begin{equation}\label{eq:cyclic-ramification-index}
                    e_P=\frac{q}{\gcd(q,v_P(f))}.
\end{equation}
We use the convention $\gcd(q,0)=q$.
\end{lemma}

\begin{proof}
The algebra and its identification with the normalization are
\cite[Proposition 3.1 and Lemma 3.2]{Amb15}.
The ramification formula is \cite[Lemma 3.5]{Amb15}.
To specify the codimension-one assertion, let $t$ be a
uniformizer at $P$ and write $f=t^nu$ with $u$ a unit.
\cite[Lemma 3.4]{Amb15} describes the stalk of
\eqref{eq:normalized-cyclic-algebra}.
In particular, if $n=qd$, substitution $w=t^{-d}z$ gives
the finite \'{e}tale algebra
$\OO_{X,\eta_P}[w]/(w^q-u)$.
The construction applies to a rational function with poles
as well as to a regular function.
\end{proof}

\subsection{Application of the bounds on the global indices of log Calabi--Yau surfaces}\label{sec:ambient-index}

The bounds on the global indices of log Calabi--Yau surfaces are proved separately in
\cite{BWX26}.

\begin{lemma}\label{lem:ambient-section-lifting}
Let $Y$ be a normal Cohen--Macaulay threefold,
$T$ a reduced effective $\Q$-Cartier divisor, and
$D$ an integral $\Q$-Cartier divisor.
Assume that $\OO_Y(-T)$, $\OO_Y(D)$ and $\OO_Y(D-T)$
are maximal Cohen--Macaulay. Set
\[
                  \mathcal E_D=\OO_Y(D)/\OO_Y(D-T).
\]
Assume that $D$ is Cartier at every generic curve point
of $T$ and that $\mathcal E_D\simeq\OO_T$. Then $D$
is Cartier near $T$ and $\OO_Y(D)|_T\simeq\OO_T$.
\end{lemma}

\begin{proof}
Choose $N>0$ with $ND$ Cartier. On the complement of
finitely many points of $T$, actual multiplication gives
an isomorphism
\[
                    \mathcal E_D^{\otimes N}
                       \longrightarrow\mathcal E_{ND}.
\]
Both sides are lines on $T$. Its inverse extends across
the omitted points by $S_2$, so it is an isomorphism
everywhere.

At a point of $T$ lift the generating section of $\mathcal E_D$
to $u\in\OO_Y(D)$. The element $u^N$ generates
$\OO_Y(ND)$ modulo the ideal of $T$, and hence
locally by Nakayama's lemma. In the fractional
divisorial modules set
\[
                  v=u^{N-1}(u^N)^{-1}\in\OO_Y(-D).
\]
Then $uv=1$, and every $w\in\OO_Y(D)$ is
$u(vw)$ with $vw\in\OO_Y$. Thus $u$ freely generates
$\OO_Y(D)$ and restricts to the chosen generating section.
\end{proof}

\begin{lemma}\label{lem:ambient-log-index}
Let $V$ be a normal threefold and $T$
a connected proper reduced effective $\Q$-Cartier divisor.
Assume that $V$ is klt and $(V,T)$ is lc near $T$, and that
\[
 q(K_V+T)\text{ is Cartier},\qquad
 \OO_V(q(K_V+T))|_T\simeq\OO_T
\]
for some $q>0$, with Cartierness required only near $T$.
Then there is such an integer $q_0$ prime to $p$.
\end{lemma}

\begin{proof}
Set $L=K_V+T$. The divisorial Cohen--Macaulay theorem
\cite[Theorems 3 and 19(2)]{BK23} makes each
\[
            \mathcal E_D=\OO_V(D)/\OO_V(D-T)
\]
an $S_2$ sheaf when $D$ and $D-T$ are $\Q$-Cartier.
All divisors used below have this property. It also
makes $T$ an $S_2$ surface. The reduced lc boundary
is nodal in codimension one and weakly normal by \cite[Lemma 2.3.4]{AP26}.

Adjunction gives standard-coefficient lc pairs
$(T_i^\nu,\Delta_i)$ with
\[
               (K_V+T)|_{T_i^\nu}=K_{T_i^\nu}+\Delta_i
                                      \sim_{\Q}0.
\]
By \cite{BWX26}, the global orders
on these normal surfaces are prime to $p$.
At a plt curve point the different denominator
annihilates the local Weil class group
\cite[Lemma 3.3]{HW19}.
At a non-plt curve point
\cite{BWX26} applies. Hence
there is a common prime-to-$p$ multiple $q_1$ such that
$q_1L$ is Cartier at every generic curve point of $T$
and $\mathcal F=\mathcal E_{q_1L}$ has trivial
normalized reflexive pullback. The identification is
given by the actual pluricanonical residue on this
locus and extended on each normalization by reflexivity.

Some reflexive power of $\mathcal F$ is trivial by
the given restriction isomorphism. Its reflexive powers are precisely
$\mathcal E_{dq_1L}$ since this holds where $q_1L$ is
Cartier and extends by $S_2$. Write an annihilating
integer as $p^hs$, $(s,p)=1$.
\cite{BWX26}, applied to
$\mathcal F^{[s]}$, gives
$\mathcal E_{sq_1L}\simeq\OO_T$.
The divisor $sq_1L$ is still Cartier at all generic
curve points. Lemma~\ref{lem:ambient-section-lifting}
therefore applies with $D=sq_1L$.
Taking $q_0=sq_1$ proves the assertion.
\end{proof}

The lemma concerns the logarithmic index of $K_V+T$.
It does not assert prime-to-$p$ Cartierness of $T$.

\subsection{Length-two Witt vectors and their connecting homomorphism}
\label{sec:witt-preliminaries}

We recall the properties of length-two Witt vectors used in
Section~\ref{sec:witt-lifting} and Appendix~\ref{app:witt}.
All exact sequences below are taken in the category of
abelian sheaves.

For a ring $B$ of characteristic $p$, the underlying set of
$W_2(B)$ is $B\times B$. Set
\begin{equation}\label{eq:witt-addition-polynomial}
 B_p(x,y)=\frac{x^p+y^p-(x+y)^p}{p}\in\Z[x,y].
\end{equation}
The quotient is taken over $\Z$ before reduction to $B$.
The addition law is
\[
 (x_0,x_1)+(y_0,y_1)
   =(x_0+y_0,x_1+y_1+B_p(x_0,y_0)).
\]
We use the following maps:
\[
 \begin{aligned}
 [x]&=(x,0),& V_W(x)&=(0,x),\\
 F_W(x_0,x_1)&=(x_0^p,x_1^p),&
 \operatorname{pr}_1(x_0,x_1)&=x_0.
 \end{aligned}
\]
Here $[x]$ is the Teichm\"uller representative, $F_W$ is
Frobenius, and $V_W$ is Verschiebung. The last three maps
are additive, whereas $x\mapsto[x]$ is generally not.
In particular,
\[
 [x]+[y]=[x+y]+V_W(B_p(x,y)),\qquad F_WV_W(x)=V_W(x^p).
\]
The second coordinate $\operatorname{pr}_2(x_0,x_1)=x_1$
is not an additive map. Nevertheless, since $p>2$,
\[
                  (x,a)-(x,b)=V_W(a-b).
\]
Thus two Witt vectors with the same first coordinate have
difference $V_W(b)$ for a uniquely determined $b\in B$.
This elementary observation is used in Appendix~\ref{app:witt}.

For a proper $k$-scheme $Z$, the Witt sequence is
\begin{equation}\label{eq:witt-short-exact}
 0\longrightarrow\OO_Z\xrightarrow{V_W}W_2\OO_Z
 \xrightarrow{\operatorname{pr}_1}\OO_Z\longrightarrow0.
\end{equation}
Write
\[
 \beta_Z:H^1(Z,\OO_Z)\longrightarrow H^2(Z,\OO_Z)
\]
for its connecting homomorphism, also called the first Witt
Bockstein. Exactness gives
\[
 \ker(\beta_Z)=
 \im\bigl(H^1(Z,W_2\OO_Z)\xrightarrow{\operatorname{pr}_1}
                                      H^1(Z,\OO_Z)\bigr).
\]
The map $\beta_Z$ is additive and functorial in $Z$.
Multiplication by $[c]$, for $c\in k$, acts on the left
term of \eqref{eq:witt-short-exact} by $c^p$ and on the
right term by $c$. Consequently
\[
                         \beta_Z(c\alpha)=c^p\beta_Z(\alpha).
\]
If $\Pic_{Z/k}$ is smooth at the identity, Mumford's
criterion gives
\begin{equation}\label{eq:smooth-picard-bockstein}
                              \beta_Z=0,
\end{equation}
see \cite[Remark 5.15]{Kle05}. No reducedness or smoothness
assumption on $Z$ is required here.

Fix an indexed affine covering $(U_a)$ of $Z$, with
\[
 (df)_{ab}=f_b-f_a,\qquad
 (dc)_{abc}=c_{bc}-c_{ac}+c_{ab}.
\]
The two copies of $\OO_Z$ in \eqref{eq:witt-short-exact}
have no higher cohomology on affine intersections.
The same is therefore true of $W_2\OO_Z$, so this covering
computes all three cohomology groups.
For an additive $1$-cocycle $z=(z_{ab})$, one has
\begin{equation}\label{eq:bockstein-cocycle}
 \begin{gathered}
 \beta_Z([z])=[B_p(z_{ab},z_{bc})],\\
 (z,b)\text{ is a Witt }1\text{-cocycle}
       \quad\Longleftrightarrow\quad
 db=-B_p(z_{ab},z_{bc}).
 \end{gathered}
\end{equation}
In the first line, brackets denote the cohomology class of
the displayed $2$-cocycle.
If $[z]\in\ker\beta_Z$, the particular representative $z$
can be prescribed as the first coordinate: choose a Witt
cocycle $\zeta'$ with $\operatorname{pr}_1\zeta'-z=da$ and
replace it by $\zeta'-d[a]$.

\section{Surface cohomology on tame formal covers}\label{sec:surface-cover}

In this section $k$ is algebraically closed of characteristic $p>42$.
In this section, we prove the following result, which reduces the general case to two special cases.

\begin{proposition}\label{prop:surface-cover}
Let $V$ be a normal projective threefold and $T$ a nonempty
connected reduced effective $\Q$-Cartier divisor.
Assume that $V$ is klt and $(V,T)$ is lc on a neighbourhood of
$T$. Suppose, on this neighbourhood, that
\begin{equation}\label{eq:geometric-hypotheses}
 NT\text{ is Cartier},\quad\OO_V(NT)|_T\simeq\OO_T,\quad
 q(K_V+T)\text{ is Cartier},\quad
 \OO_V(q(K_V+T))|_T\simeq\OO_T
\end{equation}
for some $N,q>0$. There is a finite abelian quasi-\'{e}tale cover
of degree prime to $p$
\[
 h:\widetilde\YY\longrightarrow\widehat V_T,\qquad
 \widetilde\YY/\widetilde G=\widehat V_T,\qquad
                         \widetilde A=h^*T,
\]
with $\widetilde A$ connected and reduced, satisfying
\eqref{eq:formal-assumptions}. Moreover one of the following holds:
\begin{enumerate}
\item $H^1(\widetilde A,\widetilde\FF_j)=0$ for all $j\in\Z$;
\item for some $c>0$, \eqref{eq:acyclic-layers} holds on
$\widetilde\YY$, and $\Pic^0_{\widetilde A/k}$ is smooth
semiabelian of dimension two, with
$h^*(\OO_{\widetilde A})=(1,2,1)$.
\end{enumerate}
In case \textup{(i)}, $T$ is semiample.
\end{proposition}

The singularity and log canonical hypotheses are needed only
near $T$, and all the covers below are constructed on $\widehat V_T$.
Projectivity of $V$ is used for applying Lemma~\ref{lem:principal-parts}.

\Needspace{6\baselineskip}
\subsection{The log canonical cover}

First we make the log canonical line invertible. 

\begin{lemma}\label{lem:formal-log-cover}
Under \eqref{eq:geometric-hypotheses}, there is a cyclic
quasi-\'{e}tale formal cover of degree prime to $p$ with klt
ambient space and connected reduced Gorenstein slc boundary $A$
satisfying $\omega_A\simeq\OO_A$. All $\II_j$ are maximal Cohen--Macaulay.
\end{lemma}

\begin{proof}
For every $j\in\Z$, $\OO_V(jT)$ is Cohen--Macaulay near $T$ by
\cite[Theorems 3 and 19(2)]{BK23}, applied with
$\Delta=\Delta_D=0$. Here $jT$ is $\Q$-Cartier. The depth
lemma makes $T$ Cohen--Macaulay and hence $S_2$. Its generic
double points are nodes, so $T$ is weakly normal for $p\ne2$ by \cite[Lemma 2.3.4]{AP26}.
Lemma~\ref{lem:ambient-log-index} supplies $q_0$ prime to $p$
such that $q_0(K_V+T)$ is Cartier near $T$ and trivial on $T$.

Set $\mathcal L=\OO_{\widehat V_T}(q_0(K_V+T))$.
We construct an invertible sheaf $\mathcal U$ with
\[
 \mathcal U^{q_0}\simeq\mathcal L,\qquad
                               \mathcal U|_T\simeq\OO_T.
\]
Let $\mathcal K$ be the ideal of $T$ in $\widehat V_T$.
Completeness and Hensel's lemma, with $q_0$ invertible, give
\[
 1+\mathcal K\xrightarrow[\ u\mapsto u^{q_0}\ ]{\sim}1+\mathcal K;
\]
see \cite[Tag 0ALJ]{Sta}.
Choose local generators of $\mathcal L$ lifting a fixed
trivialization on $T$. Their transition functions satisfy
$g_{ab}\in1+\mathcal K$.
Let $h_{ab}\in1+\mathcal K$ be their unique $q_0$-th roots.
Uniqueness gives
\[
 (h_{ab}h_{bc})^{q_0}=g_{ab}g_{bc}=g_{ac}=h_{ac}^{q_0},
 \qquad h_{ab}h_{bc}=h_{ac}.
\]
The functions $h_{ab}$ therefore define $\mathcal U$ directly
on $\widehat V_T$, with the required power isomorphism and
trivial reduced restriction.
Form
the cyclic algebra
\begin{equation}\label{eq:formal-log-algebra}
 \bigoplus_{j=0}^{q_0-1}
       \OO_{\widehat V_T}(j(K_V+T))\mathcal U^{-j},
\end{equation}
whose multiplication uses the chosen $q_0$-th power isomorphism.
Its summands are Cohen--Macaulay. The algebra is \'{e}tale in codimension one, so it is normal. \cite[Proposition 5.20]{KM98} gives the singularity conditions.
The reflexive pullback of $\OO(K_V+T)$ is the invertible
log canonical line, and its restriction to $A$ is trivial.
Lemma~\ref{lem:quotient-duality} proves the Cohen--Macaulayness.

Idempotents lift along nilpotent thickenings. Choose a connected
component of the cover and replace the group by its stabilizer.
The components form one orbit because the quotient is connected.
Their images coincide and cover $T$, so the chosen component
still surjects onto $T$.
\end{proof}

\Needspace{6\baselineskip}
\subsection{Normal surfaces}

In this subsection, we deal with the case where $A$ in Lemma \ref{lem:formal-log-cover} is normal. The lc but non-klt case has vanishing first cohomology.
In the klt case a tame root removes the prime-to-$p$ part of
the normal divisor. The remaining surface is K3 or abelian,
where the cohomology of non-trivial torsion line bundles vanishes.

\begin{lemma}\label{lem:normal-surface-cover}
If the boundary in Lemma~\ref{lem:formal-log-cover} is normal,
then, after a further tame abelian quasi-\'{e}tale formal cover,
one of the cases of Proposition~\ref{prop:surface-cover}
holds.
\end{lemma}

\begin{proof}
If $A$ is not klt, Proposition~\ref{lem:normal-strict-lc-cohomology}
gives $H^1(A,\FF_j)=0$ for every $j$. Indeed the quotients $\FF_j$ are
rank-one reflexive sheaves, with the power maps induced by
$\II_N|_A\simeq\OO_A$.

Suppose $A$ is klt. At a regular point of $A$, the
Cohen--Macaulay module $\FF_1$ is free. Thus $A$ is Cartier
there by Lemma~\ref{lem:quotient-duality}, and
$\FF_j=\FF_1^{[j]}$ on the normal surface. By
\cite{BWX26}, some integer $b$ prime to $p$
makes $\FF_b$ invertible. Hence $bA$ is Cartier. Write
$N=p^vr$, $(r,p)=1$. B\'ezout's identity implies that $rA$
is Cartier.

Choose a formal root $\mathcal U^r=\II_N$ with trivial
reduced root and take the reflexive cyclic algebra of
$\II_{p^v}\mathcal U^{-1}$. Its pullback boundary $A'$ has
$p^vA'$ Cartier and trivial on $A'$. Since $rA'$ is also
Cartier, $A'$ is Cartier. The cover is \'{e}tale at the
generic points of curves in $A$, and its boundary is
Cohen--Macaulay and regular in codimension one, hence normal.

Let $f:\widetilde A'\to A'$ be the minimal resolution.
By Lemma~\ref{lem:normal-klt-classification}, $\widetilde A'$
is a K3 surface or an abelian surface. Since $A'$ is normal,
$f_*\OO_{\widetilde A'}=\OO_{A'}$.
For every invertible sheaf $L$ on $A'$, the projection
formula and the Leray spectral sequence give
\cite[Tags 01E6 and 01EY]{Sta}
\[
 f_*f^*L\simeq L,\qquad
 H^1(A',L)\hookrightarrow H^1(\widetilde A',f^*L).
\]
In particular, $f^*L\simeq\OO_{\widetilde A'}$ implies
$L\simeq\OO_{A'}$.
For a nontrivial torsion invertible sheaf $M$ on $\widetilde A'$,
$H^0(M)=H^2(M)=0$. Riemann--Roch gives
\[
 \chi(M)=\begin{cases}2&\text{on a K3 surface},\\0&\text{on an
 abelian surface}.
 \end{cases}
\]
Thus a K3 surface has no nontrivial torsion invertible
sheaf. It follows that
give $\Pic(A')_{\mathrm{tors}}=0$ and
$H^1(A',\OO_{A'})=0$, thereby proving (i).
If $\widetilde A'$ is abelian, it has no exceptional
rational curves; hence $f$ is an isomorphism.
Every nontrivial torsion invertible sheaf is acyclic
by the same Riemann--Roch calculation.
Taking the order of $\OO_{A'}(-A')$ gives (ii).

A root of order $r$ replaces a
character $\eta$ by the group
$\{(g,z):z^r=\eta(g)\}$, and its kernel is the cyclic grading
group. This is finite abelian of order prime to $p$.
\end{proof}

\Needspace{6\baselineskip}
\subsection{Nonnormal surfaces with elliptic conductor}
\label{subsec:elliptic-conductor}
In this subsection, we deal with the case where $A$ is non-normal with elliptic conductor.
Let $A$ be the boundary of Lemma~\ref{lem:formal-log-cover}.
Assume its normalized conductor is elliptic.
Use the normalization and conductor morphisms
\[
 \nu:\coprod_{i\in I}V_i\to A,\qquad
 q:\overline C\to C,\qquad\tau:\overline C\to\overline C.
\]
By Lemma~\ref{lem:nonnormal-conductor-types}, $\overline C$
is a disjoint union of smooth elliptic curves and the
components form a chain or a cycle.
For a fixed integer $j$, set
\[
       \mathcal L_i=(\FF_j\otimes\OO_{V_i})^{**}.
\]
The calculations on the normal surfaces $V_i$ are in
Appendix~\ref{app:surface}. We now specify how their
restriction maps compute cohomology on $A$.

\begin{lemma}\label{lem:elliptic-quotients}
If $\FF_j$ is not locally free at a generic conductor
point, then $H^*(A,\FF_j)=0$.
Otherwise there is an exact sequence
\begin{equation}\label{eq:conductor-sequence}
0\to\FF_j\to\nu_*\Bigl(\bigoplus_i\mathcal L_i\Bigr)
 \xrightarrow{\rho}(q_*\mathcal L_{\overline C})^-\to0,
\end{equation}
where $\mathcal L_{\overline C}$ is the restriction of the
$\mathcal L_i$ to $\overline C$ with its induced
$\tau$-linearization, and
$\rho=\tfrac12(1-\tau)\circ\operatorname{res}_{\overline C}$.
\end{lemma}

\begin{proof}
\noindent\emph{Step 1. The local modules and their power maps.}
At a split generic conductor point, write
\[
 R=K[[x,y]]/(xy),\qquad
 \overline R=K[[x]]\oplus K[[y]],\qquad
 \mathfrak c=(x)\oplus(y).
\]
A torsion-free module of rank one on both branches is
isomorphic to $R$ or $\overline R$.
For $R$ the product with its dual has image $R$;
for $\overline R$ its image is the conductor $\mathfrak c$.
The multiplication \eqref{eq:quotient-duality} and the
specified $N$-th power map therefore give
\begin{equation}\label{eq:conductor-fractions}
 \mathcal L_i^{[N]}
       \simeq\OO_{V_i}\Bigl(-\sum_{E\subset C_i}b_{i,E}E\Bigr),
                   \qquad0\leq b_{i,E}<N.
\end{equation}
Here $b_{i,E}=0$ at a free conductor point.
At a nonfree point, the two dual coefficients on a
fixed branch sum to $N$, because the evaluation image
has conductor order one. Neither coefficient is zero:
if the $N$-th power of the generator $(1,0)$ were
$(u(x),0)$ with $u(0)\ne0$, it would not belong to $R$.
Applying this to both branches and to the dual module
proves $0<b_{i,E}<N$ in the nonfree case.

\smallskip
\noindent\emph{Step 2. The nonfree case.}
On a component ruled over an elliptic curve, the degree
of \eqref{eq:conductor-fractions} on a general fibre gives
\[
              (b_{i,C_-}+b_{i,C_+})/N\in\{0,1\}.
\]
One coefficient is positive if and only if both are
positive. Nonfreeness therefore propagates through every
interior component of the connected chain or cycle.
It holds at all generic conductor points.
Lemma~\ref{lem:elliptic-ruling-sheaves},
Proposition~\ref{lem:rational-elliptic-sheaves}, and
Lemma~\ref{lem:contracted-elliptic-sheaves} make all
$\mathcal L_i$ acyclic.
Moreover
$\FF_j=\nu_*(\bigoplus_i\mathcal L_i)$:
the equality holds in codimension one by the module
$\overline R$, and both sheaves satisfy $S_2$.
This proves the vanishing assertion.

\smallskip
\noindent\emph{Step 3. The free case.}
Now all $b_{i,E}$ are zero. Near each conductor curve
$\mathcal L_i$ is invertible. Its generic gluing
isomorphism extends over the whole smooth curve:
its $N$-th power is the prescribed unit, so its valuation
at every closed point is zero.
Restriction to $\overline C$, followed by the
anti-invariant projector, is surjective because $p\ne2$.
The resulting kernel has depth two at closed points:
the normalization terms have depth two and the conductor
term has depth one.
This kernel agrees with $\FF_j$ in codimension one,
where it imposes equality of the two branch values.
The $S_2$ property proves \eqref{eq:conductor-sequence}.
This argument includes fixed points of $\tau$.
\end{proof}

If we are in the second case of Lemma \ref{lem:elliptic-quotients}, define
$\mathcal Q_j=(q_*\mathcal L_{\overline C})^-$.
The exact sequence, Serre duality, and
Lemma~\ref{lem:section-period} give
\begin{align}
 \chi(A,\FF_j)&=\sum_i\chi(V_i,\mathcal L_i)-\chi(C,\mathcal Q_j),
                                      \label{eq:conductor-Euler}\\
 h^1(A,\FF_j)&=2h^0(A,\FF_j)-\chi(A,\FF_j).
                                      \label{eq:Euler-first-cohomology}
\end{align}
Here $h^0(A,\FF_j)=1$ if $\FF_j\simeq\OO_A$, and it
is zero otherwise, by Lemma~\ref{lem:section-period}.
Thus it remains to compute the explicitly specified terms
in \eqref{eq:conductor-Euler}.

\begin{lemma}\label{lem:elliptic-cycle}
If the conductor graph is a cycle, then every
$\FF_j\not\simeq\OO_A$ is acyclic. Moreover
\begin{equation}\label{eq:elliptic-picard}
 h^*(\OO_A)=(1,2,1),\qquad
 1\to\Gm\to\Pic^0_{A/k}\to E^\vee\to0
\end{equation}
for an elliptic curve $E$. The integer $c$ of
Lemma~\ref{lem:section-period} satisfies
\eqref{eq:acyclic-layers}.
\end{lemma}

\begin{proof}
\noindent\emph{Step 1. Cohomology of $\FF_j$.}
The nonfree case was proved in
Lemma~\ref{lem:elliptic-quotients}.
In the free case, each ruling $f_i:V_i\to E_i$ satisfies
\[
 \mathcal L_i=f_i^*L_i,\qquad Rf_{i*}\mathcal L_i=L_i,
                         \qquad L_i\text{ torsion},
\]
by Lemma~\ref{lem:elliptic-ruling-sheaves}.
The conductor terms are torsion line bundles on elliptic
curves as well. All their Euler characteristics are zero.
Equations \eqref{eq:conductor-Euler} and
\eqref{eq:Euler-first-cohomology} give
\[
 h^*(A,\FF_j)=
 \begin{cases}(1,2,1),&\FF_j\simeq\OO_A,\\
               (0,0,0),&\FF_j\not\simeq\OO_A.
 \end{cases}
\]
The assertion about $c$ follows from
Lemma~\ref{lem:section-period}.

\smallskip
\noindent\emph{Step 2. The Picard scheme.}
Write \(\nu:\coprod_{i=1}^r V_i\to A\), with the components
ordered cyclically. Let \(f_i:V_i\to E_i\) be the elliptic
ruling, with conductor sections \(C_i^-,C_i^+\), and let
\(\varphi_i:C_i^+\xrightarrow{\sim}C_{i+1}^-\) be the gluing
isomorphism. Indices are taken modulo \(r\). Set
\[
 g_i=(f_{i+1}|_{C_{i+1}^-})\circ\varphi_i
                 \circ(f_i|_{C_i^+})^{-1},
 \qquad
 E=E_1,\qquad a=g_r\circ\cdots\circ g_1.
\]

Choose a generator \(\eta\) of \(\omega_A\), and denote its
pullback to \(V_i\) by
\[
 \eta_i\in
 H^0\bigl(V_i,\omega_{V_i}(C_i^-+C_i^+)\bigr).
\]
The residue theorem on the fibres of \(f_i\) gives a nonzero
\(\omega_i\in H^0(E_i,\omega_{E_i})\) such that
\[
 \operatorname{Res}_{C_i^-}\eta_i
       =(f_i|_{C_i^-})^*\omega_i,
 \qquad
 \operatorname{Res}_{C_i^+}\eta_i
       =-(f_i|_{C_i^+})^*\omega_i.
\]
Since the \(\eta_i\) come from a section of \(\omega_A\),
their residues on the identified branches satisfy
\[
 \operatorname{Res}_{C_i^+}\eta_i+
 \varphi_i^*\operatorname{Res}_{C_{i+1}^-}\eta_{i+1}=0.
\]
Thus \(g_i^*\omega_{i+1}=\omega_i\), and hence
\(a^*\omega_1=\omega_1\).
Since \(p>3\), an automorphism of an elliptic curve preserving
a nonzero regular differential is a translation.
Consequently \(a^*=\mathrm{id}\) on \(E^\vee\).

Let
\[
 P=(\nu^*)^{-1}
 \left(\prod_i\operatorname{Pic}^0_{V_i/k}\right)
 \subset\operatorname{Pic}_{A/k}.
\]
Under the identifications
\(f_i^*:E_i^\vee\simeq\operatorname{Pic}^0_{V_i/k}\),
compatible line-bundle classes satisfy
\(M_i=g_i^*M_{i+1}\).
They are therefore parametrized by
\(\ker(1-a^*:E^\vee\to E^\vee)=E^\vee\).
Gluing line bundles along the conductor gives
\[
 1\longrightarrow\mathbf G_m
 \longrightarrow P
 \longrightarrow E^\vee
 \longrightarrow0.
\]
Indeed, the kernel consists of scalar gluing parameters
\((\lambda_i)\), modulo the changes
\[
 \lambda_i\longmapsto(u_{i+1}/u_i)\lambda_i;
\]
their product identifies this quotient with \(\mathbf G_m\).
The same construction applies to families, so the sequence
is exact as a sequence of fppf group schemes
\cite[Tag 0D2J]{Sta}.
Thus \(P\to E^\vee\) is a \(\mathbf G_m\)-torsor, and \(P\)
is smooth and connected.
Since \(P\) is an open and closed subgroup scheme containing
the identity, \(P=\operatorname{Pic}^0_{A/k}\).
This proves \eqref{eq:elliptic-picard} and smoothness.
\end{proof}

The chain is the graph defined in the proof of
Lemma~\ref{lem:nonnormal-conductor-types}. An endpoint is
either a component with one elliptic conductor curve or
an auxiliary vertex for an involution $E\to E/\tau\simeq\mathbb P^1$.
This description fixes the objects referred to in the
following calculation.

\begin{lemma}\label{lem:elliptic-path}
Suppose that the conductor graph of $A$ is a chain.
Then either $H^1(A,\FF_j)=0$ for every $j\in\Z$, or there
is a finite quasi-\'{e}tale double cover $\pi:\YY'\to\YY$
whose reduced boundary $A'=(\pi^{-1}A)_{\mathrm{red}}$ is
connected and has conductor graph a cycle of elliptic ruled
components. In the latter case, the deck involution $\iota$
satisfies
\[
 \dim_k H^1(A',\OO_{A'})=2,\qquad
 \iota^*|_{H^1(A',\OO_{A'})}=-\operatorname{id}.
\]
\end{lemma}

\begin{proof}
\noindent\emph{Step 1. Restrictions to the conductor.}
Fix $j\in\Z$. By Lemma~\ref{lem:elliptic-quotients}, assume
that $\FF_j$ is free at generic conductor points.
On each interior ruling, both conductor restrictions of
$\mathcal L_i=f_i^*L_i$ identify with $L_i$.
Thus nontriviality of one restriction propagates along the chain.
In that case Lemmas~\ref{lem:elliptic-ruling-sheaves} and
\ref{lem:contracted-elliptic-sheaves},
Proposition~\ref{lem:rational-elliptic-sheaves},
and Corollary~\ref{cor:exceptional-restriction} give
$H^*(V_i,\mathcal L_i)=0$. The conductor line bundles are
also acyclic. Hence $H^*(C,\mathcal Q_j)=0$, and
\eqref{eq:conductor-sequence} gives $H^*(A,\FF_j)=0$.

Assume now that all conductor restrictions are trivial.
Every interior $\mathcal L_i$ is then trivial.
At a rational endpoint, Corollary~\ref{cor:exceptional-restriction}
gives $\mathcal L_i\simeq\OO_{V_i}$ or $\varepsilon_i$.
Here $\varepsilon_i$ denotes the order-two sheaf in that corollary.
At an endpoint obtained by contracting a section, a nontrivial
$\mathcal L_i$ would give, for its conductor curve $E$,
\[
 k=H^0(E,\OO_E)\hookrightarrow H^1(V_i,\mathcal L_i(-E))
           \simeq H^1(V_i,\mathcal L_i^\vee)^*=0.
\]
The injection comes from restriction to $E$, and the vanishing
is Lemma~\ref{lem:contracted-elliptic-sheaves}.
Thus $\mathcal L_i\simeq\OO_{V_i}$ there.

\smallskip
\noindent\emph{Step 2. Cohomology from the two endpoints.}
At a surface endpoint $e$, put
$t_e=0$ if $\mathcal L_i\simeq\OO_{V_i}$ and $t_e=1$ if
$\mathcal L_i\simeq\varepsilon_i$. The cited surface results give
\[
                       \chi(V_i,\mathcal L_i)=1-2t_e.
\]
For an auxiliary endpoint $q_e:E\to E/\tau\simeq\mathbb P^1$,
the invariant and anti-invariant summands are
\[
 q_{e*}\OO_E=\OO_{\mathbb P^1}\oplus\OO_{\mathbb P^1}(-2).
\]
Put $t_e=0$ for the natural linearization of $\OO_E$ and
$t_e=1$ for its negative. Twisting the linearization
interchanges the two summands, so
\[
 \mathcal Q_j|_{\mathbb P^1}\simeq\OO_{\mathbb P^1}(-2+2t_e),
 \qquad-\chi(\mathbb P^1,\mathcal Q_j|_{\mathbb P^1})=1-2t_e.
\]
All interior contributions vanish. For the two endpoints
$e_1,e_2$ and $t=t_{e_1}+t_{e_2}$, \eqref{eq:conductor-Euler} gives
\begin{equation}\label{eq:path-Euler}
                         \chi(A,\FF_j)=2-2t.
\end{equation}
If $t=0$, rescale generators successively along the chain
to make the paired conductor restrictions agree.
The natural endpoint linearizations preserve these generators,
so $\FF_j\simeq\OO_A$.
If $t>0$, a nontrivial endpoint sheaf or a negative
linearization prevents such an isomorphism.
Lemma~\ref{lem:section-period} and
\eqref{eq:Euler-first-cohomology} therefore give
\begin{equation}\label{eq:path-cohomology}
 h^*(A,\FF_j)=
 \begin{cases}(1,0,1),&t=0,\\
               (0,0,0),&t=1,\\
               (0,2,0),&t=2.
 \end{cases}
\end{equation}
Taking $j=0$ gives $H^1(A,\OO_A)=0$.

\smallskip
\noindent\emph{Step 3. The double cover and the involution.}
Suppose $H^1(A,\FF_j)\ne0$. Then $t_{e_1}=t_{e_2}=1$.
The exceptional endpoint sheaves have trivial reflexive
squares, and the negative linearizations square to the natural
ones. Multiplication and the gluing argument in Step~2 give
$\FF_{2j}\simeq\OO_A$.
For the period $d$ of Lemma~\ref{lem:section-period},
\[
 d\mid2j,\qquad d\nmid j,\qquad
 2\mid d,\qquad j\equiv d/2\pmod d.
\]
Periodicity gives $\FF_{d/2}\simeq\FF_j$.
Choose $\mathcal U^2\simeq\II_d$ with $\mathcal U|_A\simeq\OO_A$
by the formal root construction of Lemma~\ref{lem:formal-log-cover}.
The multiplication on divisorial ideals defines the double cover
\begin{equation}\label{eq:elliptic-double-algebra}
 \pi:\YY'\longrightarrow\YY,\qquad
 \pi_*\OO_{\YY'}=\OO_{\YY}\oplus\II_{d/2}\mathcal U^{-1}.
\end{equation}
The same local calculation proves normality, quasi-\'{e}taleness,
and the $S_2$ property of $A'$.
For $\pi_A=\pi|_{A'}$, reduction of the two graded summands gives
\[
 \pi_{A*}\OO_{A'}\simeq\OO_A\oplus\FF_{d/2}.
\]
The deck involution acts by $+1$ and $-1$ on these summands.
Taking eigenspaces in cohomology, finiteness of $\pi_A$ gives
\[
 \begin{aligned}
 H^0(A',\OO_{A'})&\simeq k,\\
 H^1(A',\OO_{A'})^+&\simeq H^1(A,\OO_A)=0,\\
 H^1(A',\OO_{A'})^-&\simeq H^1(A,\FF_{d/2})\simeq k^2.
 \end{aligned}
\]
Thus $A'$ is connected, and the asserted action on
$H^1(A',\OO_{A'})$ follows.

Each interior normalization component splits into two copies.
Each rational endpoint lifts to the elliptic ruling with two
conductor sections of Proposition~\ref{lem:rational-elliptic-sheaves}.
At an auxiliary endpoint, the negative linearization pairs
the conductor curves in the two copies by
$(x,+)\sim(\tau(x),-)$.
Consequently every vertex of the new conductor graph has degree
two. The graph is connected because $A'$ is connected and $S_2$,
so it is a cycle.
The preceding abelian group $G_0$ lifts to
$\{(g,z)\in G_0\times k^*:z^2=\eta(g)\}$, where $\eta$ is
its character on $\II_d|_A$. This group remains abelian
of order prime to $p$.
\end{proof}

\subsection{Nonnormal surfaces with rational conductor}
\label{subsec:rational-conductor}
In this subsection, we deal with the case where $A$ is non-normal with rational conductor.

\begin{lemma}\label{lem:power-splitting}
Suppose $A$ is Gorenstein, $\omega_A\simeq\OO_A$,
Frobenius is bijective on $H^*(A,\OO_A)$, and
\eqref{eq:quotient-duality} holds. If
$\FF_{p^ej}\simeq\OO_A$ but $\FF_j\not\simeq\OO_A$,
then $H^*(A,\FF_j)=0$.
\end{lemma}

\begin{proof}
Set $Q=p^e$. Duality applied to
$\OO_A\to F_*^e\OO_A$ gives a retraction
$\tau:F_*^e\OO_A\to\OO_A$, with $\tau(a^Q)=a$.
The multiplication pairing defines
\begin{equation}\label{eq:power-retraction}
 R:F_*^e\FF_{Qj}\to\FF_j,\qquad
                \langle R(z),t\rangle=\tau(zt^Q)
                      \quad(t\in\FF_{-j}).
\end{equation}
By \eqref{eq:quotient-duality}, $R(s^Q)=s$.
Thus $\FF_j$ is a direct summand of $F_*^e\FF_{Qj}$.
Under the chosen identification $\FF_{Qj}\simeq\OO_A$,
its power map has image in $\ker\tau$, because
\[
 \Hom_A(\FF_j,\OO_A)=H^0(A,\FF_{-j})=0
\]
by Lemma~\ref{lem:section-period}.
Finally $F_*^e\OO_A=\OO_A\oplus\ker\tau$, and the
first summand induces an isomorphism on cohomology by
the Frobenius hypothesis. Hence $\ker\tau$, and its
direct summand $\FF_j$, are acyclic.
\end{proof}

\begin{lemma}\label{lem:rational-surface-cover}
If $A$ has rational conductor, a further tame abelian
formal cover satisfies one of the cases of
Proposition~\ref{prop:surface-cover}.
\end{lemma}

\begin{proof}
Write $N=p^vr$ with $(r,p)=1$.
Choose $\mathcal U^r\simeq\II_N$ and
$\mathcal U|_A\simeq\OO_A$, and take the normalized
cyclic algebra of $\II_{p^v}\mathcal U^{-1}$.
For its connected boundary $A'$, the divisor $p^vA'$
is Cartier and its associated line bundle is trivial on $A'$.
Thus the new period $c$ divides $p^v$.

Proposition~\ref{cor:rational-conductor-picard} gives
\[
 H^*(A',\OO_{A'})\simeq H^*(\Delta,k),\qquad
 \Pic^0_{A'/k}\simeq\Gm^{b_1(\Delta)},\qquad b_1(\Delta)\in\{0,2\}.
\]
Frobenius is bijective. For $c\nmid j$, the quotient
$\FF_j$ is nontrivial and $\FF_{p^vj}\simeq\OO_{A'}$.
Lemma~\ref{lem:power-splitting} gives $H^*(A',\FF_j)=0$.
For $c\mid j$, $\FF_j\simeq\OO_{A'}$.
If $b_1(\Delta)=0$, all first cohomology groups vanish.
If $b_1(\Delta)=2$, \eqref{eq:acyclic-layers} and
case (ii) follow. The cyclic root construction
lifts the original action through a finite abelian
extension of order prime to $p$, as in
Lemma~\ref{lem:normal-surface-cover}.
\end{proof}

\begin{proof}[Proof of Proposition~\ref{prop:surface-cover}]
Take the cover of Lemma~\ref{lem:formal-log-cover}.
If $A$ is normal, apply Lemma~\ref{lem:normal-surface-cover}.
Otherwise Lemma~\ref{lem:nonnormal-conductor-types}
separates the elliptic and rational cases.
Lemmas~\ref{lem:elliptic-cycle}, \ref{lem:elliptic-path},
and \ref{lem:rational-surface-cover} prove the two cases.

In (i), filter the restriction map
$\OO_{NA}(NA)\to\OO_A(NA)$ by the quotients $\FF_j$.
Every kernel has zero $H^1$.
The given generating section on $A$ therefore lifts to
$NA$. The line bundle and its initial section are pulled
back from $V$. Averaging each lift over the finite group
preserves its restriction. Exactness of invariants
descends the generating section to
$\OO_V(NT)|_{NT}$.
Lemma~\ref{lem:principal-parts} proves semiampleness.
\end{proof}

\subsection{Descent and the action on first cohomology}

The surface calculations give a tame abelian cover but
do not bound the order of its group. We quotient by the
kernel of its action on first cohomology. Three assertions
must be checked separately: the period $c$,
the connected Picard scheme, and the orders of the
remaining characters.

\begin{proposition}\label{prop:bounded-action}
In case \textup{(ii)} of
Proposition~\ref{prop:surface-cover}, set
\[
 H=\ker\bigl(\widetilde G\to
       \operatorname{GL}(H^1(\widetilde A,\OO_{\widetilde A}))\bigr),
 \qquad \YY=\widetilde\YY/H,\qquad A=\widetilde A/H.
\]
The pair $(\YY,A)$ satisfies
\eqref{eq:formal-assumptions}, \eqref{eq:acyclic-layers},
and \eqref{eq:semiabelian-hypothesis}.
The group $\widetilde G/H$ acts faithfully on $H^1(A,\OO_A)$,
and each of its characters has order at most $12$.
\end{proposition}

\begin{proof}
\noindent\emph{Step 1. The new period.}
Let $c_0$ be the period on $\widetilde A$, let $\lambda$
be the character on $\widetilde\II_{c_0}|_{\widetilde A}$,
and set $t=\ord(\lambda|_H)$.
For the finite quotient $q:\widetilde\YY\to\YY$,
\[
 (q_*\widetilde\II_j)^H=\II_j,\qquad
 (q_*\widetilde\FF_j)^H=\FF_j.
\]
The first equality follows in codimension one from
quasi-\'{e}taleness and extends by reflexivity; the second
uses exactness of invariants.
Thus the cohomology is zero if $c_0\nmid j$, and
\begin{equation}\label{eq:kernel-invariants}
 H^i(A,\FF_{c_0b})
  =\bigl(H^i(\widetilde A,\OO_{\widetilde A})
                                      \otimes\lambda^b\bigr)^H.
\end{equation}
The group $H$ fixes $H^0$ and $H^1$, and also $H^2$
because the cup product $\bigwedge^2H^1\to H^2$ is
an isomorphism. Hence the right side is zero unless
$t\mid b$.

Before applying Lemma~\ref{lem:section-period} on $\YY$,
we verify its hypothesis on the log canonical sheaf.
The reduced character of
$\widetilde{\mathcal M}=\omega_{\widetilde\YY}(\widetilde A)$
is trivial on $H$ by duality with $H^2$.
On an affine inverse image, lift its invariant reduced
generator and average over $H$. Nakayama's lemma makes
the average a generator. Its invariants therefore descend
$\widetilde{\mathcal M}$ to the invertible sheaf
$\mathcal M=\omega_{\YY}(A)$ with $\mathcal M|_A\simeq\OO_A$.
The quotient remains klt, and the Cohen--Macaulay property
of the divisorial ideals follows from
\cite[Theorems 3 and 19(2)]{BK23}.
Thus \eqref{eq:formal-assumptions} holds on $\YY$.
For $c=c_0t$, \eqref{eq:kernel-invariants} gives
$h^0(A,\FF_c)=1$.
Lemma~\ref{lem:section-period} now makes $cA$ Cartier and
$\FF_c\simeq\OO_A$. Formula~\eqref{eq:kernel-invariants}
then proves \eqref{eq:acyclic-layers}.

\smallskip
\noindent\emph{Step 2. The Picard scheme.}
Set $P=\Pic^0_{\widetilde A/k}$.
Each $h\in H$ induces the identity on
$\operatorname{Lie}(P)=H^1(\widetilde A,\OO_{\widetilde A})$.
If $h$ has order $r$, choose regular parameters $x_1,x_2$
at the identity of $P$ and set
\[
                    y_i=r^{-1}\sum_{a=0}^{r-1}(h^a)^*x_i.
\]
Then $y_i\equiv x_i\pmod{\mathfrak m^2}$ and $h^*y_i=y_i$.
They are formal parameters, so $h$ acts identically on
the completed local ring. The equality extends to $P$
because $P$ is smooth and connected. Thus $H$ fixes $P$.

For $q_A:\widetilde A\to A$, let
$u=q_A^*:\Pic^0_{A/k}\to P$.
A line bundle with trivial pullback is determined by a
character of $H$ on the trivial line bundle $\OO_{\widetilde A}$. Hence
\[
                      \ker u\subseteq\Hom(H,\Gm),
\]
a finite \'{e}tale group. For a line bundle $L$ on
$\widetilde A$, the tensor product
$\bigotimes_{h\in H}h^*L$ descends to $A$.
This construction is compatible with base change.
It defines $v:P\to\Pic^0_{A/k}$ with
\[
                         uv=[|H|]_P.
\]
Thus $u$ is an isogeny with \'{e}tale kernel.
Also
$\dim\operatorname{Lie}(\Pic^0_{A/k})
 =\dim H^1(A,\OO_A)=2=\dim P$,
so $\Pic^0_{A/k}$ is smooth.
Its connected unipotent subgroup maps trivially to
the semiabelian variety $P$ and must lie in the finite
kernel; it is therefore zero. Consequently
$\Pic^0_{A/k}$ is semiabelian of dimension two.
Together with \eqref{eq:kernel-invariants}, this proves
\eqref{eq:semiabelian-hypothesis}.

\smallskip
\noindent\emph{Step 3. The orders of characters.}
Write $0\to T\to\Pic^0_{A/k}\to B\to0$, with $T$
a torus and $B$ an abelian variety.
A finite-order group automorphism is detected on
\[
             X^*(T)\otimes\Q\quad\text{and}\quad V_\ell(B),
                         \qquad \ell\ne p.
\]
Indeed an automorphism trivial on both would differ
from the identity by a homomorphism $B\to T$, which is zero.
The characteristic polynomial is integral, of degree
$\dim T+2\dim B\leq4$, where for the abelian factor we use
\cite[Proposition 12.9]{Mil86}.
Its cyclotomic factors have total degree at most four.
The possible orders of the automorphism are therefore
\[
                          1,2,3,4,5,6,8,10,12.
\]
The residual action on the tangent space is faithful
by the definition of $H$. Every character of a finite
abelian group attains its order on an element.
Hence the assertion holds.
\end{proof}

Let $\mu^{-1}$ be the character on $\JJ|_A$ after this
quotient, and set $R=\ord(\mu)\leq12$.
If
\[
 \dim_k\Bigl(\bigoplus_{j=1}^{R-1}
                    H^1(A,\OO_A)_{\mu^j}\Bigr)\leq1,
\]
Corollary~\ref{cor:one-direction} and
Lemma~\ref{lem:invariant-descent} prove semiampleness.
Otherwise both characters of $H^1(A,\OO_A)$ belong to
$\langle\mu\rangle$. Then $\ker\mu$ acts trivially on
$H^1(A,\OO_A)$, so it is trivial by the proposition.
The remaining group is cyclic of order $R<p$, as
required in Section~\ref{sec:connecting}.

\section{Picard schemes and connecting homomorphisms}\label{sec:connecting}

\begin{notation}\label{not:equivariant}
Throughout Sections~\ref{sec:connecting} and~\ref{sec:finite},
let $\YY$ be a normal Cohen--Macaulay formal threefold
complete along a connected reduced proper divisor $A$.
Assume that the following conditions hold:
\[
 \begin{gathered}
 \II_j=\OO_{\YY}(-jA)\text{ is maximal Cohen--Macaulay}
                                      \quad(j\in\Z),\\
 \FF_j=\II_j/\II_{j+1},\qquad
 \MM=\omega_{\YY}(A)\text{ is invertible},\qquad
 \MM|_A\simeq\OO_A,\\
 \JJ=\II_c\text{ is invertible},\quad
 \JJ|_A\simeq\OO_A,\quad
 H^*(A,\FF_j)=0\quad(0<j<c),\\
 D_\ell=V(\JJ^\ell),\qquad D=D_1=cA .
 \end{gathered}
\]
The $D_\ell$ are proper over $k$.
These are the hypotheses of
Settings~\ref{not:formal} and~\ref{not:cartier-thickenings}.
Assume also that
\begin{equation}\label{eq:semiabelian-hypothesis}
 \Pic^0_{A/k}\text{ is smooth semiabelian},\qquad
                         h^1(A,\OO_A)=2.
\end{equation}
Let a finite abelian group $G$ of order prime to $p$ act
on these data. Set
\[
                         V=H^1(D,\OO_D)=H^1(A,\OO_A).
\]
Fix a generating section $s\in H^0(D,\JJ|_D)$, and write
$g(s)=\mu(g)^{-1}s$. Set $R=\ord(\mu)$ and assume $R<p$.
\end{notation}

Lemmas~\ref{lem:quotient-duality}
and~\ref{lem:cohomology-and-picard} give
\begin{equation}\label{eq:perfect-cup}
 H^0(D,\OO_D)=k,\qquad
 \bigwedge\nolimits^2V\xrightarrow{\sim}H^2(D,\OO_D)=k.
\end{equation}
The last identification uses a generating section of
$\omega_A$ and canonical trace.
The cup product is nondegenerate by Serre duality on the
Gorenstein surface $A$. It is alternating since $p\ne2$.

In this section, we prove two boundary formulas (see Proposition \ref{prop:first-boundary} and Proposition \ref{prop:canonical-boundary}), and a lifting result for certain character subaspaces of the first cohomology of some thickenings of $D$ (see Proposition \ref{lem:witt-lift}).
Fix a generating section $s$ of $\JJ|_D$.
If a generating section $\sigma$ of $\JJ^q|_{D_\ell}$
reduces to $s^q$, let $\delta^0_{\JJ^{-q}}$ be the boundary
for restriction of $\JJ^{-q}$ from $D_{\ell+1}$ to $D_\ell$.
The obstruction to extending this specified section is
\[
 \theta=s^{q-\ell}\delta^0_{\JJ^{-q}}(\sigma^{-1})\in V.
\]
By Lemma \ref{lem:obstruction-rules}, $\theta=0$ means that $\sigma$ extends, and nontriviality
of $\JJ^q|_{D_{\ell+1}}$ means that no generating section
on $D_\ell$ extends. 
We consider the following cohomology maps
\[
 \partial_{r,\ell}:
 H^1(\JJ^r/\JJ^{r+\ell})
       \longrightarrow H^2(\JJ^{r+\ell}/\JJ^{r+\ell+1})=k.
\]
For $\alpha\in H^1(\JJ^r/\JJ^{r+\ell})$, $\alpha_0\in V$ denotes its restriction
to $D$, using $s$ to identify the tensor powers of $\JJ|_D$.
The two bounadry formulas are
\begin{align*}
 \partial_{r,\ell}(\alpha)&=r\,\alpha_0\smile\theta
       &&\text{if }q=1\text{ and }\ell<p,\\
 \partial_{r,\ell}(\alpha)&=
       \frac{\beta+\ell+r}{q}\,\alpha_0\smile\theta
       &&\text{if }\MM\JJ^\beta\simeq\OO_{\YY}
          \text{ and }q\mid\beta+\ell+r.
\end{align*}

\subsection{Restriction maps and obstruction classes}

We first specify the two restriction sequences used below.
For $r\in\Z$ and $\ell\geq1$, the inclusion
$\JJ^{r+\ell+1}\subset\JJ^{r+\ell}\subset\JJ^r$ gives
\begin{equation}\label{eq:restriction-sequence}
 0\longrightarrow\JJ^{r+\ell}/\JJ^{r+\ell+1}
 \longrightarrow\JJ^r/\JJ^{r+\ell+1}
 \longrightarrow\JJ^r/\JJ^{r+\ell}\longrightarrow0.
\end{equation}
Its connecting homomorphism, followed by multiplication
by $s^{-r-\ell}$, is denoted by
\[
 \partial_{r,\ell}:H^1(\JJ^r/\JJ^{r+\ell})
                    \longrightarrow H^2(D,\OO_D)=k.
\]
For a class $\alpha$ in the source, set
\[
 \alpha_0=s^{-r}(\alpha|_D)\in V.
\]
The notation includes these identifications with powers
of $s$; it does not identify sheaves on different thickenings.

For $q\geq1$, the restriction sequence for $\JJ^{-q}$ is
\begin{equation}\label{eq:section-obstruction-sequence}
 0\longrightarrow\JJ^{\ell-q}/\JJ^{\ell-q+1}
 \longrightarrow\JJ^{-q}|_{D_{\ell+1}}
 \longrightarrow\JJ^{-q}|_{D_\ell}\longrightarrow0.
\end{equation}
Write its degree-zero connecting homomorphism as
\[
 \epsilon_{q,\ell}:H^0(D_\ell,\JJ^{-q}|_{D_\ell})
              \longrightarrow H^1(D,\JJ^{\ell-q}|_D).
\]
If $\sigma$ generates $\JJ^q|_{D_\ell}$ and $\sigma|_D=s^q$,
let $\sigma^{-1}$ be its dual section and define
\begin{equation}\label{eq:theta-definition}
 \theta_{q,\ell}(\sigma)
             :=s^{q-\ell}\epsilon_{q,\ell}(\sigma^{-1})\in V.
\end{equation}
The inverse in this definition fixes the sign in the
two boundary formulas.

We will use the following elementary product rule.
It is stated explicitly to specify the source and target
of each term in the formulas that follow.

\begin{lemma}\label{lem:boundary-product}
Set $Z=D_{\ell+1}$, $Z_0=D_\ell$, and
$\mathcal K_\ell=\JJ^\ell/\JJ^{\ell+1}$.
For an invertible sheaf $\mathcal H$ on $Z$, denote by
\[
 \delta^j_{\mathcal H}:H^j(Z_0,\mathcal H|_{Z_0})
        \longrightarrow H^{j+1}(D,\mathcal H|_D\otimes\mathcal K_\ell)
\]
the boundary of its restriction sequence.
For invertible sheaves $\mathcal L,\mathcal N$ on $Z$,
$\alpha\in H^1(Z_0,\mathcal L|_{Z_0})$, and
$\tau\in H^0(Z_0,\mathcal N|_{Z_0})$, one has
\begin{equation}\label{eq:boundary-product}
 \delta^1_{\mathcal L\mathcal N}(\alpha\tau)
  =\delta^1_{\mathcal L}(\alpha)\tau|_D
       -\alpha|_D\smile\delta^0_{\mathcal N}(\tau).
\end{equation}
Both sides belong to
$H^2(D,\mathcal L|_D\otimes\mathcal N|_D\otimes\mathcal K_\ell)$.
For two degree-zero classes, the corresponding formula has
a plus sign between its two terms.
\end{lemma}

\begin{proof}
Fix an affine covering $\mathcal U=(U_i)$ of $Z$, and set
$U_{i_0\ldots i_r}=U_{i_0}\cap\cdots\cap U_{i_r}$.
We use the \v Cech complex
\[
 C^r(\mathcal U,\mathcal F)
   =\prod_{i_0,\ldots,i_r}\Gamma(U_{i_0\ldots i_r},\mathcal F),
\]
with differential
\[
 (dc)_{i_0\ldots i_{r+1}}
   =\sum_{a=0}^{r+1}(-1)^a
          c_{i_0\ldots\widehat{i_a}\ldots i_{r+1}}|_{U_{i_0\ldots i_{r+1}}}.
\]
For cochains $a$ and $b$ of degrees $r$ and $t$, respectively,
their cup product is
\[
 (a\smile b)_{i_0\ldots i_{r+t}}
   =a_{i_0\ldots i_r}|_{U_{i_0\ldots i_{r+t}}}
       \otimes b_{i_r\ldots i_{r+t}}|_{U_{i_0\ldots i_{r+t}}}.
\]
It satisfies the usual Leibniz rule
\cite[Tag 01FP]{Sta}:
\begin{equation}\label{eq:cech-leibniz}
 d(a\smile b)=da\smile b+(-1)^r a\smile db.
\end{equation}
Choose lifts $a\in C^1(\mathcal U,\mathcal L)$ and
$b\in C^0(\mathcal U,\mathcal N)$ of representatives
of $\alpha$ and $\tau$. Their coboundaries have coefficients
in $\mathcal L\mathcal K_\ell$ and
$\mathcal N\mathcal K_\ell$.
Since $\JJ\mathcal K_\ell=0$, the other factor in each
product in \eqref{eq:cech-leibniz} restricts to $D$.
Passing to cohomology gives \eqref{eq:boundary-product}.
For two degree-zero classes the same rule has sign $+1$.
\end{proof}

\begin{lemma}\label{lem:obstruction-rules}
Let $\sigma$ and $\theta=\theta_{q,\ell}(\sigma)$ be as above.
\begin{enumerate}
\item The specified section $\sigma$ extends to a generating
section on $D_{\ell+1}$ if and only if $\theta=0$.
\item For every $b\in\Z$, the boundary of $\sigma^{-b}$ satisfies
\begin{equation}\label{eq:integer-power-obstruction}
 s^{bq-\ell}\delta^0_{\JJ^{-bq}}(\sigma^{-b})=b\theta.
\end{equation}
\item If $\sigma$ is equivariant with character $\mu^{-q}$,
then $\theta\in V_{\mu^\ell}$.
\item If $u\in H^0(D_\ell,\OO_{D_\ell}^*)$ and $u|_D=1$,
then
\begin{equation}\label{eq:unit-obstruction-change}
 \theta_{q,\ell}(u\sigma)
     =\theta_{q,\ell}(\sigma)-s^{-\ell}\delta^0_{\OO}(u).
\end{equation}
\end{enumerate}
\end{lemma}

\begin{proof}
Exactness of \eqref{eq:section-obstruction-sequence}
identifies $\theta=0$ with extension of $\sigma^{-1}$.
Such an extension generates at every point by Nakayama's
lemma, since its reduction generates on $D$.
Taking its inverse proves (i), including the prescribed
reduction to $\sigma$.

The degree-zero product rule of
Lemma~\ref{lem:boundary-product} gives, for $b\geq0$,
\[
 \delta^0_{\JJ^{-bq}}(\sigma^{-b})
  =b\,s^{-q(b-1)}\delta^0_{\JJ^{-q}}(\sigma^{-1}).
\]
Apply the same rule to a section and its inverse to obtain
the formula for negative $b$ as well. Multiplication by
$s^{bq-\ell}$ proves (ii). In particular,
\[
       s^{-q-\ell}\delta^0_{\JJ^q}(\sigma)=-\theta.
\]
No integer other than a unit in $k$ has been inverted.

For (iii), the boundary is equivariant.
The characters of $\sigma^{-1}$ and $s^{q-\ell}$ are
$\mu^q$ and $\mu^{\ell-q}$, respectively. Their product
is $\mu^\ell$.
Finally $\delta^0_{\OO}(u^{-1})=-\delta^0_{\OO}(u)$
because $u|_D=1$; applying the product rule to
$u^{-1}\sigma^{-1}$ proves (iv).
\end{proof}

Thus $\theta$ is the obstruction for a specified trivialization.
Nontriviality of the line bundle on $D_{\ell+1}$ is the
stronger assertion that this class is nonzero for every
choice of $\sigma$ on $D_\ell$.

\subsection{Picard schemes and the short-length formula}

The first boundary formula requires a vanishing statement
for the structure sheaf. We obtain it from the Picard
scheme, using the truncated logarithm only for ideals
whose $p$-th power is zero.

\begin{lemma}\label{lem:short-picard}
Let $Z$ be a proper nilpotent thickening of $D$, with ideal
$\mathcal K$ satisfying $\mathcal K^p=0$.
Let $P$ be the inverse image of $\Pic^0_{D/k}$ under
$\Pic_{Z/k}\to\Pic_{D/k}$.
There is an exact sequence of smooth group schemes
\begin{equation}\label{eq:short-picard}
 0\longrightarrow H^1(Z,\mathcal K)
 \longrightarrow P\longrightarrow\Pic^0_{D/k}\longrightarrow0.
\end{equation}
In particular, $H^1(Z,\OO_Z)\to V$ is surjective.
\end{lemma}

\begin{proof}
\noindent\emph{Step 1. The obstruction homomorphism.}
The polynomials
\[
 \log(1+x)=\sum_{j=1}^{p-1}\frac{(-1)^{j+1}x^j}{j},
 \qquad \exp(x)=\sum_{j=0}^{p-1}\frac{x^j}{j!}
\]
give inverse group isomorphisms
$(1+\mathcal K,\cdot)\simeq(\mathcal K,+)$.
The usual polynomial identities apply because all terms
of degree at least $p$ vanish in $\mathcal K$.
The units sequence \cite[Tag 0C6Q]{Sta} therefore gives
an exact sequence of fppf group functors
\begin{equation}\label{eq:short-picard-obstruction}
 0\longrightarrow H^1(Z,\mathcal K)\longrightarrow P
 \longrightarrow\Pic^0_{D/k}
          \xrightarrow{o}H^2(Z,\mathcal K).
\end{equation}
To justify this sequence after base change, let $B$ be a
$k$-algebra. Coherent cohomology on $Z_B$ is obtained by
tensoring over $k$, and the logarithm commutes with that
tensor product. Moreover
$H^0(D_B,\OO_{D_B})=B$; its units lift to $Z_B$ as constants.
The preceding boundary on global units is consequently zero,
which gives the initial injectivity in
\eqref{eq:short-picard-obstruction}.

\smallskip
\noindent\emph{Step 2. Vanishing on the semiabelian quotient.}
Write $0\to T\to\Pic^0_{D/k}\to A_0\to0$, with $T$ a
torus and $A_0$ an abelian variety.
Every group homomorphism $T\to\Ga$ is zero:
for $T=\Gm^d$, the identity $f(xy)=f(x)+f(y)$ has only
the zero solution among Laurent polynomials.
Thus $o|_T=0$, and $o$ factors through $A_0$.
A morphism from the proper connected variety $A_0$ to a
vector group is constant, and a constant group homomorphism
is zero. Hence $o=0$, proving exactness of
\eqref{eq:short-picard}.

\smallskip
\noindent\emph{Step 3. Smoothness and tangent spaces.}
The morphism $P\to\Pic^0_{D/k}$ is an fppf torsor under
the vector group $H^1(Z,\mathcal K)$, so it is smooth.
Its target is smooth by \eqref{eq:semiabelian-hypothesis}
and Lemma~\ref{lem:cohomology-and-picard}; hence $P$ is smooth.
The tangent space of $\Pic_{Z/k}$ is $H^1(Z,\OO_Z)$:
apply the units sequence with ideal
$\varepsilon\OO_Z$ on $Z\times\Spec k[\varepsilon]/(\varepsilon^2)$.
The tangent map of restriction is therefore
$H^1(Z,\OO_Z)\to H^1(D,\OO_D)$, which is surjective.
\end{proof}

\begin{proposition}\label{prop:first-boundary}
Let $1\leq n<p$, and let $\sigma$ generate
$\JJ|_{D_n}$ with $\sigma|_D=s$.
Set $\theta=\theta_{1,n}(\sigma)$.
For every $r\in\Z$ and
$\alpha\in H^1(\JJ^r/\JJ^{r+n})$, one has
\begin{equation}\label{eq:first-boundary}
                       \partial_{r,n}(\alpha)=r\,\alpha_0\smile\theta.
\end{equation}
\end{proposition}

\begin{proof}
\noindent\emph{Step 1. Flatness over a truncated polynomial ring.}
The inclusion $\JJ/\JJ^{n+1}\subset\OO_{D_{n+1}}$
sends $\sigma$ to a global function $t$.
It defines a proper morphism
\[
 f:D_{n+1}\longrightarrow\Spec R_n,
                \qquad R_n=k[t]/(t^{n+1}).
\]
Locally, $t$ is a unit times an equation of the Cartier
divisor $D$. Thus
\[
 \operatorname{Ann}_{\OO_{D_{n+1}}}(t^j)=(t^{n+1-j})
                  \qquad(1\leq j\leq n+1).
\]
The local flatness criterion over $R_n$ implies that $f$
is flat, and its fibre over $t=0$ is $D$.
Reduction modulo $t^j$ gives $D_j$ for $1\leq j\leq n+1$.

\smallskip
\noindent\emph{Step 2. The structure-sheaf boundary.}
The ideal $\mathcal K=\JJ/\JJ^{n+1}$ of $D$ in
$D_{n+1}$ satisfies $\mathcal K^p=0$ because $n<p$.
Lemma~\ref{lem:short-picard} gives
\[
                  H^1(D_{n+1},\OO)\twoheadrightarrow V.
\]
Apply cohomology and base change
\cite[Tags 07VJ and 0A1V]{Sta} to the proper flat
morphism $f$. The hypotheses in degrees zero and one are
\[
 H^i(D_{n+1},\OO)\otimes_{R_n}k
        \twoheadrightarrow H^i(D,\OO_D)\qquad(i=0,1):
\]
the first follows from the constant section $1$, and the
second was just proved. It follows that $H^1(D_{n+1},\OO)$
is a finite free $R_n$-module and that its formation
commutes with base change. In particular,
\[
 H^1(D_{n+1},\OO)\otimes_{R_n}R_n/(t^j)
     \xrightarrow{\sim}H^1(D_j,\OO)
                    \qquad(1\leq j\leq n+1).
\]
Taking $j=n$ gives
\begin{equation}\label{eq:structure-boundary-zero}
 H^1(D_{n+1},\OO)\twoheadrightarrow H^1(D_n,\OO),
                         \qquad \partial_{0,n}=0.
\end{equation}

\smallskip
\noindent\emph{Step 3. Tensor powers of the trivializing section.}
Set $\alpha'=\sigma^{-r}\alpha\in H^1(D_n,\OO_{D_n})$.
Then $\alpha'|_D=\alpha_0$ and
$\alpha=\alpha'\sigma^r$.
Lemma~\ref{lem:obstruction-rules}(ii) gives
\[
             s^{-r-n}\delta^0_{\JJ^r}(\sigma^r)=-r\theta.
\]
Apply Lemma~\ref{lem:boundary-product} to
$\alpha'\sigma^r$ and use \eqref{eq:structure-boundary-zero}.
After multiplying the target by $s^{-r-n}$, the result is
\[
 \partial_{r,n}(\alpha)
        =0-\alpha_0\smile(-r\theta)
        =r\,\alpha_0\smile\theta.
\]
This proves the formula for all integers $r$.
\end{proof}

\subsection{Duality and the log canonical line bundle}

At arbitrary length, the structure-sheaf boundary need
not vanish. We use the Serre duality to obtain another boundary formula.

\begin{proposition}\label{prop:canonical-boundary}
Fix an isomorphism $\MM\JJ^\beta\simeq\OO_{\YY}$ with
$\beta\in\Z$. Let $q,\ell\geq1$ and $r\in\Z$.
Let $\sigma$ generate $\JJ^q|_{D_\ell}$ with
$\sigma|_D=s^q$, and set $\theta=\theta_{q,\ell}(\sigma)$.
If $q\mid\beta+\ell+r$, then
\begin{equation}\label{eq:canonical-boundary}
 \partial_{r,\ell}(\alpha)
  =\frac{\beta+\ell+r}{q}\,\alpha_0\smile\theta
       \qquad\bigl(\alpha\in H^1(\JJ^r/\JJ^{r+\ell})\bigr).
\end{equation}
The quotient on the right is an integer before its image
in $k$ is taken.
\end{proposition}

\begin{proof}
\noindent\emph{Step 1. The section paired with the source.}
Let $u$ be the generating section of $\MM\JJ^\beta$
specified by the isomorphism, and set
\[
 b=\frac{\beta+\ell+r}{q}\in\Z,
 \qquad\tau=u\sigma^{-b}
    \in H^0(D_\ell,\MM\JJ^{-\ell-r}|_{D_\ell}).
\]
Here $\beta-bq=-\ell-r$, which verifies the line bundle
containing $\tau$.
The section $u$ is defined on the whole formal scheme,
so its restriction has zero boundary.
Lemma~\ref{lem:obstruction-rules}(ii) therefore gives
\begin{equation}\label{eq:canonical-section-boundary}
 \delta^0_{\MM\JJ^{-\ell-r}}(\tau)
                =b\,u|_D\,s^{\ell-bq}\theta.
\end{equation}
The target of this equality is $H^1(D,\MM\JJ^{-r}|_D)$.
In particular, this step uses the integer power $b$ and
does not divide by $q$ in $k$.

\smallskip
\noindent\emph{Step 2. Vanishing of the dualizing boundary.}
The comparison maps of Lemma~\ref{lem:dualizing-comparison}
give a commutative diagram with exact rows:
\begin{equation}\label{eq:canonical-diagram}
\begin{tikzcd}[column sep=small,row sep=large]
0\ar[r]&\MM|_D\ar[r]\ar[d,"\phi_1"']&
\MM\JJ^{-\ell}|_{D_{\ell+1}}\ar[r]\ar[d,"\phi_{\ell+1}"]&
\MM\JJ^{-\ell}|_{D_\ell}\ar[r]\ar[d,"\phi_\ell\otimes\JJ^{-1}"]&0\\
0\ar[r]&\omega_D\ar[r]&\omega_{D_{\ell+1}}\ar[r]&
\omega_{D_\ell}\JJ^{-1}\ar[r]&0.
\end{tikzcd}
\end{equation}
The lower row is obtained by dualizing
$0\to\JJ/\JJ^{\ell+1}\to\OO_{D_{\ell+1}}\to\OO_D\to0$.
The modules are Cohen--Macaulay, so duality in dimension two
gives the displayed exact sequence.
The upper row is the restriction sequence for
$\MM\JJ^{-\ell}$; all its vertical maps are defined
by the fractional-ideal inclusions in \eqref{eq:phi}.
Thus no choice of a local equation changes the diagram.

Under Serre duality the map
\[
 H^2(D,\omega_D)\longrightarrow
                     H^2(D_{\ell+1},\omega_{D_{\ell+1}})
\]
is dual to
$H^0(D_{\ell+1},\OO)\to H^0(D,\OO_D)=k$.
The latter is surjective because $1$ lifts. Hence the
former is injective, and exactness makes the preceding
connecting map zero.
The map $H^2(\phi_1)$ is an isomorphism by
Lemma~\ref{lem:dualizing-comparison}.
Commutativity of \eqref{eq:canonical-diagram} consequently gives
\begin{equation}\label{eq:canonical-product-boundary-zero}
 \delta^1_{\MM\JJ^{-\ell}}:
 H^1(D_\ell,\MM\JJ^{-\ell}|_{D_\ell})
                    \longrightarrow H^2(D,\MM|_D),
 \qquad\delta^1_{\MM\JJ^{-\ell}}=0.
\end{equation}

\smallskip
\noindent\emph{Step 3. Evaluation of the product rule.}
The product $\alpha\tau$ belongs to the source in
\eqref{eq:canonical-product-boundary-zero}.
Apply Lemma~\ref{lem:boundary-product} and then
\eqref{eq:canonical-section-boundary}:
\[
 0=\delta^1_{\JJ^r}(\alpha)\tau|_D
               -\alpha|_D\smile
                    \delta^0_{\MM\JJ^{-\ell-r}}(\tau).
\]
Since
\[
 \alpha|_D=s^r\alpha_0,\qquad
 \tau|_D=u|_D s^{-bq},\qquad
 bq=\beta+\ell+r,
\]
both terms contain the same generating section
$u|_D s^{-\beta}$ of $\MM|_D$.
Canceling that section gives
\[
                   0=\partial_{r,\ell}(\alpha)-b\alpha_0\smile\theta,
\]
as required. This argument uses dualizing modules and
does not assume that $\omega_{D_\ell}$ is invertible.
\end{proof}

\subsection{Lifting to length-two Witt cohomology}
\label{sec:witt-lifting}

Throughout this subsection, $G$ is cyclic, $\mu$ is faithful,
and $|G|=R<p$. Suppose that $\JJ|_{D_n}$ is trivial but
$\JJ|_{D_{n+1}}$ is not, where $1\leq n<R$.
Choose an equivariant generating section $\sigma$ on $D_n$
with $\sigma|_D=s$, and set
\[
 \theta=\theta_{1,n}(\sigma)\ne0,\qquad \chi=\mu^n,
 \qquad V_0=H^1(\JJ/\JJ^p).
\]
Lemma~\ref{lem:obstruction-rules} gives $\theta\in V_\chi$.

Our aim is to prove
\[
 \beta_{D_L}\bigl(H^1(D_L,\OO_{D_L})_\chi\bigr)=0
                         \qquad\bigl(L\leq p(p-n)\bigr),
\]
where $\beta_{D_L}$ is defined in
Section~\ref{sec:witt-preliminaries}.
This gives the Witt lift used in Appendix~\ref{app:witt}.
For $L\leq p$, it follows from the smoothness of the Picard
scheme. For larger $L$, we use two facts proved below.
First, every class in $(V_0)_\chi$ has a representative
with values in $\JJ^{p-n}/\JJ^p$
(Lemma~\ref{lem:high-order-class}). Second, the Witt boundary
is contained in the image of the degree-$p$ part of the
multiplicative restriction boundary for $D_p\subset D_L$
(Lemma~\ref{lem:bockstein-comparison}). The resulting terms
lie in $\JJ^{p(p-n)}/\JJ^L=0$, which proves the assertion
in Lemma~\ref{lem:witt-lift}.

We begin with the ordinary coherent-cohomology statement.

\begin{lemma}\label{lem:high-order-class}
One has
\begin{equation}\label{eq:high-order-class}
 (V_0)_\chi\subseteq
 \im\bigl(H^1(\JJ^{p-n}/\JJ^p)\longrightarrow V_0\bigr).
\end{equation}
\end{lemma}

\begin{proof}
For $1\leq r<p$, denote the maps induced by inclusion
and restriction by
\[
 \begin{aligned}
 \iota_r&:H^1(\JJ^r/\JJ^p)\longrightarrow V_0,\\
 \rho_r&:H^1(\JJ^r/\JJ^p)
                     \longrightarrow H^1(D,\JJ^r|_D).
 \end{aligned}
\]
Thus $\iota_r$ comes from $\JJ^r\subseteq\JJ$, whereas
$\rho_r$ comes from the quotient
$\JJ^r/\JJ^p\to\JJ^r/\JJ^{r+1}=\JJ^r|_D$.
Set $F^rV_0=\im(\iota_r)$, so $F^1V_0=V_0$.
Since $|G|=R$ is invertible in $k$, taking
$\chi$-eigenspaces is exact. Consequently
\[
 (F^rV_0)_\chi=\iota_r\bigl(H^1(\JJ^r/\JJ^p)_\chi\bigr).
\]
It therefore suffices to prove
$(F^rV_0)_\chi\subseteq F^{r+1}V_0$
for $1\leq r<p-n$.
Fix such an $r$ and a class
$a\in H^1(\JJ^r/\JJ^p)_\chi$, and set
\[
                         v=s^{-r}\rho_r(a)\in V.
\]
We will construct $b$ with the same image in $V_0$ as $a$
and with $\rho_r(b)=0$.

\noindent\emph{Step 1. Restriction to $D$.}
Let $a_n$ be the image of $a$ under the quotient map
\[
 H^1(\JJ^r/\JJ^p)\longrightarrow
                            H^1(\JJ^r/\JJ^{r+n}).
\]
The inequality $r+n+1\leq p$ shows that $a_n$ lifts to
$H^1(\JJ^r/\JJ^{r+n+1})$. Its connecting image is
therefore zero. Proposition~\ref{prop:first-boundary} gives
\[
                   0=\partial_{r,n}(a_n)=r\,v\smile\theta.
\]
Since $1\leq r<p$, the coefficient $r$ is nonzero in $k$.
By \eqref{eq:perfect-cup}, the perfect alternating pairing
on the two-dimensional space $V$ has
\[
 \ker\bigl(V\xrightarrow{\,w\mapsto w\smile\theta\,}
                              H^2(D,\OO_D)\bigr)=k\theta.
\]
Hence $v=\lambda\theta$ for some $\lambda\in k$.

Suppose $v\ne0$. Equivariance of $\rho_r$ and the character
$\mu^{-1}$ of $s$ give
\[
 g(v)=\chi(g)\mu(g)^r v\qquad(g\in G).
\]
On the other hand, $v=\lambda\theta$ with
$\lambda\ne0$ and $\theta\in V_\chi$, so
$g(v)=\chi(g)v$. Thus $\mu^r=1$, or equivalently $R\mid r$.
Since $r>0$ and $n<R$, it follows that
\[
                  r\geq R>n,\qquad 1\leq r-n<p.
\]
In particular, $r-n$ is invertible in $k$.

\smallskip
\noindent\emph{Step 2. A connecting class in $\ker(\iota_r)$.}
Assume $v\ne0$, and set $d=r-n$.
Because $d+n=r$, the section $\sigma^d$ belongs to
\[
 H^0(D_n,\JJ^d|_{D_n})=H^0(\JJ^d/\JJ^r).
\]
Let $\delta$ be the connecting homomorphism of
\[
 0\longrightarrow\JJ^r/\JJ^p
 \longrightarrow\JJ^d/\JJ^p
 \longrightarrow\JJ^d/\JJ^r\longrightarrow0,
\]
and define
\[
                         \gamma=\delta(\sigma^d)
                                     \in H^1(\JJ^r/\JJ^p).
\]
Its two required properties follow from the following maps.

Since $d\geq1$, the inclusion $\JJ^d\subseteq\JJ$
gives a factorization of $\iota_r$:
\[
 H^1(\JJ^r/\JJ^p)\longrightarrow
                  H^1(\JJ^d/\JJ^p)\longrightarrow V_0.
\]
Exactness of the preceding cohomology sequence sends
$\gamma$ to zero under the first arrow. Hence
\[
                              \iota_r(\gamma)=0.
\]
Moreover, $R\mid r$ implies
$\mu^{-d}=\mu^{n-r}=\mu^n=\chi$.
Thus $\sigma^d$ has character $\chi$, and equivariance
of $\delta$ gives $\gamma\in H^1(\JJ^r/\JJ^p)_\chi$.

To compute $\rho_r(\gamma)$, use $r+1\leq p$ to quotient
the preceding short exact sequence by $\JJ^{r+1}$ in its
first two terms.
This gives the restriction sequence
\[
 0\longrightarrow\JJ^r/\JJ^{r+1}
 \longrightarrow\JJ^d/\JJ^{r+1}
 \longrightarrow\JJ^d/\JJ^r\longrightarrow0
\]
for $\JJ^d$ from $D_{n+1}$ to $D_n$.
Naturality of connecting homomorphisms gives
\[
 \begin{tikzcd}[column sep=large,row sep=large]
 H^0(\JJ^d/\JJ^r)\ar[r,"\delta"]\ar[d,equal]
   &H^1(\JJ^r/\JJ^p)\ar[d,"\rho_r"]\\
 H^0(D_n,\JJ^d|_{D_n})\ar[r,"\delta^0_{\JJ^d}"']
   &H^1(D,\JJ^r|_D).
 \end{tikzcd}
\]
Applying Lemma~\ref{lem:obstruction-rules}(ii) to
$\sigma^d$ therefore gives
\[
 s^{-r}\rho_r(\gamma)
      =s^{-r}\delta^0_{\JJ^d}(\sigma^d)=-d\theta.
\]

\smallskip
\noindent\emph{Step 3. Lifting from the next ideal quotient.}
If $v=0$, set $b=a$. Otherwise use the class $\gamma$
of Step~2 and set $b=a+(\lambda/d)\gamma$.
The coefficient in the second case is defined because
$1\leq d<p$.
In both cases,
\[
 b\in H^1(\JJ^r/\JJ^p)_\chi,\qquad
 \iota_r(b)=\iota_r(a),\qquad \rho_r(b)=0;
\]
in the second case the last equality follows explicitly from
\[
                 s^{-r}\rho_r(b)
                    =\lambda\theta+\frac{\lambda}{d}(-d\theta)=0.
\]
The restriction sequence
\[
 0\longrightarrow\JJ^{r+1}/\JJ^p
 \longrightarrow\JJ^r/\JJ^p
 \longrightarrow\JJ^r/\JJ^{r+1}\longrightarrow0
\]
and exactness of taking $\chi$-eigenspaces give an exact sequence
\[
 H^1(\JJ^{r+1}/\JJ^p)_\chi\longrightarrow
 H^1(\JJ^r/\JJ^p)_\chi
       \xrightarrow{\rho_r}H^1(D,\JJ^r|_D)_\chi.
\]
Thus $b$ lifts to $a'\in H^1(\JJ^{r+1}/\JJ^p)_\chi$, and
\[
              \iota_r(a)=\iota_r(b)=\iota_{r+1}(a')\in F^{r+1}V_0.
\]
This proves $(F^rV_0)_\chi\subseteq F^{r+1}V_0$.
Starting from $F^1V_0=V_0$ and repeating for
$r=1,\ldots,p-n-1$ gives
$(V_0)_\chi\subseteq F^{p-n}V_0$, which is
\eqref{eq:high-order-class}.
\end{proof}

We now compare the Witt boundary with the obstruction to
extending a line bundle from $D_p$ to $D_L$.
For $p<L<p^2$, write
\[
 \mathcal K_L=\JJ^p/\JJ^L,\qquad
 W_L=H^2(\mathcal K_L),\qquad
 j_L:W_L\longrightarrow H^2(D_L,\OO_{D_L}).
\]
The map $j_L$ is induced by $\mathcal K_L\subset\OO_{D_L}$.
Let $P$ be the inverse image of $\Pic^0_{D/k}$ in
$\Pic_{D_p/k}$. Lemma~\ref{lem:short-picard} gives
\begin{equation}\label{eq:picard-vector-extension}
 0\longrightarrow V_0\xrightarrow{\iota}P
 \longrightarrow\Pic^0_{D/k}\longrightarrow0,
 \qquad P\text{ smooth}.
\end{equation}
Since $\mathcal K_L^p=0$, the units sequence
\cite[Tag 0C6Q]{Sta} and the truncated logarithm define
\begin{equation}\label{eq:long-picard-connecting}
 \delta_L:P\longrightarrow W_L.
\end{equation}
Explicitly, after any base change by a $k$-algebra $B$, this
is the connecting map
\[
 \Pic((D_p)_B)\longrightarrow
 H^2((D_L)_B,1+(\mathcal K_L)_B)
 \xrightarrow{\log}W_L\otimes_k B,
\]
followed by fppf sheafification. We regard a finite-dimensional
$k$-vector space as its additive vector group in these sequences.

\begin{lemma}
\label{lem:bockstein-comparison}
There are a $k$-linear map $\lambda_L$ and an additive
$p$-semilinear map $C_L$ such that
\begin{equation}\label{eq:linear-semilinear}
 \delta_L\circ\iota=\lambda_L+C_L:V_0\longrightarrow W_L.
\end{equation}
The map $\lambda_L$ is the coherent-cohomology boundary for
\begin{equation}\label{eq:linear-part-sequence}
 0\longrightarrow\JJ^p/\JJ^L
 \longrightarrow\JJ/\JJ^L
 \longrightarrow\JJ/\JJ^p\longrightarrow0.
\end{equation}
For every character $\psi$ of $G$, one has
\begin{equation}\label{eq:bockstein-image-bound}
 \beta_{D_L}\bigl(H^1(D_L,\OO_{D_L})_\psi\bigr)
       \subseteq j_L\bigl(C_L((V_0)_\psi)\bigr).
\end{equation}
\end{lemma}

\begin{proof}
\noindent\emph{Step 1. The linear and $p$-semilinear parts.}
Set $E(T)=\sum_{j=0}^{p-1}T^j/j!$.
Represent a basis $e_1,\ldots,e_d$ of $V_0$ by cocycles
$c^{(i)}$ in $\JJ/\JJ^p$, and lift their entries to
$\JJ/\JJ^L$. With independent variables $T_i$, set
$x_{ab}(T)=\sum_iT_i\widetilde c^{(i)}_{ab}$.
Then
\begin{equation}\label{eq:picard-polynomial-cocycle}
 (\delta_L\circ\iota)\Bigl(\sum_iT_i e_i\Bigr)
 =\left[\log\left(E(x_{ab})E(x_{bc})E(x_{ac})^{-1}\right)\right].
\end{equation}
The expression inside the logarithm is $1$ modulo
$\mathcal K_L$. A coefficient of parameter degree $j$
belongs to $\JJ^j/\JJ^L$, so every degree $j\geq L$ is zero.
Thus \eqref{eq:picard-polynomial-cocycle} is an additive
polynomial of degree $<L<p^2$. Its only possible monomials
are $T_i$ and $T_i^p$, proving \eqref{eq:linear-semilinear}.
Its linear term is $dx(T)$, which gives precisely
\eqref{eq:linear-part-sequence}. In particular,
\begin{equation}\label{eq:linear-boundary-killed}
                              j_L\lambda_L=0.
\end{equation}

\smallskip
\noindent\emph{Step 2. The image of the multiplicative connecting map.}
Since $k$ is perfect, $j_L(C_L(V_0))$ is a $k$-linear
subspace of $H^2(D_L,\OO_{D_L})$.
By \eqref{eq:linear-semilinear} and
\eqref{eq:linear-boundary-killed}, the homomorphism
\[
 P\xrightarrow{\delta_L}W_L\xrightarrow{j_L}
 H^2(D_L,\OO_{D_L})\longrightarrow
 \frac{H^2(D_L,\OO_{D_L})}{j_L(C_L(V_0))}
\]
vanishes on $V_0$. It therefore factors through the
semiabelian quotient $P/V_0=\Pic^0_{D/k}$.
A semiabelian variety has no nonzero homomorphism to a
vector group; see Lemma~\ref{lem:short-picard}. Hence
\begin{equation}\label{eq:picard-image-linear-span}
                     j_L\delta_L(P)\subseteq j_L(C_L(V_0))
\end{equation}
as group sheaves. In particular, this containment holds
after base change by a nonreduced $k$-algebra.

\smallskip
\noindent\emph{Step 3. Comparison with the Witt boundary.}
Let $z$ be an additive cocycle on $D_L$, and let $\bar z$
be its restriction to $D_p$.
Since $P$ is smooth, \eqref{eq:smooth-picard-bockstein}
gives a Witt cocycle $(\bar z,b)$ on $D_p$.
Choose lifts $\widetilde b_{ab}$ to $D_L$, and set
\[
 c_{abc}=B_p(z_{ab},z_{bc})+(d\widetilde b)_{abc}.
\]
By \eqref{eq:bockstein-cocycle}, $c$ is an additive
$2$-cocycle in $\mathcal K_L$ and
\begin{equation}\label{eq:bockstein-comparison-class}
                  j_L([c])=\beta_{D_L}([z]).
\end{equation}
We will realize $\beta_{D_L}([z])$ as a coefficient of
$j_L\delta_L$ and apply \eqref{eq:picard-image-linear-span}.

Use the auxiliary ring $B=k[\tau]/(\tau^{p+1})$.
The identity
\begin{equation}\label{eq:exponential-witt-comparison}
 \frac{E(\tau x)E(\tau y)}{E(\tau(x+y))}
                 =1+\tau^p B_p(x,y)\pmod{\tau^{p+1}}
\end{equation}
follows by comparing the degree-$p$ coefficients:
$1/(j!(p-j)!)=-\binom pj/p$ in $k$, since $(p-1)!=-1$.
Consequently
\[
 u_{ab}=E(\tau\bar z_{ab})(1+\tau^p b_{ab})
\]
are transition functions of a line bundle $\mathcal N$
on $(D_p)_B$. Its restriction at $\tau=0$ is trivial,
so $[\mathcal N]\in P(B)$.
Lifting these functions by
$E(\tau z_{ab})(1+\tau^p\widetilde b_{ab})$ shows that
\[
                  \delta_L([\mathcal N])=\tau^p[c].
\]
Apply $j_L$ and use \eqref{eq:bockstein-comparison-class}:
\[
 j_L\delta_L([\mathcal N])=\tau^p\beta_{D_L}([z])
                           \in j_L(C_L(V_0))\otimes_k B.
\]
Taking the coefficient of $\tau^p$ gives
\[
 \beta_{D_L}(H^1(D_L,\OO_{D_L}))\subseteq j_L(C_L(V_0)).
\]
Finally, all the maps commute with $G$, and
\[
 C_L((V_0)_\rho)\subseteq (W_L)_{\rho^p},\qquad
 \beta_{D_L}(H^1(\OO_{D_L})_\psi)
                         \subseteq H^2(\OO_{D_L})_{\psi^p}.
\]
The map $\rho\mapsto\rho^p$ permutes the characters because
$(p,R)=1$. Taking the $\psi^p$-component proves
\eqref{eq:bockstein-image-bound}.
\end{proof}

Lemma \ref{lem:bockstein-comparison} has removed the linear contribution to the
Witt boundary. It remains to use the vanishing order supplied
by Lemma~\ref{lem:high-order-class} to eliminate $C_L$.

\Needspace{12\baselineskip}
\begin{lemma}\label{lem:witt-lift}
For $1\leq L\leq L_0:=p(p-n)$,
\[
 H^1(D_L,\OO_{D_L})_\chi\subseteq
 \im\bigl(H^1(D_L,W_2\OO_{D_L})
             \xrightarrow{\operatorname{pr}_1}H^1(D_L,\OO_{D_L})\bigr).
\]
On any fixed affine covering of $D_L$, any additive cocycle
representing such a class can be prescribed as the first
coordinate of a Witt cocycle on the same covering.
\end{lemma}

\begin{proof}
For $L\leq p$, Lemma~\ref{lem:short-picard} and
\eqref{eq:smooth-picard-bockstein} give $\beta_{D_L}=0$.
Suppose $p<L\leq L_0<p^2$.
By Lemma~\ref{lem:high-order-class}, representatives of a
basis of $(V_0)_\chi$ can be chosen in $\JJ^{p-n}/\JJ^p$.
Lift them to $\JJ^{p-n}/\JJ^L$ in
\eqref{eq:picard-polynomial-cocycle}. Every parameter term
of degree $p$ then belongs to $\JJ^{p(p-n)}/\JJ^L=0$.
Therefore
\[
 C_L((V_0)_\chi)=0,\qquad
 \beta_{D_L}(H^1(D_L,\OO_{D_L})_\chi)=0
\]
by Lemma~\ref{lem:bockstein-comparison}.
The cohomology sequence of \eqref{eq:witt-short-exact}
proves the asserted lifting. The adjustment
$\zeta'\mapsto\zeta'-d[a]$ following
\eqref{eq:bockstein-cocycle} preserves any specified
additive representative.
\end{proof}

\section{Extension of line bundles to finite thickenings}\label{sec:finite}

We now apply the two boundary formulas to powers of $\JJ$.
The aim is to find an integer $N>0$ for which
$\JJ^N|_{D_N}$ is trivial. 

\begin{notation}\label{not:finite-calculation}
In this section $\YY$ is a normal formal threefold, $A$ is
its reduced boundary, and
\[
 \JJ=\OO_{\YY}(-cA),\qquad D_\ell=V(\JJ^\ell),\qquad
 D=D_1,\qquad \MM=\omega_{\YY}(A).
\]
Assume we are in Setting~\ref{not:equivariant}. In particular
$\JJ|_D=\OO_Ds$, $\Pic^0_{D/k}$ is smooth semiabelian, and
$V=H^1(D,\OO_D)$ has dimension two with perfect alternating
cup product. The finite abelian group $G$ acts faithfully
on $V$. The character of $s$ is $\mu^{-1}$, and
$R=\ord(\mu)<p$.
\end{notation}

The main theorem of this section is the following.

\begin{theorem}\label{thm:finite-boundary}
Assume that $G$ is cyclic, $\mu$ is faithful, and
$\MM\JJ^\beta$ has a $G$-invariant generating section for
an integer $\beta\geq1$.
If $e\geq1$ satisfies
\begin{equation}\label{eq:finite-bound}
                         (p-R)p^{e-1}\geq\beta,
\end{equation}
then for some $0\leq i\leq e$ there is a $G$-invariant
isomorphism
\begin{equation}\label{eq:finite-conclusion}
                      \JJ^{Rp^i}|_{D_{Rp^i}}\simeq\OO_{D_{Rp^i}}.
\end{equation}
\end{theorem}

Suppose that \eqref{eq:finite-conclusion} fails for every
$0\leq i\leq e$. For each $i$, let $\ell_i$ be the largest integer
such that $\JJ^{p^i}|_{D_{\ell_i}}$ is trivial.
In \S5.1 we derive two divisibility statements from a change
of generating section and from the canonical boundary formula.
In \S5.2 they imply
\[
 p^i\mid\beta+\ell_{i-1}+\ell_i,\qquad
                 \theta_{i-1}\smile\theta_i\ne0
                                    \qquad(1\leq i\leq e),
\]
where $\theta_i$ is the obstruction to extending a generator
from $D_{\ell_i}$ to $D_{\ell_i+1}$.
In \S5.3 a class supplied by Appendix~\ref{app:witt} allows
us to evaluate the same connecting homomorphism in two ways.
The resulting congruence contradicts \eqref{eq:finite-bound}.

\subsection{Changes of trivialization}

We use the class $\theta_{q,\ell}(\sigma)$ defined in
\eqref{eq:theta-definition}. It is the obstruction to extending
the specified generator $\sigma$, rather than an invariant
of the line bundle alone. Nontriviality on $D_{\ell+1}$
means that this class is nonzero for every generator on $D_\ell$.

Whenever $\JJ^q|_{D_t}$ is trivial, its generator can be
chosen to restrict to $s^q$ and to have character $\mu^{-q}$.
Indeed, if $\tau|_D=s^q$, then
\begin{equation}\label{eq:generator-character-projection}
 \tau_{\mu^{-q}}=
 \frac{1}{|G|}\sum_{g\in G}\mu(g)^qg(\tau),
 \qquad \tau_{\mu^{-q}}|_D=s^q.
\end{equation}
The last equality implies that $\tau_{\mu^{-q}}$ generates,
by Nakayama's lemma. We will also use the consequence
\[
 u\smile v\ne0\quad\Longleftrightarrow\quad ku\ne kv
                       \qquad(0\ne u,v\in V)
\]
of the perfect alternating pairing on the two-dimensional
space $V$.

\begin{lemma}\label{lem:change-of-trivialization}
Let $q,\ell\geq1$, and let $\sigma$ be a generating section
of $\JJ^q|_{D_\ell}$ with $\sigma|_D=s^q$.
Set $\theta=\theta_{q,\ell}(\sigma)\ne0$.
Let $\mathcal H$ be an invertible sheaf on $\YY$, trivial
on $D_r$ but nontrivial on $D_{r+1}$, where $r>\ell$.
Choose a generating section $\tau$ of $\mathcal H|_{D_r}$.
Denote by $\delta^0_{\mathcal H^{-1}}$ the connecting map
\[
 H^0(D_r,\mathcal H^{-1}|_{D_r})
 \longrightarrow
 H^1(D,\mathcal H^{-1}|_D\otimes\JJ^r|_D)
\]
for restriction from $D_{r+1}$ to $D_r$, and set
\[
 \eta=(\tau|_D)s^{-r}
                     \delta^0_{\mathcal H^{-1}}(\tau^{-1})\in V.
\]
Then $\eta\ne0$, and there exists a class satisfying
\begin{equation}\label{eq:lifted-transition-class}
 \widetilde\eta\in H^1(\JJ^r/\JJ^{r+\ell}),\qquad
 \widetilde\eta_0=\eta,\qquad
                  \partial_{r,\ell}(\widetilde\eta)=0.
\end{equation}
Moreover:
\begin{enumerate}
\item If $q\mid r-\ell$ and $\eta\in k\theta$, then
      $pq\mid r-\ell$.
\item Suppose also that $\MM\JJ^\beta\simeq\OO_{\YY}$.
      If $q\mid\beta+r+\ell$ and
      $\eta\smile\theta\ne0$, then
      $pq\mid\beta+r+\ell$.
\end{enumerate}
\end{lemma}

\begin{proof}
\noindent\emph{Step 1. A class with zero connecting image.}
Set
\[
 Z=D_{r+\ell+1},\qquad
 \mathcal I=\JJ^r/\JJ^{r+\ell+1}\subset\OO_Z.
\]
Since $r>\ell$, one has $2r\geq r+\ell+1$, and hence
$\mathcal I^2=0$. Thus
\[
 (1+\mathcal I,\cdot)\xrightarrow{\sim}(\mathcal I,+),
                              \qquad 1+x\longmapsto x
\]
is an isomorphism of sheaves of abelian groups.
Choose an affine covering $(U_a)$ of $Z$ and local
generators $v_a$ of $\mathcal H^{-1}|_Z$ lifting
$\tau^{-1}|_{U_a\cap D_r}$. On $U_a\cap U_b$, their
transition function $v_b/v_a$ lies in $1+\mathcal I$.
The displayed isomorphism gives
\[
 \widehat\eta=[(v_b/v_a)-1]
                  \in H^1(Z,\mathcal I)
                    =H^1(\JJ^r/\JJ^{r+\ell+1}).
\]
Modulo $\JJ^{r+1}$, the identity
$v_b-v_a=v_a((v_b/v_a)-1)$ identifies this class with
$s^r\eta$ under the chosen trivialization of
$\mathcal H^{-1}|_D$.
Let $\widetilde\eta$ be its image in
$H^1(\JJ^r/\JJ^{r+\ell})$.
It satisfies \eqref{eq:lifted-transition-class}: the first
equality of classes follows by restriction to $D$, and
the vanishing follows because $\widehat\eta$ is a lift.
Finally $\eta=0$ would extend $\tau^{-1}$ to a generating
section on $D_{r+1}$, contradicting the hypothesis on
$\mathcal H$.

\smallskip
\noindent\emph{Step 2. The case $\eta\in k\theta$.}
Assume $q\mid r-\ell$, and set $d=(r-\ell)/q$.
The natural inclusion
\[
 \JJ^{r-\ell}/\JJ^r\hookrightarrow\OO_{D_r}
\]
makes $\sigma^d$ a global nilpotent function on $D_r$.
Let $\delta^0_{\OO}$ denote the boundary for
$\OO_{D_{r+1}}\to\OO_{D_r}$.
Naturality of the restriction sequence and
Lemma~\ref{lem:obstruction-rules}(ii) give
\[
                  s^{-r}\delta^0_{\OO}(\sigma^d)=-d\theta.
\]
For $b\in k$, set $\tau_b=(1+b\sigma^d)\tau$.
Since $(1+b\sigma^d)|_D=1$, the degree-zero product
rule yields the explicit change of obstruction
\begin{equation}\label{eq:scalar-unit-change}
 \begin{aligned}
 (\tau|_D)s^{-r}\delta^0_{\mathcal H^{-1}}(\tau_b^{-1})
   &=\eta-s^{-r}\delta^0_{\OO}(1+b\sigma^d)\\
   &=\eta+bd\theta.
 \end{aligned}
\end{equation}
Write $\eta=\lambda\theta$, where $\lambda\in k^*$.
If $p\nmid d$, take $b=-\lambda/d$.
Then \eqref{eq:scalar-unit-change} vanishes, so
$\tau_b^{-1}$ extends to a generating section on $D_{r+1}$.
This contradicts $\mathcal H|_{D_{r+1}}\not\simeq\OO_{D_{r+1}}$.
Consequently $p\mid d$, proving (i).

\smallskip
\noindent\emph{Step 3. The case $\eta\smile\theta\ne0$.}
Under the hypotheses of (ii), apply
Proposition~\ref{prop:canonical-boundary} to the class
$\widetilde\eta$ of Step~1. It gives
\[
 0=\partial_{r,\ell}(\widetilde\eta)
           =\frac{\beta+r+\ell}{q}\,\eta\smile\theta.
\]
The coefficient is an integer and the cup product is
nonzero. Its image in $k$ is therefore zero, which is
equivalent to $pq\mid\beta+r+\ell$.
\end{proof}

For the next corollary the short-length formula suffices:
we apply it to the class in \eqref{eq:lifted-transition-class}.
No trivialization of $\MM\JJ^\beta$ is needed.

\begin{corollary}\label{cor:one-direction}
In Setting~\ref{not:finite-calculation}, suppose that
\[
              \dim_k\Bigl(\bigoplus_{j=1}^{R-1}V_{\mu^j}\Bigr)\leq1.
\]
Then \eqref{eq:finite-conclusion} holds for $i=0$ or $1$.
No hypothesis on the line bundle $\MM\JJ^\beta$ is required.
The group $G$ need not be cyclic.
\end{corollary}

\begin{proof}
If $R=1$, the invariant generator $s$ proves the assertion
with $i=0$. Assume $R>1$, and set
\[
                         W=\bigoplus_{j=1}^{R-1}V_{\mu^j}.
\]

\noindent\emph{Step 1. The first obstruction for $\JJ$.}
If $\JJ|_{D_R}$ is trivial, proceed to Step~3 with $i=0$.
Otherwise define
\[
 n=\max\{t\geq1:\JJ|_{D_t}\simeq\OO_{D_t}\},
                         \qquad 1\leq n<R.
\]
Choose a generator $\sigma$ of $\JJ|_{D_n}$ with
$\sigma|_D=s$ and character $\mu^{-1}$, using
\eqref{eq:generator-character-projection}. Then
\[
 0\ne\theta=\theta_{1,n}(\sigma)\in V_{\mu^n}\subset W.
\]
Since $0<n<R$ and $\dim W\leq1$, one obtains
\[
                         \mu^n\ne1,\qquad W=k\theta.
\]
For later use, taking $p$-th powers of local generators gives
\begin{equation}\label{eq:power-triviality}
 \JJ^q|_{D_t}\simeq\OO_{D_t}
       \quad\Longrightarrow\quad
 \JJ^{pq}|_{D_{pt}}\simeq\OO_{D_{pt}}.
\end{equation}
Indeed, the transition functions can be chosen in
$1+\JJ^t$, and $(1+\JJ^t)^p\subseteq1+\JJ^{pt}$.
In particular, $\JJ^p|_{D_{pn}}$ is trivial.

\smallskip
\noindent\emph{Step 2. Triviality of $\JJ^p$ on $D_{Rp}$.}
Suppose, to the contrary, that this restriction is nontrivial.
Let
\[
 r=\max\{t\geq1:\JJ^p|_{D_t}\simeq\OO_{D_t}\},
                          \qquad pn\leq r<Rp.
\]
Choose a generator $\tau$ of $\JJ^p|_{D_r}$ of
character $\mu^{-p}$, normalized by $\tau|_D=s^p$, and set
\[
                    0\ne\eta=\theta_{p,r}(\tau)
                                  \in V_{\mu^r}.
\]
Use Lemma~\ref{lem:change-of-trivialization} with the
sections $\sigma,\tau$, the line bundle $\mathcal H=\JJ^p$,
and parameters $q=1$ and $\ell=n$.
Its length inequality holds because $r\geq pn>n$.

If $\mu^r\ne1$, then $\eta\in W=k\theta$.
Equality of the two characters and part~(i) of the lemma give
\[
 R\mid r-n,\qquad p\mid r-n.
\]
As $(R,p)=1$, these imply $Rp\mid r-n$, contrary to
$0<r-n<Rp$.

If $\mu^r=1$, then $\eta\in V^G$, whereas
$\theta\in V_{\mu^n}$ and $\mu^n\ne1$.
Thus $\eta\smile\theta\ne0$.
The class in \eqref{eq:lifted-transition-class} satisfies
\[
 \widetilde\eta\in H^1(\JJ^r/\JJ^{r+n}),\qquad
 \widetilde\eta_0=\eta,\qquad
                       \partial_{r,n}(\widetilde\eta)=0.
\]
Since $n<R<p$, Proposition~\ref{prop:first-boundary} applies:
\[
                  0=\partial_{r,n}(\widetilde\eta)
                                    =r\,\eta\smile\theta.
\]
Hence $p\mid r$. The equality $\mu^r=1$ also gives
$R\mid r$, so $Rp\mid r$, contrary to $0<r<Rp$.
Both cases give a contradiction. Therefore
$\JJ^p|_{D_{Rp}}\simeq\OO_{D_{Rp}}$.

\smallskip
\noindent\emph{Step 3. An invariant trivialization.}
For the index $i\in\{0,1\}$ obtained above, choose a
generator $\gamma$ of $\JJ^{p^i}|_{D_{Rp^i}}$ with
character $\mu^{-p^i}$ by
\eqref{eq:generator-character-projection}.
Then $\gamma^R$ generates $\JJ^{Rp^i}|_{D_{Rp^i}}$ and
\[
                 g(\gamma^R)=\mu(g)^{-Rp^i}\gamma^R
                                  =\gamma^R\qquad(g\in G).
\]
This proves \eqref{eq:finite-conclusion}.
\end{proof}

\subsection{The first lengths of nontriviality}

We now assume the hypotheses of Theorem~\ref{thm:finite-boundary}
and suppose that \eqref{eq:finite-conclusion} fails for
every $0\leq i\leq e$.
To remove the factor $R$ from the exponent of $\JJ$, set
\[
 K_t=\ker\bigl(\Pic(D_t)\longrightarrow\Pic(D)\bigr).
\]
If $p^h\geq t$, the ideal $\mathcal I_t=\JJ/\JJ^t$
of $D$ in $D_t$ satisfies $\mathcal I_t^{p^h}=0$.
Consequently $(1+\mathcal I_t)^{p^h}=1$, and the units
sequence gives $p^hK_t=0$.
Since $\JJ|_D$ is trivial and $(R,p)=1$, it follows that
\[
 \JJ^{p^i}|_{D_t}\simeq\OO_{D_t}
 \quad\Longleftrightarrow\quad
 \JJ^{Rp^i}|_{D_t}\simeq\OO_{D_t}.
\]
In the right-hand condition a generating section can be
chosen invariant by \eqref{eq:generator-character-projection},
because $\mu^{Rp^i}=1$.
Thus our assumption is equivalent to
\[
 \JJ^{p^i}|_{D_{Rp^i}}\not\simeq\OO_{D_{Rp^i}}
                                      \qquad(0\leq i\leq e).
\]
Define
\begin{equation}\label{eq:orders-definition}
 \ell_i=\max\{\ell:\JJ^{p^i}|_{D_\ell}\simeq\OO_{D_\ell}\},
 \qquad n=\ell_0,\qquad 0\leq i\leq e.
\end{equation}
Choose an equivariant generating section $\sigma_i$ of
$\JJ^{p^i}|_{D_{\ell_i}}$, normalized by
$\sigma_i|_D=s^{p^i}$, and set
\[
                         \theta_i=\theta_{p^i,\ell_i}(\sigma_i).
\]
The maximality of $\ell_i$ implies $\theta_i\ne0$ for
every choice of $\sigma_i$.
Equations~\eqref{eq:power-triviality} and
\eqref{eq:orders-definition}, together with
Lemma~\ref{lem:obstruction-rules}(iii), give
\begin{equation}\label{eq:order-bounds}
 \begin{gathered}
 1\leq n<R<p,\qquad p\ell_{i-1}\leq\ell_i<Rp^i
                                      \quad(1\leq i\leq e),\\
 0\ne\theta_i\in V_{\mu^{\ell_i}}
                                      \quad(0\leq i\leq e).
 \end{gathered}
\end{equation}
Since $e\geq1$, Corollary~\ref{cor:one-direction} shows that
\[
                         V=\bigoplus_{j=1}^{R-1}V_{\mu^j}.
\]
In particular, every character occurring in $V$ is a
nontrivial power of $\mu$.

\begin{lemma}\label{lem:order-congruences}
There are one-dimensional subspaces $L_0,L_1\subset V$
and choices of $\sigma_0,\ldots,\sigma_e$ such that
\begin{equation}\label{eq:obstruction-parity-lines}
 V=L_0\oplus L_1,\qquad
 \theta_i\in L_{i\bmod 2}\setminus\{0\}
                                      \qquad(0\leq i\leq e).
\end{equation}
For these choices one has
\begin{align}
 p&\mid\beta+n,\label{eq:first-congruence}\\
 p^i&\mid\beta+\ell_{i-1}+\ell_i
                    \qquad(1\leq i\leq e).\label{eq:adjacent-congruence}
\end{align}
For every $0\leq i<j\leq e$, one also has
\begin{equation}\label{eq:later-congruence}
 \ell_j\equiv
 \begin{cases}
 \ell_i&\text{if }\theta_j\in k\theta_i,\\
 -\beta-\ell_i&\text{if }\theta_j\notin k\theta_i
 \end{cases}
                         \pmod {p^{i+1}}.
\end{equation}
\end{lemma}

\begin{proof}
\noindent\emph{Step 1. The congruences modulo $p$.}
The ideal $\JJ/\JJ^n$ of $D$ in $D_n$ has $p$-th
power zero, since $n<p$.
Lemma~\ref{lem:short-picard} therefore gives
\[
                 H^1(D_n,\OO_{D_n})\twoheadrightarrow V.
\]
Choose $\alpha$ in this group with
$\alpha|_D\smile\theta_0\ne0$.
Apply Propositions~\ref{prop:first-boundary} and
\ref{prop:canonical-boundary} with $r=0$, $q=1$ and
$\ell=n$. They give
\[
 0=\partial_{0,n}(\alpha)
                    =(\beta+n)(\alpha|_D)\smile\theta_0.
\]
Thus $p\mid\beta+n$.

For $j\geq1$, use Lemma~\ref{lem:change-of-trivialization}
with $\sigma=\sigma_0$, $q=1$, $\ell=n$, $r=\ell_j$,
$\mathcal H=\JJ^{p^j}$ and $\tau=\sigma_j$.
The class denoted by $\eta$ there is precisely $\theta_j$.
Its two conclusions are
\[
 \begin{cases}
 p\mid\ell_j-n,&\theta_j\in k\theta_0,\\
 p\mid\beta+n+\ell_j,&\theta_j\notin k\theta_0.
 \end{cases}
\]
This is \eqref{eq:later-congruence} for $i=0$.

\smallskip
\noindent\emph{Step 2. Choice of the two lines.}
First, $\theta_1\notin k\theta_0$.
Indeed, proportionality would imply both
$p\mid\ell_1-n$ by Step~1 and $R\mid\ell_1-n$ by
$\theta_j\in V_{\mu^{\ell_j}}$.
It would follow that
\[
                  Rp\mid\ell_1-n,\qquad 0<\ell_1-n<Rp.
\]
This is impossible. Set
\[
                      L_0=k\theta_0,\qquad L_1=k\theta_1.
\]
These are $G$-stable lines and $V=L_0\oplus L_1$.
If their characters differ, every eigenvector $\theta_j$
already belongs to $L_0\cup L_1$.

Suppose their characters are equal. The action of $G$ on
$V$ is then scalar. Fix $j\geq2$ with $\theta_j\notin L_0$,
and write
\[
 \theta_j=a_j\theta_0+b_j\theta_1,\qquad b_j\ne0,
                         \qquad d_j=\ell_j-n.
\]
Step~1 and $p\mid\beta+n$ give $p\mid\ell_j$.
Hence $p\nmid d_j$, since $0<n<p$.
Equality of characters also gives $R\mid d_j$.
Define
\[
 u_j=1-\frac{a_j}{d_j}\sigma_0^{d_j}
                   \in H^0(D_{\ell_j},\OO_{D_{\ell_j}}^*),
                         \qquad \sigma'_j=u_j\sigma_j.
\]
Here $\sigma_0^{d_j}$ is viewed as a function through
\[
 \JJ^{d_j}/\JJ^{d_j+n}
       =\JJ^{d_j}/\JJ^{\ell_j}\hookrightarrow\OO_{D_{\ell_j}}.
\]
It is nilpotent and has character $\mu^{-d_j}=1$.
Thus $u_j$ is invariant, and $\sigma'_j$ has the same
character and the same restriction to $D$ as $\sigma_j$.
Formula~\eqref{eq:scalar-unit-change} gives
\[
 \theta_{p^j,\ell_j}(\sigma'_j)
          =\theta_j-a_j\theta_0=b_j\theta_1\in L_1\setminus\{0\}.
\]
Make this replacement for each such $j$, retaining all
generators whose classes belong to $L_0$.
Every $\theta_j$ now lies in $(L_0\cup L_1)\setminus\{0\}$.
The conclusions of Step~1 remain valid, since their proof
applies to every choice of the generating sections.

\smallskip
\noindent\emph{Step 3. Induction on the power of $p$.}
Let $1\leq i\leq e$, and assume that
\eqref{eq:later-congruence} has been proved with every first
index less than $i$.
If $k\theta_i=k\theta_{i-1}$, its instance with indices
$(i-1,i)$ gives $p^i\mid\ell_i-\ell_{i-1}$.
Equality of characters gives $R\mid\ell_i-\ell_{i-1}$.
Consequently
\[
        Rp^i\mid\ell_i-\ell_{i-1},\qquad
                     0<\ell_i-\ell_{i-1}<Rp^i,
\]
which is impossible. Therefore
\[
 k\theta_i\ne k\theta_{i-1},\qquad
                         p^i\mid\beta+\ell_{i-1}+\ell_i.
\]
The second assertion is the other case of
\eqref{eq:later-congruence} for $(i-1,i)$.
The first assertion, for successive $i$, proves
\eqref{eq:obstruction-parity-lines}.

Fix $j>i$. There are exactly two possible lines for
$\theta_j$, so the induction hypothesis gives the
following two cases.
If $k\theta_j=k\theta_i$, then
\[
 \beta+\ell_{i-1}+\ell_j\equiv
 \beta+\ell_{i-1}+\ell_i\equiv0\pmod{p^i}.
\]
Subtracting gives $p^i\mid\ell_j-\ell_i$.
Apply Lemma~\ref{lem:change-of-trivialization}(i) with
\[
 \sigma=\sigma_i,\quad q=p^i,\quad \ell=\ell_i,\quad
 r=\ell_j,\quad \mathcal H=\JJ^{p^j},\quad\tau=\sigma_j.
\]
It follows that $p^{i+1}\mid\ell_j-\ell_i$.

If $k\theta_j\ne k\theta_i$, then
$k\theta_j=k\theta_{i-1}$, and
\[
 \ell_j-\ell_{i-1}\equiv
 \beta+\ell_{i-1}+\ell_i\equiv0\pmod{p^i}.
\]
Adding gives $p^i\mid\beta+\ell_i+\ell_j$.
For the same data, part~(ii) of
Lemma~\ref{lem:change-of-trivialization} applies because
$\theta_j\smile\theta_i\ne0$.
Thus $p^{i+1}\mid\beta+\ell_i+\ell_j$.
These two cases prove \eqref{eq:later-congruence} for $i$
and complete the induction.
\end{proof}

\subsection{Comparison of two values of the connecting map}

Keep the generators of Lemma~\ref{lem:order-congruences},
and set $v_i=n\theta_i$.
Since $0<n<p$, \eqref{eq:obstruction-parity-lines} gives
$v_{i-1}\smile v_i\ne0$.
We next produce a class whose connecting image is this
nonzero cup product. Its construction is a product of a
class extending to $D_{\ell_i+1}$ and the section
$\sigma_i^n$ on $D_{\ell_i}$.

\begin{proposition}\label{prop:comparison-class}
For $1\leq i\leq e$, there exists
\[
 \alpha_i\in H^1(\JJ^{\ell_{i-1}}/\JJ^{\ell_{i-1}+\ell_i})
\]
such that $(\alpha_i)_0=v_{i-1}$ and
\begin{equation}\label{eq:comparison-class-boundary}
               \partial_{\ell_{i-1},\ell_i}(\alpha_i)
                                      =v_{i-1}\smile v_i.
\end{equation}
\end{proposition}

\begin{proof}
\noindent\emph{Step 1. The class and its restriction.}
Set $\rho_i=\ell_{i-1}-np^i$.
Lemma~\ref{lem:extended-class}, proved in
Appendix~\ref{app:witt}, gives
\[
 h_i\in H^1(D_{\ell_i+1},\JJ^{\rho_i}|_{D_{\ell_i+1}}),
                         \qquad h_i|_D=s^{\rho_i}v_{i-1}.
\]
Define
\[
 h'_i=h_i|_{D_{\ell_i}},\qquad
 \alpha_i=h'_i\sigma_i^n
       \in H^1(D_{\ell_i},\JJ^{\ell_{i-1}}|_{D_{\ell_i}}).
\]
The line bundle in the target follows from
$\rho_i+np^i=\ell_{i-1}$. Moreover,
\[
 (\alpha_i)_0
   =s^{-\ell_{i-1}}(s^{\rho_i}v_{i-1})s^{np^i}=v_{i-1}.
\]

\smallskip
\noindent\emph{Step 2. The connecting image.}
Use the restriction sequences from $D_{\ell_i+1}$ to
$D_{\ell_i}$. The class $h'_i$ lifts to $h_i$, so
\[
                         \delta^1_{\JJ^{\rho_i}}(h'_i)=0.
\]
Lemma~\ref{lem:obstruction-rules}(ii) gives the other factor's
connecting image:
\[
 s^{-np^i-\ell_i}\delta^0_{\JJ^{np^i}}(\sigma_i^n)
      =-n\theta_i=-v_i.
\]
Apply Lemma~\ref{lem:boundary-product} to $h'_i\sigma_i^n$
and multiply the target by $s^{-\ell_{i-1}-\ell_i}$.
This gives
\[
 \partial_{\ell_{i-1},\ell_i}(\alpha_i)
       =0-v_{i-1}\smile(-v_i)=v_{i-1}\smile v_i.
\]
\end{proof}

\begin{proof}[Proof of Theorem~\ref{thm:finite-boundary}]
The case $R=1$ is immediate from $\JJ|_D=\OO_Ds$.
Corollary~\ref{cor:one-direction} proves the assertion
when the nontrivial eigenspaces have total dimension
at most one.
Otherwise, suppose that \eqref{eq:finite-conclusion} fails
for every $0\leq i\leq e$, and choose $\ell_i,\sigma_i$
as in \eqref{eq:orders-definition} and
Lemma~\ref{lem:order-congruences}.

\noindent\emph{Step 1. Comparison with the canonical formula.}
By \eqref{eq:adjacent-congruence}, the number
\[
                     t_i=\frac{\beta+\ell_{i-1}+\ell_i}{p^i}
\]
is an integer.
Apply Proposition~\ref{prop:canonical-boundary} to
$\alpha_i$ with
\[
 q=p^i,\quad \ell=\ell_i,\quad r=\ell_{i-1},\quad
                    \theta=\theta_i=v_i/n.
\]
Together with Proposition~\ref{prop:comparison-class},
this gives
\[
 v_{i-1}\smile v_i
   =\partial_{\ell_{i-1},\ell_i}(\alpha_i)
   =\frac{t_i}{n}\,v_{i-1}\smile v_i.
\]
By \eqref{eq:obstruction-parity-lines} this cup product
is nonzero. Since $0<n<R<p$, the integer $n$ is invertible
in $k$. Therefore
\begin{equation}\label{eq:coefficient-congruence}
                              t_i\equiv n\pmod p.
\end{equation}

\smallskip
\noindent\emph{Step 2. Divisibility by $R$.}
The nonzero class $v_{i-1}\smile v_i$ generates
$H^2(D,\OO_D)=H^2(A,\OO_A)$ and has character
$\mu^{\ell_{i-1}+\ell_i}$.
The equivariant Serre-duality pairing
\[
 H^0(A,\omega_A)\otimes H^2(A,\OO_A)
                       \longrightarrow k
\]
has invariant target. By
Lemma~\ref{lem:quotient-duality},
$\omega_A\simeq\MM|_A$ equivariantly, so
$H^0(A,\MM|_A)$ has character $\mu^{-\ell_{i-1}-\ell_i}$.
Since $s^\beta|_A$ generates $\JJ^\beta|_A$ with character
$\mu^{-\beta}$, the character on
$H^0(A,\MM\JJ^\beta|_A)$ is
$\mu^{-\beta-\ell_{i-1}-\ell_i}$.
The assumed invariant generating section forces
\[
 R\mid\beta+\ell_{i-1}+\ell_i=p^it_i.
\]
As $(R,p)=1$, one concludes that $R\mid t_i$.

\smallskip
\noindent\emph{Step 3. The contradiction.}
Using \eqref{eq:finite-bound} and \eqref{eq:order-bounds},
we obtain
\[
 0<t_e<
 \frac{(p-R)p^{e-1}+Rp^{e-1}+Rp^e}{p^e}=R+1.
\]
Because $t_e$ is a positive integer divisible by $R$,
this implies $t_e=R$.
Equation~\eqref{eq:coefficient-congruence} would then give
$R\equiv n\pmod p$, which contradicts $0<n<R<p$.

It follows that $\JJ^{p^i}|_{D_{Rp^i}}$ is trivial for
some $0\leq i\leq e$.
Choose its generator with character $\mu^{-p^i}$ using
\eqref{eq:generator-character-projection}.
Its $R$-th power is invariant and gives
\eqref{eq:finite-conclusion}.
\end{proof}

\section{The log canonical relation and semiampleness}
\label{sec:canonical-relation}

In this section, we prove Corollary \ref{cor:local-canonical-relation}, which is the main technical result of this paper.
Theorem~\ref{thm:finite-boundary} assumes a specified
isomorphism $\MM\JJ^\beta\simeq\OO_{\YY}$ with an integer
$\beta>0$. We will remove this assumption by Subsections \ref{Subsection:inv-coh}, \ref{Subsection:compatible-root} and \ref{Subsection:tame-covers}.

\begin{notation}\label{not:principal-relation}
Let $X$ be a normal projective threefold, and let $S$
be a prime $\Q$-Cartier divisor. Suppose its formal
completion admits a finite quasi-\'{e}tale quotient
\[
                         \rho:\YY\longrightarrow\widehat X_S
\]
as in Lemma~\ref{lem:invariant-descent}.
Write $A=(\rho^{-1}S)_{\mathrm{red}}$ and
\[
 \II_j=\OO_{\YY}(-jA),\quad
 \JJ=\II_c,\quad D_\ell=V(\JJ^\ell),\quad
 D=D_1,\quad\MM=\omega_{\YY}(A),\quad\LL=\MM\JJ.
\]
Assume we are in
Setting~\ref{not:equivariant}.
The quotient group $G$ is cyclic, the character
$\mu^{-1}$ on $\JJ|_D$ is faithful, and
$1<R=|G|\leq12$.
Both characters of $V=H^1(D,\OO_D)$ are nontrivial.
In this section no integer $\beta$ or isomorphism
$\MM\JJ^\beta\simeq\OO_{\YY}$ is part of the setting.
\end{notation}

The bound $R\leq12$ is supplied by
Proposition~\ref{prop:bounded-action} for the covers used in
Corollary~\ref{cor:local-canonical-relation}. In particular,
$p>42$ gives $2R<p$. We will use this inequality to work
with square-zero extensions in
Lemma~\ref{lem:canonical-cyclic-subgroup}.

\begin{proposition}\label{prop:primitive-relation}
In Setting~\ref{not:principal-relation}, choose a nonzero
rational top form $\kappa$ on $X$ and set
$K_X=\divisor(\kappa)$. Suppose that
\begin{equation}\label{eq:original-relation}
 \divisor(f)|_U=\bigl(q(K_X+S)-bS\bigr)|_U,\qquad
 f\in k(X)^*,\qquad q,b\in\Z_{>0}.
\end{equation}
on an open neighbourhood $U$ of $S$.
Write $q=p^hq_0$ with $(q_0,p)=1$.
Then either $\JJ^{Rp^i}|_{D_{Rp^i}}$ has an invariant
generating section for some $i\geq0$, or
\begin{equation}\label{eq:primitive-relation}
 \begin{gathered}
 f=u^{p^h}\text{ for some }u\in k(X)^*,\qquad p^h\mid b,\\
 \divisor(u)|_U=\bigl(q_0(K_X+S)-(b/p^h)S\bigr)|_U.
 \end{gathered}
\end{equation}
\end{proposition}

If the first case of
Proposition~\ref{prop:primitive-relation} does not hold,
we first prove $p^h\mid b$ and triviality of
$\MM^{q_0}\II_{b/p^h}$ on finite thickenings.
We then construct compatible generating sections whose
$p^h$-th powers equal the prescribed section
$\rho^*(\kappa^q/f)$.

\subsection{Invariant cohomology on finite thickenings}\label{Subsection:inv-coh}

Throughout this subsection, assume that the first case
of Proposition~\ref{prop:primitive-relation} does not hold.
The kernel of $\Pic(D_N)\to\Pic(D)$ is $p$-primary, and
$R$ is prime to $p$. Moreover, a trivial line bundle with
trivial character on $D$ has an invariant generating section,
by averaging. Consequently our assumption is equivalent to
\[
                 \JJ^{p^e}|_{D_{Rp^e}}\not\simeq\OO_{D_{Rp^e}}
                 \qquad(e\geq0).
\]
Set
\begin{equation}\label{eq:primitive-orders}
 Q_e=Rp^e,\qquad
 \ell_e=\max\{\ell\geq1:
                   \JJ^{p^e}|_{D_\ell}\simeq\OO_{D_\ell}\}.
\end{equation}
Thus $\ell_e+1$ is the first length at which
$\JJ^{p^e}$ is nontrivial. Taking $p$-th powers of local
generators which agree modulo $\JJ^{\ell_e}$ gives generators
which agree modulo $\JJ^{p\ell_e}$. It follows that
\begin{equation}\label{eq:primitive-order-bounds}
 \begin{gathered}
 1\leq n:=\ell_0<R,\qquad
 np^e\leq\ell_e<Q_e,\qquad p\ell_e\leq\ell_{e+1},\\
       \ord_{\Pic(D_{\ell_e+1})}(\JJ|_{D_{\ell_e+1}})
                                                   =p^{e+1}.
 \end{gathered}
\end{equation}
The last assertion concerns the order of a line bundle
in the ordinary Picard group: its $a$-th tensor power
is trivial if and only if $p^{e+1}\mid a$.
Indeed, maximality and $\ell_e+1\leq p\ell_e$ give, respectively,
\[
 \JJ^{p^e}|_{D_{\ell_e+1}}\not\simeq\OO_{D_{\ell_e+1}},
 \qquad
 \JJ^{p^{e+1}}|_{D_{\ell_e+1}}\simeq\OO_{D_{\ell_e+1}}.
\]
The same $p$-primary Picard kernel gives the useful criterion
\[
 \JJ^a|_{D_N}\simeq\OO_{D_N}
       \quad\Longleftrightarrow\quad N\leq\ell_{v_p(a)}
       \qquad(a\in\Z\setminus\{0\}).
\]
Indeed, writing $a=p^{v_p(a)}a_0$, multiplication by $a_0$
is an automorphism of this kernel. Here $v_p(a)$ denotes
the exponent of $p$ in the nonzero integer $a$.

We will prove that $\LL=\MM\JJ$ is a power of $\JJ$ on
$D_{\ell_e+1}$, and that invariant functions on $D_{Q_e}$
are constant. The first assertion will give integer
congruences from \eqref{eq:original-relation}; the second
will give uniqueness of normalized generating sections.

For the cohomology calculation, set
\[
 h_G^j(\mathcal F)=\dim_kH^j(\mathcal F)^G,\qquad
 \chi_G(\mathcal F)=\sum_{j=0}^2(-1)^jh_G^j(\mathcal F).
\]
All sheaves below are supported on specified finite proper
thickenings. We regard their extensions by zero as sheaves
on $\YY$ when writing an exact sequence involving different
thickenings. We write
$\mathcal F\simeq_G\OO_{D_N}$ to mean that $\mathcal F$
has a $G$-invariant generating section on $D_N$.

We use two facts from Section~\ref{sec:divisorial}.
For an equivariant line bundle $\mathcal F$ on $\YY$
with $\mathcal F|_D\simeq\OO_D$,
Lemma~\ref{lem:section-filtration} filters
$H^0(D_N,\mathcal F|_{D_N})^G$ by order of vanishing
along $D$. Its order-$j$ quotient has dimension one
precisely when
$\mathcal F\JJ^j|_{D_{N-j}}\simeq_G\OO_{D_{N-j}}$.
Lemma~\ref{lem:dualizing-comparison} gives
\begin{equation}\label{eq:canonical-L-map}
                 \LL\JJ^{-N}|_{D_N}\longrightarrow\omega_{D_N}.
\end{equation}
This map induces cohomology isomorphisms after tensoring
by any power of $\JJ$. Thus Serre duality reduces the
calculation of $H^2$ to the same count of sections.

\begin{lemma}\label{lem:canonical-cyclic-subgroup}
Use Setting~\ref{not:principal-relation}, with
$\JJ=\OO_{\YY}(-cA)$, $D_N=V(\JJ^N)$, and
$\LL=\omega_{\YY}(A)\otimes\JJ$.
Assume that $\JJ^{p^e}|_{D_{Rp^e}}$ is nontrivial for
every $e\geq0$. For each $e$, let $\ell_e$ be the largest
integer for which $\JJ^{p^e}|_{D_{\ell_e}}$ is trivial.
Then there is an integer $\sigma_e$ such that
\begin{equation}\label{eq:canonical-cyclic-subgroup}
                  \LL|_{D_{\ell_e+1}}
                      \simeq\JJ^{\sigma_e}|_{D_{\ell_e+1}}.
\end{equation}
This is an isomorphism of line bundles; it is not asserted
to preserve the $G$-actions. Equivalently,
$\omega_{\YY}(A)|_{D_{\ell_e+1}}
 \simeq\JJ^{\sigma_e-1}|_{D_{\ell_e+1}}$.
\end{lemma}

\begin{proof}
Fix $e$, set $Q=Q_e$, and suppose that
\eqref{eq:canonical-cyclic-subgroup} fails.
Define a line bundle on $D_Q$ and coherent quotient sheaves by
\[
 \begin{gathered}
 \mathcal P=\JJ^Q|_{D_Q},\qquad
 \mathcal P^u=\JJ^{uQ}/\JJ^{(u+1)Q},\\
 \mathcal E_u=\JJ^{uQ}/\JJ^{(u+2)Q}
             =\JJ^{uQ}|_{D_{2Q}}\qquad(1\leq u<p).
 \end{gathered}
\]
The first three steps will show that the restriction map
\[
            H^1(\mathcal E_u)^G\longrightarrow H^1(\mathcal P^u)^G
\]
is zero. In Step~4, the line bundle $\JJ^{pQ}$ will give a
class whose image under this map is nonzero.

\noindent\emph{Step 1. Compute $H^0$ and the Euler characteristic.}
The same section calculation also applies to
$\mathcal P^0=\OO_{D_Q}$. We claim
\begin{equation}\label{eq:invariant-Euler}
 \begin{gathered}
 h_G^0(\mathcal P^0)=1,\qquad
 h_G^0(\mathcal P^u)=0\quad(1\leq u<p),\\
                  \chi_G(\mathcal P^u)=0\quad(0\leq u<p).
 \end{gathered}
\end{equation}
This calculation uses only \eqref{eq:primitive-order-bounds},
not our assumption that \eqref{eq:canonical-cyclic-subgroup}
fails.

Fix $0\leq u<p$. A nonzero order-$j$ quotient of the
section filtration, with $0\leq j<Q$, has a nonzero image in
$H^0(D,\JJ^{uQ+j}|_D)$, on which $G$ acts by $\mu^{-j}$.
Hence $j=Ra$ for an integer $0\leq a<p^e$.
Lemma~\ref{lem:section-filtration} would give
\[
                         \JJ^{uQ+j}|_{D_{Q-j}}
                                      \simeq_G\OO_{D_{Q-j}}.
\]
Suppose $(u,j)\ne(0,0)$ and set $\nu=v_p(uQ+j)$.
Since $uQ+j=R(up^e+a)$, we have
\[
 \nu=
 \begin{cases}
 e,&a=0,\\
 v_p(a)<e,&a>0.
 \end{cases}
\]
In both cases $Q-j=R(p^e-a)$ is a positive multiple of
$Rp^\nu$. Therefore
\[
                    Q-j\geq Rp^\nu>\ell_\nu,
\]
contrary to the triviality criterion following
\eqref{eq:primitive-order-bounds}.
Thus all graded quotients vanish unless $u=j=0$.
In that case the constant section $1$ gives a
one-dimensional quotient. This proves the two assertions
about $H^0$ in \eqref{eq:invariant-Euler}.

The filtration of $\mathcal P^u$ has quotients
$\JJ^{uQ+j}|_D$ for $0\leq j<Q$.
Their characters run through all characters of $G$
exactly $p^e$ times. Additivity of the Euler characteristic
and exactness of invariants give
\[
 \begin{aligned}
 \chi_G(\mathcal P^u)
 &=p^e\sum_{a=0}^{R-1}\sum_{i=0}^2
            (-1)^i\dim_k H^i(D,\OO_D)_{\mu^a}\\
 &=p^e(1-2+1)=0.
 \end{aligned}
\]

\smallskip
\noindent\emph{Step 2. Compare $H^2$ on $D_Q$ and $D_{2Q}$.}
For $u\geq1$ and $N\in\{Q,2Q\}$, equivariant Serre duality
and \eqref{eq:canonical-L-map} identify the dual of
$H^2(D_N,\JJ^{uQ}|_{D_N})$ with
\[
                     H^0(D_N,\LL\JJ^{-uQ-N}|_{D_N}).
\]
A $G$-invariant section of order $j$ in this space gives
an invariant generator of the line bundle
\[
 \LL\JJ^{-uQ-N+j}|_{D_{N-j}}
       =\LL\JJ^{-uQ-\ell}|_{D_\ell},\qquad\ell=N-j.
\]
Lemma~\ref{lem:section-filtration} therefore gives the count
\begin{equation}\label{eq:dual-section-filtration}
 h_G^2(\JJ^{uQ}|_{D_N})
  =\#\{\ell\in\Z:1\leq\ell\leq N,\
       \LL\JJ^{-uQ-\ell}|_{D_\ell}\simeq_G\OO_{D_\ell}\}.
\end{equation}
If a length $\ell>\ell_e$ occurred, restriction to
$D_{\ell_e+1}$ would give
\[
                  \LL|_{D_{\ell_e+1}}
                      \simeq\JJ^{uQ+\ell}|_{D_{\ell_e+1}},
\]
contrary to our assumption.
For $1\leq\ell\leq\ell_e$, the line $\JJ^Q|_{D_\ell}$
is trivial and its character on $D$ is $\mu^{-Q}=1$.
It therefore has an invariant generator. Tensoring by its
powers shows that the condition in
\eqref{eq:dual-section-filtration} is independent of $u$.
Thus, if
\[
 T=\{\ell\in\Z:1\leq\ell\leq\ell_e,\
                      \LL\JJ^{-\ell}|_{D_\ell}\simeq_G\OO_{D_\ell}\},
 \qquad t=\#T,
\]
then
\[
 h_G^2(\mathcal P^u)=h_G^2(\mathcal P^{u+1})
                  =h_G^2(\mathcal E_u)=t
                                      \qquad(1\leq u<p).
\]
In particular, this comparison includes $\mathcal P^p$;
no assertion about $H^0(\mathcal P^p)^G$ is needed.

\smallskip
\noindent\emph{Step 3. The connecting homomorphism is injective.}
The quotient of $\JJ^{uQ}$ by two successive powers
gives an exact sequence of coherent sheaves
\begin{equation}\label{eq:two-thickenings}
 0\longrightarrow\mathcal P^{u+1}\longrightarrow\mathcal E_u
                         \longrightarrow\mathcal P^u\longrightarrow0.
\end{equation}
All three sheaves are supported on proper surfaces, so their
cohomology vanishes in degree three. Hence
\[
                 H^2(\mathcal E_u)^G\longrightarrow
                         H^2(\mathcal P^u)^G\longrightarrow0
\]
is surjective. Its source and target both have dimension
$t$, by Step~2, so it is an isomorphism.
The long exact sequence of \eqref{eq:two-thickenings}
now shows that its connecting homomorphism
$\delta_u:H^1(\mathcal P^u)^G\to H^2(\mathcal P^{u+1})^G$
is surjective. On the other hand, Step~1 gives
\[
 h_G^1(\mathcal P^u)=h_G^2(\mathcal P^u)=t
                                      =h_G^2(\mathcal P^{u+1}).
\]
Therefore
\begin{equation}\label{eq:injective-connecting}
 \delta_u:H^1(\mathcal P^u)^G
                   \xrightarrow{\ \sim\ }H^2(\mathcal P^{u+1})^G.
\end{equation}
In particular, exactness gives
\[
 \im\bigl(H^1(\mathcal E_u)^G\longrightarrow
                  H^1(\mathcal P^u)^G\bigr)=\ker\delta_u=0.
\]

\smallskip
\noindent\emph{Step 4. A line bundle gives a nonzero image.}
Set $r=\ell_{e+1}$. Since $p>2R$ and
$\ell_e\geq p^e$, \eqref{eq:primitive-order-bounds} gives
\[
              2Q<p^{e+1}\leq p\ell_e\leq r<pQ.
\]
Choose $2\leq u<p$ such that $uQ\leq r<(u+1)Q$.
Set
\[
 \mathcal H=\JJ^{pQ},\qquad
 Z_0=D_{uQ}\subset Z_1=D_{(u+1)Q}\subset Z_2=D_{(u+2)Q}.
\]
Since $v_p(pQ)=e+1$, the triviality criterion gives
\[
           \mathcal H|_{Z_0}\simeq\OO_{Z_0},
           \qquad \mathcal H|_{Z_1}\not\simeq\OO_{Z_1}.
\]
The character of $\mathcal H|_D$ is $\mu^{-pQ}=1$.
Choose an invariant generating section $\tau$ on $Z_0$
by averaging.

The ideals of $Z_0$ in $Z_2$ and $Z_1$ are respectively
$\mathcal E_u$ and $\mathcal P^u$.
Since $u\geq2$, both are square-zero:
\[
 \JJ^{2uQ}\subseteq\JJ^{(u+2)Q}\subseteq\JJ^{(u+1)Q}.
\]
For either ideal $\mathcal K$, the map
$1+\mathcal K\to(\mathcal K,+)$ is simply $1+x\mapsto x$.
The exact sequences of units for $Z_0\subset Z_2$ and
$Z_0\subset Z_1$ therefore give the commutative diagram
\[
\begin{tikzcd}[column sep=large]
 H^1(\mathcal E_u)\arrow[r]\arrow[d]
   &\Pic(Z_2)\arrow[r]\arrow[d,"\mathrm{res}"]
   &\Pic(Z_0)\arrow[d,equal]\\
 H^1(\mathcal P^u)\arrow[r]
   &\Pic(Z_1)\arrow[r]&\Pic(Z_0).
\end{tikzcd}
\]
The fixed trivialization $\tau$ determines a class
$\eta\in H^1(\mathcal E_u)$ mapping to $[\mathcal H|_{Z_2}]$:
lift $\tau$ locally to generators on $Z_2$ and apply
the map $1+x\mapsto x$ to their transition functions.
Changing the lifts changes this class by a coboundary.
Since $\mathcal H$ and $\tau$ are equivariant, this
construction gives $\eta\in H^1(\mathcal E_u)^G$.

Let $\bar\eta\in H^1(\mathcal P^u)^G$ be its restriction.
The diagram sends $\bar\eta$ to
$[\mathcal H|_{Z_1}]\ne0$, so $\bar\eta\ne0$.
But Step~3 says that the map sending $\eta$ to $\bar\eta$
is zero. This contradiction proves
\eqref{eq:canonical-cyclic-subgroup}.
\end{proof}

If two invariant generating sections of the same
line bundle agree on $D$, their quotient is an invariant
unit restricting to $1$ on $D$. The next lemma will force
that quotient to be $1$ on $D_{Q_e}$.

\begin{lemma}\label{lem:invariant-functions}
For every $e\geq0$,
\[
                         H^0(D_{Q_e},\OO_{D_{Q_e}})^G=k.
\]
\end{lemma}

\begin{proof}
The case $u=0$ of \eqref{eq:invariant-Euler} gives
$\dim_kH^0(D_{Q_e},\OO_{D_{Q_e}})^G=1$.
Since this space contains the constants, it equals $k$.
As noted in Step~1 of the preceding proof, that calculation
does not use the assumption that
\eqref{eq:canonical-cyclic-subgroup} fails.
\end{proof}

\subsection{A compatible root of the log canonical section}\label{Subsection:compatible-root}

We now use the two lemmas separately.
Lemma~\ref{lem:canonical-cyclic-subgroup} will imply
$p^h\mid b$ and triviality of the line bundle
$\MM^{q_0}\II_{b/p^h}$ on every $D_{Q_i}$.
Lemma~\ref{lem:invariant-functions} will then produce
compatible generating sections whose $p^h$-th powers
equal $\rho^*(\kappa^q/f)$.
Only after constructing this specified root on $\YY$
will we recover a root of $f$ in $k(X)$.

\begin{proof}[Proof of Proposition~\ref{prop:primitive-relation}]
If $h=0$, the second case holds with $u=f$.
We may therefore assume $h>0$ and that the first
case does not hold. All the conclusions of
Section~6.1 are then available.

\noindent\emph{Step 1. Prove $p^h\mid b$ in the integers.}
The quasi-\'{e}tale pullback of
\eqref{eq:original-relation} gives an invariant
generating section
\[
          \Gamma=\rho^*(\kappa^q/f):
          \OO_{\YY}\xrightarrow{\sim}\MM^q\II_b.
\]
In particular, $\II_b\simeq\MM^{-q}$ is invertible and
$\II_b|_A\simeq\OO_A$. Lemma~\ref{lem:section-period}
therefore gives $c\mid b$. Thus we may write
\begin{equation}\label{eq:canonical-period}
            d=b/c>0,\qquad
            \Gamma:\OO_{\YY}\xrightarrow{\sim}\MM^q\JJ^d.
\end{equation}

For each $e$, Lemma~\ref{lem:canonical-cyclic-subgroup}
provides an integer $\sigma_e$ with
\[
        \MM|_{D_{\ell_e+1}}
                  \simeq\JJ^{\sigma_e-1}|_{D_{\ell_e+1}}.
\]
Taking its $q$-th tensor power and using the trivialization
$\Gamma$ in \eqref{eq:canonical-period} gives
\[
 \OO_{D_{\ell_e+1}}
   \simeq\MM^q\JJ^d|_{D_{\ell_e+1}}
   \simeq\JJ^{q(\sigma_e-1)+d}|_{D_{\ell_e+1}}.
\]
The line bundle $\JJ|_{D_{\ell_e+1}}$ has exact order
$p^{e+1}$ by \eqref{eq:primitive-order-bounds}.
Thus the exponent in the last expression is divisible
by $p^{e+1}$:
\begin{equation}\label{eq:integer-relation}
                         q(\sigma_e-1)+d\equiv0\pmod{p^{e+1}}.
\end{equation}
In particular, take $e=h$ and recall $q=p^hq_0$.
For some integer $t$ we obtain
\[
 \begin{aligned}
 p^hq_0(\sigma_h-1)+d&=p^{h+1}t,\\
 d&=p^h\bigl(pt-q_0(\sigma_h-1)\bigr).
 \end{aligned}
\]
This proves $p^h\mid d$ in $\Z$, and therefore
$p^h\mid b=cd$. No divisibility property of $c$ is used.
Define
\[
 d_0=d/p^h,\qquad b_0=b/p^h=cd_0,\qquad
 \mathcal H=\MM^{q_0}\JJ^{d_0}.
\]
Then $\mathcal H^{p^h}=\MM^q\JJ^d$, with the specified
generating section $\Gamma$.
The integers $\sigma_e$ may depend on $e$; the line
bundle $\mathcal H$ is fixed.

\smallskip
\noindent\emph{Step 2. Trivialize $\mathcal H$ on the thickenings $D_{Q_i}$.}
Fix $i\geq0$ and apply \eqref{eq:integer-relation}
with $e=i+h$. Dividing by $p^h$ gives
\[
          q_0(\sigma_{i+h}-1)+d_0=p^{i+1}t_i,
                                      \qquad t_i\in\Z.
\]
The inequalities $h\geq1$, $R<p$, and
\eqref{eq:primitive-order-bounds} give
\[
                   Q_i=Rp^i<p^{i+1}\leq\ell_{i+1}
                                                \leq\ell_{i+h}.
\]
Thus both the isomorphism for $\MM$ from Step~1 and
the trivialization of $\JJ^{p^{i+1}}$ on $D_{\ell_{i+1}}$
restrict to $D_{Q_i}$. They give
\[
 \begin{aligned}
 \mathcal H|_{D_{Q_i}}
   &\simeq\JJ^{q_0(\sigma_{i+h}-1)+d_0}|_{D_{Q_i}}\\
   &=\bigl(\JJ^{p^{i+1}}|_{D_{Q_i}}\bigr)^{t_i}
                                                \simeq\OO_{D_{Q_i}}.
 \end{aligned}
\]
This proves
\begin{equation}\label{eq:primitive-finite-lines}
                 \mathcal H|_{D_{Q_i}}
                          \simeq\OO_{D_{Q_i}}\qquad(i\geq0).
\end{equation}
At this stage these are isomorphisms of line bundles;
no compatibility between their generating sections
has yet been asserted.

\smallskip
\noindent\emph{Step 3. Choose invariant generators with a fixed restriction.}
Let $\chi$ be the character of $G$ on
$H^0(D,\mathcal H|_D)$. The section $\Gamma|_D$ is invariant,
so $\chi^{p^h}=1$. Since $\chi^R=1$ and $(R,p)=1$,
we have $\chi=1$.
The equality $H^0(D,\OO_D)=k$ and algebraic closedness
therefore allow us to choose an invariant generating section
$\gamma_D$ satisfying
\[
                              \gamma_D^{p^h}=\Gamma|_D.
\]

For each $i\geq0$, choose a generating section of
$\mathcal H|_{D_{Q_i}}$, which exists by Step~2.
Its restriction to $D$ is a nonzero scalar multiple
of $\gamma_D$. Averaging over $G$ preserves that
restriction, so the average is still generating by
Nakayama's lemma. Rescaling gives an invariant section
\[
 \gamma_i\in H^0(D_{Q_i},\mathcal H|_{D_{Q_i}})^G,
                         \qquad\gamma_i|_D=\gamma_D.
\]

\smallskip
\noindent\emph{Step 4. Prove compatibility and the prescribed power identity.}
If $i'\geq i$, the following quotients of generating
sections are invariant units on $D_{Q_i}$:
\[
          \frac{\gamma_{i'}|_{D_{Q_i}}}{\gamma_i},
          \qquad
          \frac{\gamma_i^{p^h}}{\Gamma|_{D_{Q_i}}}.
\]
Both restrict to $1$ on $D$.
Lemma~\ref{lem:invariant-functions} says that they belong
to $k^*$, and hence both equal $1$. Thus
\[
 \gamma_{i'}|_{D_{Q_i}}=\gamma_i,\qquad
                 \gamma_i^{p^h}=\Gamma|_{D_{Q_i}}.
\]
The lengths $Q_i=Rp^i$ are cofinal among all positive
integers. By completeness, the compatible sections
$\gamma_i$ define a generating section $\gamma$ of
$\mathcal H$ on $\YY$. They satisfy
\begin{equation}\label{eq:formal-power-relation}
                                  \gamma^{p^h}=\Gamma.
\end{equation}
This constructs a root of the specified section $\Gamma$,
not only an abstract trivialization of $\mathcal H^{p^h}$.

\smallskip
\noindent\emph{Step 5. Recover the root in $k(X)$.}
Choose $x\in S(k)$ at which both $X$ and $S$ are smooth
and $\rho$ is \'{e}tale. Such points form a nonempty open
subset of $S$, since $X$ is normal and $\rho$ is quasi-\'{e}tale.
For a point $y$ above $x$, the residue field is $k$,
and the completed local rings are isomorphic
\cite[Tag 0AGX]{Sta}.
Under this identification, $\mathcal H=\MM^{q_0}\II_{b_0}$
is the completion at $x$ of the line bundle
\[
                   \omega_X^{\otimes q_0}((q_0-b_0)S)
\]
on a smooth neighbourhood of $x$.
Set $\widehat K_x=\operatorname{Frac}\widehat\OO_{X,x}$.
The numerator $\kappa^{q_0}$ and denominator $\gamma$
are rational sections of the same rank-one module over
$\widehat\OO_{X,x}$. Their quotient is therefore an
element of $\widehat K_x^*$.
Using $\Gamma=\rho^*(\kappa^q/f)$ in
\eqref{eq:formal-power-relation} gives
\begin{equation}\label{eq:local-pth-power}
                     f=(\kappa^{q_0}/\gamma)^{p^h}
                                      \quad\text{in }\widehat K_x.
\end{equation}
Lemma~\ref{lem:field-powers} applies at the smooth point $x$:
\[
               k(X)\cap\widehat K_x^{p^h}=k(X)^{p^h}.
\]
Hence $f=u^{p^h}$ for some $u\in k(X)^*$.
Taking divisors and dividing \eqref{eq:original-relation}
by $p^h$ now gives
\[
             \divisor(u)|_U
                   =\bigl(q_0(K_X+S)-b_0S\bigr)|_U,
\]
which is \eqref{eq:primitive-relation}.
\end{proof}

\subsection{The tame cyclic cover}\label{Subsection:tame-covers}

The root used in this subsection is algebraic, i.e. its defining function
belongs to $k(X)$ by Proposition~\ref{prop:primitive-relation}.
The normalized cyclic algebra and its ramification
indices are recalled in
Lemma~\ref{lem:normalized-cyclic-algebra}.

\begin{corollary}\label{cor:tame-canonical-relation}
Let $X$ be a normal variety, $S$ a prime
effective $\Q$-Cartier divisor, and
$K_X=\divisor(\kappa)$ the divisor of a nonzero rational
top form.
Suppose that
\[
 \divisor(f)=q(K_X+S)-bS,\qquad
                   (q,p)=1,\qquad q,b\in\Z_{>0}.
\]
Let $\pi:X'\to X$ be a connected component of the
normalization in $z^q=f$.
Set
\[
 g=\gcd(q,b),\quad e=q/g,\quad B=b/g,\quad
                         D'=(\pi^{-1}S)_{\mathrm{red}}.
\]
Then
\begin{equation}\label{eq:tame-global-formula}
 \pi^*S=eD',\qquad
                  \divisor(\pi^*\kappa/z)=(B-1)D'.
\end{equation}
In particular, $K_{X'}+D'\sim BD'$ with $B>0$.
If $X$ is klt and $(X,S)$ is lc, then $X'$ is klt
and $(X',D')$ is lc.
\end{corollary}

\begin{proof}
For a prime divisor $P\ne S$, the valuation of $f$
is $q\operatorname{coeff}_P(K_X)$.
Lemma~\ref{lem:normalized-cyclic-algebra} gives
ramification index one at every prime above $P$.
Above $S$ the index is $e=q/\gcd(q,b)$.
Thus $\pi^*S=eD'$ and the ramification divisor is
$(e-1)D'$.
All indices are prime to $p$, so
\[
 \begin{aligned}
 \divisor(\pi^*\kappa)&=\pi^*K_X+(e-1)D',\\
 \divisor(z)&=q^{-1}\pi^*\divisor(f)
             =\pi^*K_X+(1-b/q)eD'.
 \end{aligned}
\]
Subtracting, and using $eb/q=B$, gives
\eqref{eq:tame-global-formula}.
The latter asertion holds by \cite[Proposition 5.20]{KM98}.
\end{proof}

\subsection{The local semiampleness criterion}\label{sec:main-proof}

Now we can prove the main technical result of this paper.

\begin{corollary}\label{cor:local-canonical-relation}
Let $X$ be a normal projective threefold over an algebraically
closed field of characteristic $p>42$, and let $S$ be a prime
effective $\Q$-Cartier divisor. Assume that $X$ is klt and
$(X,S)$ is lc on an open neighbourhood $U$ of $S$. Suppose
\[
 (K_X+S)|_U\sim_{\Q}aS|_U,\qquad a\in\Q_{>0},
 \qquad
 mS\text{ is Cartier},\quad \OO_X(mS)|_S\simeq\OO_S
\]
for some integer $m>0$. Then $S$ is semiample on $X$.
\end{corollary}

\begin{proof}
\noindent\emph{Step 1. Reduction to a prime-to-$p$ denominator.}
Choose a nonzero rational top form $\kappa$ and a function
$f\in k(X)^*$ such that
\[
 K_X=\divisor(\kappa),\qquad
 \divisor(f)|_U=\bigl(q(K_X+S)-bS\bigr)|_U,
 \qquad q,b\in\Z_{>0},\quad b/q=a.
\]
After clearing denominators, the hypotheses imply that
a Cartier multiple of $K_X+S$ restricts trivially to $S$.
Apply Proposition~\ref{prop:surface-cover} on $\widehat X_S$.
If the first cohomology groups specified there vanish,
the proposition proves semiampleness directly.
Otherwise apply Proposition~\ref{prop:bounded-action}.
If Corollary~\ref{cor:one-direction} applies,
Lemma~\ref{lem:invariant-descent} again proves semiampleness.

In the remaining case the formal cover satisfies
Setting~\ref{not:principal-relation}, with cyclic group
of order $1<R\leq12<p$ and two nontrivial characters on
$H^1(D,\OO_D)$.
Proposition~\ref{prop:primitive-relation} either gives
the invariant finite trivialization needed by
Lemma~\ref{lem:invariant-descent}, or replaces $q,b,f$ so that
\begin{equation}\label{eq:local-prime-to-p-relation}
 (q,p)=1,
 \qquad \divisor(f)|_U=\bigl(q(K_X+S)-bS\bigr)|_U,
 \qquad b>0.
\end{equation}
We consider the latter case.

\smallskip
\noindent\emph{Step 2. The algebraic tame cyclic cover.}
Normalize $X$ in a field component of $z^q=f$ and let
$\pi:X'\to X$ be the resulting finite projective morphism.
Set
\[
 g=\gcd(q,b),\qquad e=q/g,\qquad B=b/g,
 \qquad D'=(\pi^{-1}S)_{\mathrm{red}},\quad U'=\pi^{-1}U.
\]
Corollary~\ref{cor:tame-canonical-relation}, applied to $U$,
gives
\[
 \pi^*S=eD',\qquad
 \divisor(\pi^*\kappa/z)|_{U'}=(B-1)D'|_{U'},
 \qquad B\in\Z_{>0}.
\]
It also proves that $X'$ is klt and $(X',D')$ is lc near
$D'$. No singularity assertion outside $U'$ is required.
Pullback of the given restriction isomorphism yields
\begin{equation}\label{eq:global-reduced-trivialization}
 meD'\text{ is Cartier},\qquad
                \OO_{X'}(meD')|_{D'}\simeq\OO_{D'}.
\end{equation}
For a connected component $T$ of $D'$, disjointness of
the connected components shows that $meT$ is Cartier on $X'$
and $\OO_{X'}(meT)|_T\simeq\OO_T$.
On $\widehat{X'}_T$, the rational form $\pi^*\kappa/z$
specifies the relation
\begin{equation}\label{eq:integral-log-relation}
                         K_{X'}+T\sim BT.
\end{equation}

\smallskip
\noindent\emph{Step 3. Triviality on a Cartier thickening.}
Apply Proposition~\ref{prop:surface-cover} and then
Proposition~\ref{prop:bounded-action} along $T$.
The vanishing and one-eigenspace cases are treated by the
results used in Step~1. In the remaining case let $\YY$
be the resulting formal cover, $A$ its reduced boundary,
and $\JJ=\II_c$ the invertible ideal supplied by the
surface calculation.
The section in \eqref{eq:integral-log-relation} pulls back
to an invariant generating section of $\MM\II_B$.
Lemma~\ref{lem:section-period} gives $c\mid B$, and hence
\[
 \MM\JJ^\beta\simeq\OO_{\YY},\qquad
                      \beta=B/c\in\Z_{>0},
\]
with an invariant generating section.
Theorem~\ref{thm:finite-boundary} and
Lemma~\ref{lem:invariant-descent} give an integer $N>0$ with
\[
 NT\text{ Cartier},\qquad \OO_{X'}(NT)|_{NT}\simeq\OO_{NT}.
\]
Lemma~\ref{lem:principal-parts} therefore makes $T$ semiample.

\smallskip
\noindent\emph{Step 4. Descent of semiampleness.}
Step~3 applies to every connected component of $D'$, so
$\pi^*\OO_X(mS)=\OO_{X'}(meD')$ is semiample.
The morphism $\pi$ is finite and surjective, and $X$ is
normal. Semiampleness therefore descends by
\cite[Lemma 2.10, with $D=\varnothing$]{Kee99}.
Hence $\OO_X(mS)$, and thus $S$, is semiample.
\end{proof}

\section{Proof of abundance}\label{sec:abundance}

\begin{lemma}\label{lem:abundance-component}
Let $(Y,B)$ be the pair supplied by
Proposition~\ref{prop:reduction}, and assume that
$\nu(K_Y+B)=1$. Write
\[
 B=\coprod_{i=1}^t S_i,\qquad
 L:=K_Y+B\sim_{\Q}\sum_{i=1}^t d_iS_i,\qquad d_i>0.
\]
Then $Y$ is klt. For each $i$ there exists an integer $m_i>0$
such that $m_iS_i$ is Cartier and
$\OO_Y(m_iS_i)|_{S_i}\simeq\OO_{S_i}$.
\end{lemma}

\begin{proof}
For a divisorial valuation $E$ centred in $B$,
\[
 A_E(Y,0)=A_E(Y,B)+\ord_E(B)>0,
\]
since $B$ is effective and $\Q$-Cartier.
Together with terminality of $Y\setminus B$, this proves
that $Y$ is klt.
The Cohen--Macaulayness of $\OO_Y$ and $\OO_Y(-S_i)$
\cite[Theorems 3 and 19(2)]{BK23}, and their restriction
sequence, imply that $S_i$ is Cohen--Macaulay.

Let $H$ be ample. Nefness and numerical dimension one give
\[
 0=L^2\cdot H=\sum_i d_i(L\cdot S_i\cdot H),\qquad
 L\cdot S_i\cdot H\geq0.
\]
Thus $L\cdot S_i\cdot H=0$. The Hodge index theorem on
a resolution of $S_i$ implies $L|_{S_i}\equiv0$.
Adjunction \cite[Lemma 2.19]{Xu23} gives an slc surface pair
$(S_i,\Delta_i)$ with
\[
 K_{S_i}+\Delta_i\sim_{\Q}(K_Y+S_i)|_{S_i}
                         =L|_{S_i}\equiv0.
\]
Abundance for slc surfaces \cite[Theorem 1]{Pos23}
therefore gives $L|_{S_i}\sim_{\Q}0$.
Since the $S_i$ are disjoint,
\[
                 L|_{S_i}\sim_{\Q}d_iS_i|_{S_i}.
\]
Clearing denominators proves the assertion.
\end{proof}

\begin{proof}[Proof of Theorem~\ref{thm:abundance}]
We may assume
that $k$ is algebraically closed.
The known abundance cases
\cite[Theorems 1.1, 2.28 and 2.29]{Xu23}
reduce the problem to $\nu(K_X+B)=1$ and $n(K_X+B)=3$.
Applying Proposition~\ref{prop:reduction}, we get
\[
 B_Y=\coprod_iS_i,\qquad
 L_Y:=K_Y+B_Y\sim_{\Q}\sum_i d_iS_i,\qquad d_i>0,
 \qquad \nu(L_Y)=1.
\]
Lemma~\ref{lem:abundance-component} gives $Y$ klt and
a Cartier multiple of each $S_i$ trivial on $S_i$.
For $U_i=Y\setminus\bigcup_{j\ne i}S_j$,
\[
                    (K_Y+S_i)|_{U_i}\sim_{\Q}d_iS_i|_{U_i}.
\]
Corollary~\ref{cor:local-canonical-relation} proves that every
$S_i$ is semiample. Hence $L_Y$ is semiample, and
\[
          \kappa(K_X+B)=\kappa(L_Y)=\nu(L_Y)=1.
\]
Theorem 2.28 of \cite{Xu23} now proves that $K_X+B$ is semiample.
\end{proof}

\appendix
\section{Normal surfaces and divisorial sheaves}\label{app:surface}

Throughout this appendix, $k$ is algebraically closed of
characteristic $p>42$. A rank-one reflexive sheaf $\FF$ is
called \emph{torsion} if $\FF^{[m]}\simeq\OO$ for some $m>0$.
We write $h^*(\FF)=(h^0(\FF),h^1(\FF),h^2(\FF))$. We say that
$\FF$ is \emph{acyclic} if this triple is zero.

All the surfaces in this appendix are normal.
Normalization and gluing of nonnormal surfaces are treated
separately in Appendix~\ref{app:models}.

\subsection{Normal Calabi--Yau surfaces}

The following two statements are used  in 
Section~\ref{sec:surface-cover} in the normal case.

\begin{lemma}\label{lem:normal-klt-classification}
Let $A$ be a projective Gorenstein klt surface with
$\omega_A\simeq\OO_A$. Its minimal resolution is a K3
surface or an abelian surface.
\end{lemma}

\begin{proof}
Let $f:\widetilde A\to A$ be the minimal resolution.
The discrepancies are integral and greater than $-1$.
Relative nefness of $K_{\widetilde A}$ and negativity of the
exceptional intersection matrix give
$K_{\widetilde A}=f^*K_A$. Thus
$\omega_{\widetilde A}\simeq\OO_{\widetilde A}$.
The classification of smooth surfaces of Kodaira dimension
zero in characteristic greater than three gives the assertion
(see \cite{BM77}).
\end{proof}

\begin{proposition}\label{lem:normal-strict-lc-cohomology}
Let $A$ be a projective Gorenstein lc surface with
$\omega_A\simeq\OO_A$. If $A$ is not klt, then
\[
 H^1(A,\FF)=0
 \quad\text{for every torsion rank-one reflexive sheaf }\FF.
\]
\end{proposition}
The proof is given in Section~\ref{subsec:normal-lc-proof}.

\subsection{Normal surfaces with anticanonical boundary}

Let $(W,C)$ be a normal projective lc surface pair such that
$C$ is reduced and $K_W+C\sim0$ is Cartier. 

\begin{lemma}\label{lem:log-surface-resolution}
For the minimal regular resolution $h:V\to W$, write
$K_V+B=h^*(K_W+C)$. Then $B$ is reduced and nodal.
Every exceptional curve not contained in $B$ is a
rational $(-2)$-curve disjoint from $B$.
\end{lemma}

\begin{proof}
\noindent\emph{Step 1. Exceptional coefficients.}
Write $B=C'+\sum b_iE_i$, where $C'$ is the strict
transform and the $E_i$ are exceptional. Cartierness and
log canonicity give $b_i\in\Z$ and $b_i\leq1$.
Since $K_V$ is $h$-nef, for every $i$,
\[
 \Bigl(\sum b_jE_j\Bigr)\cdot E_i
                    =-(K_V+C')\cdot E_i\leq0.
\]
The exceptional matrix is negative definite with
nonnegative off-diagonal entries. Its inverse has
nonpositive entries. Hence $b_i\geq0$, and $b_i\in\{0,1\}$.

\smallskip
\noindent\emph{Step 2. Curves of coefficient zero.}
If $b_i=0$, then
\[
 0=(K_V+B)\cdot E_i=K_V\cdot E_i+B\cdot E_i,
 \qquad K_V\cdot E_i\geq0,\quad B\cdot E_i\geq0.
\]
Thus both intersections vanish. Adjunction gives
$E_i^2=2p_a(E_i)-2<0$, so $p_a(E_i)=0$ and $E_i^2=-2$.
Moreover $E_i\cap B=\varnothing$.
The pair $(V,B)$ is lc; a reduced lc divisor on a smooth
surface has only nodes. This proves the lemma.
\end{proof}

\begin{lemma}\label{lem:surface-configurations}
Let $V$ be a smooth projective surface and let
$0\ne B\in|-K_V|$ be reduced and nodal. Exactly one of
the following occurs:
\begin{enumerate}
\item $V$ is rational and $B$ is a smooth elliptic curve
or a connected cycle of rational curves;
\item $V$ is ruled over an elliptic curve and
$B=C_-+C_+$ consists of two disjoint sections.
\end{enumerate}
\end{lemma}

\begin{proof}
\noindent\emph{Step 1. The components of $B$.}
For a component $D\subset B$, let $\widetilde D$ be its
normalization and let $r_D$ count the branches of the nodes
on $\widetilde D$. Adjunction gives
\begin{equation}\label{eq:configuration-adjunction}
                   2g(\widetilde D)-2+r_D=0.
\end{equation}
Consequently an elliptic component is disjoint from the
other components, and each rational component meets two
branches of $B$, counting a self-node twice.

\smallskip
\noindent\emph{Step 2. The rational case.}
For an ample divisor $H$, $K_V\cdot H=-B\cdot H<0$.
The classification of surfaces with Kodaira dimension
$-\infty$ makes $V$ rational or ruled over a curve of
positive genus \cite{BM77}.
If $V$ is rational, the sequence
\[
 0\longrightarrow\omega_V\longrightarrow\OO_V
                       \longrightarrow\OO_B\longrightarrow0
\]
and $H^1(V,\omega_V)=0$ imply $H^0(B,\OO_B)=k$.
Thus $B$ is connected, and \eqref{eq:configuration-adjunction}
gives (i).

\smallskip
\noindent\emph{Step 3. The ruling over a curve of positive genus.}
Let $f:V\to E$ be the ruling and let $F$ be a general fibre.
Since $B\cdot F=2$, some component of $B$ dominates $E$
with degree at most two. This morphism is separable because
$p>2$. Riemann--Hurwitz and
\eqref{eq:configuration-adjunction} give
$g(E)=1$ and show that every horizontal component is
elliptic and \'{e}tale over $E$.
There are no vertical components of $B$: they would form
a subgraph of a fibre tree with every vertex of valency
two, contrary to \eqref{eq:configuration-adjunction}.
Hence $B$ consists of two sections or one \'{e}tale bisection.

Suppose $B$ is one bisection. In the preceding ideal
sequence, $H^0(\OO_V)\to H^0(\OO_B)$ is an isomorphism
and $h^1(\omega_V)=h^1(\OO_V)=1$. Exactness implies that
\[
 \operatorname{res}_B:H^1(\OO_V)\longrightarrow H^1(\OO_B),
\]
is equal to zero. On the other hand this homomorphism is pullback along
$B\to E$. Its composite with trace is multiplication by
$2$, so it is injective. This contradiction proves (ii).
\end{proof}

The coefficient-zero exceptional curves in
Lemma~\ref{lem:log-surface-resolution} may be contracted.
After that contraction the surface has only rational
double points away from $B$. A component of $B$ exceptional
over $W$ is retained until its effect on cohomology is computed.

\subsection{Sheaves on a surface ruled over an elliptic curve}

\begin{lemma}\label{lem:elliptic-ruling-sheaves}
Let $f:V\to E$ be a normal projective surface ruled over
an elliptic curve. Assume that its singularities are rational
double points and that
$C_-,C_+\subset V_{\mathrm{reg}}$ are disjoint sections with
$K_V+C_-+C_+\sim0$ Cartier. Suppose
\[
 \FF^{[m]}\simeq\OO_V(-m\alpha C_--m\gamma C_+)
                         \hookrightarrow\OO_V,
 \qquad 0\leq\alpha,\gamma<1,
\]
for a rank-one reflexive sheaf $\FF$. Then
\[
 \alpha+\gamma\in\{0,1\},\qquad
 \begin{cases}
 Rf_*\FF=0,&\alpha+\gamma=1,\\
 \FF\simeq f^*L,\ Rf_*\FF\simeq L,\ L^m\simeq\OO_E,
       &\alpha=\gamma=0.
 \end{cases}
\]
\end{lemma}

\begin{proof}
\noindent\emph{Step 1. Fibres and degrees.}
On the minimal resolution every vertical $(-1)$-curve
meets exactly one of the two sections, since its
anticanonical degree is one. Successive contraction of
these curves gives a $\mathbb P^1$-bundle over $E$.
Reversing the contractions consists of blowing up a section
at a smooth point of a fibre component. Each fibre is
therefore a reduced chain of rational curves, with the
two sections on its end components.
Contraction of the disjoint $(-2)$-curves preserves
reducedness and $h^*(\OO_{V_e})=(1,0)$: use
$Rg_*\OO=\OO$ and the exact sequence for multiplication
by a uniformizer at $e\in E$.

The sheaf $\FF$ is flat over $E$, since a uniformizer
is a nonzerodivisor on $\FF$. Its restriction $\FF_e$
is torsion-free of rank one on each component of $V_e$.
The degree on the generic fibre and flatness give
\[
 -\alpha-\gamma\in\Z,\qquad
 \chi(V_e,\FF_e)=1-\alpha-\gamma\in\{0,1\}.
\]

\smallskip
\noindent\emph{Step 2. Sections on a fibre.}
For $0\ne s\in H^0(V_e,\FF_e)$, the specified power map
gives $s^m\in H^0(V_e,\OO_{V_e})=k$.
This power is nonzero: $\FF_e$ is torsion-free and the
power map is nonzero at the generic point of each fibre
component. If $\alpha+\gamma=1$, the power map vanishes
at $C_-\cap V_e$ and $C_+\cap V_e$, a contradiction.
Thus $h^0(\FF_e)=h^1(\FF_e)=0$ in that case.

If $\alpha=\gamma=0$, the equality $\chi(\FF_e)=1$
supplies a nonzero section $s$. Since $s^m\in k^*$,
multiplication by $s^{m-1}/s^m$ is an inverse to
$\OO_{V_e}\xrightarrow{s}\FF_e$.
Indeed it splits this map; the complementary module has
rank zero on every component and is torsion-free, hence zero.
Therefore
\[
                    \FF_e\simeq\OO_{V_e},\qquad h^*(\FF_e)=(1,0).
\]

\smallskip
\noindent\emph{Step 3. Relative cohomology.}
Cohomology and base change now give $Rf_*\FF=0$ in the
first case. In the second, $L=f_*\FF$ is invertible,
$R^1f_*\FF=0$, and the evaluation $f^*L\to\FF$ is an
isomorphism on every fibre, hence an isomorphism.
Restriction to $C_-$ gives $L^m\simeq\OO_E$.
\end{proof}

\subsection{Vanishing with fractional anticanonical coefficients}
\label{subsec:surface-witt}

This subsection proves the vanishing needed when the order
of a divisorial sheaf is divisible by $p$.
All splittings are constructed on a normal surface $W$.

For a $\mathbb Q$-divisor $T$ on $W$, set
$\OO_W(T)=\OO_W(\lfloor T\rfloor)$ and
\[
 W_n\OO_W(T)=
 \{(a_0,\ldots,a_{n-1})\in W_nk(W) \mid 
             a_i\in\OO_W(\lfloor p^iT\rfloor)\}.
\]
Define $Q_{T,n}$ by the pushout
\begin{equation}\label{eq:Q-pushout}
\begin{tikzcd}[column sep=large]
 W_n\OO_W(T)\ar[r,"F"]\ar[d,"R^{n-1}"']
 &F_*W_n\OO_W(pT)\ar[d]\\
 \OO_W(T)\ar[r,"\Phi_{T,n}"']&Q_{T,n},
\end{tikzcd}
\end{equation}
and set $B_{T,n}=\operatorname{coker}(F)$. The following are exact sequences of coherent $\OO_W$-modules
\cite[\S3.1 and Lemma 3.8]{KTT+24}:
\begin{align}
0&\to\OO_W(T)\xrightarrow{\Phi_{T,n}}Q_{T,n}\to B_{T,n}\to0,
                                      \label{eq:witt-Q-sequence}\\
0&\to F_*B_{pT,n}\to Q_{T,n+1}\to F_*\OO_W(pT)\to0,
                                      \label{eq:witt-Q-recursion}\\
0&\to F_*B_{pT,n}\to B_{T,n+1}\to B_{T,1}\to0.
                                      \label{eq:witt-B-recursion}
\end{align}
For $0\leq\Gamma<1$, the pair $(W,\Gamma)$ is
$n$-quasi-$F$-split if $\Phi_{\Gamma,n}$ has an
$\OO_W$-linear splitting. Splitting depends only on
$\{T\}$, by \cite[Lemma 3.11]{KTT+24}.

\begin{lemma}\label{lem:conductor-Witt-splitting}
Let $W$ be a normal proper Cohen--Macaulay surface and
$0\ne C=\sum C_i$ a reduced divisor. Assume
\[
 \II_C\simeq\omega_W,\qquad H^1(W,\OO_W)=H^2(W,\OO_W)=0.
\]
If $C$ is $n$-quasi-$F$-split, then $\Phi_{T,n}$ splits
whenever $\{T\}=\sum\gamma_iC_i$ with $0\leq\gamma_i<1$.
\end{lemma}

\begin{proof}
\noindent\emph{Step 1. Splitting on the canonical ideal.}
Let $W_n\II_C=\ker(W_n\OO_W\to W_n\OO_C)$ and let
$Q_n^C$ be the pushout of its Frobenius map along the
first-coordinate map $W_n\II_C\to\II_C$.
This is the Witt ideal: every coordinate is in $\II_C$.
It is not $W_n\OO_W(-C)$.
The Witt ideal sequence induces the exact diagram
\begin{equation}\label{eq:conductor-Q-diagram}
\begin{tikzcd}[column sep=large]
0\ar[r]&\II_C\ar[r]\ar[d,"\Phi_n^C"']
 &\OO_W\ar[r]\ar[d,"\Phi_{0,n}"]
 &\OO_C\ar[r]\ar[d,"\Phi_{C,n}"]&0\\
0\ar[r]&Q_n^C\ar[r]&Q_{0,n}\ar[r]&Q_{C,n}\ar[r]&0.
\end{tikzcd}
\end{equation}
Here $Q_{C,n}$ is formed on the curve $C$.
For exactness of the bottom row, quotient the Witt ideal
sequence by $VF(W_{n-1}(-))$. The intersection with the
ideal term is $VF(W_{n-1}\II_C)$, because Frobenius on
the reduced ring $\OO_C$ is injective.

The sequences above and $H^1(\OO_W)=H^2(\OO_W)=0$ give,
by induction on $n$,
\[
 H^1(B_{0,n})=H^2(B_{0,n})=H^1(Q_{0,n})=0.
\]
The connecting homomorphisms in
\eqref{eq:conductor-Q-diagram} therefore satisfy
\[
 \begin{gathered}
 \delta:H^1(\OO_C)\xrightarrow{\sim}H^2(\II_C),\qquad
 \delta_n:H^1(Q_{C,n})\hookrightarrow H^2(Q_n^C),\\
 H^2(\Phi_n^C)\delta=\delta_nH^1(\Phi_{C,n}).
 \end{gathered}
\]
The assumed splitting on $C$ makes $H^2(\Phi_n^C)$
injective. Serre duality consequently makes
\[
 \Hom_W(Q_n^C,\omega_W)\longrightarrow
                         \Hom_W(\II_C,\omega_W)
\]
surjective. A preimage of $1_{\II_C}$ gives a splitting
$r_n:Q_n^C\to\II_C$.

\smallskip
\noindent\emph{Step 2. Fractional coefficients and integral twists.}
Work first on $W_{\mathrm{reg}}$ and set $\Gamma=\{T\}$.
On an open subset $U$, choose an equation $t$ for $C|_U$.
For $0\leq i<n$ and $0\leq\gamma_j<1$,
\[
 \lfloor p^i\gamma_j\rfloor\leq p^i-1,
 \qquad\lfloor p^{i+1}\gamma_j\rfloor\leq p^{i+1}-1.
\]
Thus multiplication by the Teichm\"uller representatives
$[t]$ and $[t^p]$ gives a commutative diagram
\[
\begin{tikzcd}
W_n\OO_W(\Gamma)\ar[r,"F"]\ar[d,"{[t]}"']
 &F_*W_n\OO_W(p\Gamma)\ar[d,"{[t^p]}"]\\
W_n\II_C\ar[r,"F"']&F_*W_n\II_C.
\end{tikzcd}
\]
On first coordinates the left vertical map is
$\OO_W\xrightarrow{t}\II_C$.
The induced map of pushouts
$\psi_t:Q_{\Gamma,n}|_U\to Q_n^C|_U$ satisfies
$\psi_t\Phi_{\Gamma,n}=t\Phi_n^C$. Therefore
\[
 \rho_t:=t^{-1}r_n\psi_t:Q_{\Gamma,n}|_U\longrightarrow\OO_U,
 \qquad \rho_t\Phi_{\Gamma,n}=1_{\OO_U}.
\]
If $t'=ut$ for a unit $u\in\OO_U^*$, the two vertical
maps change by $[u]$ and $[u^p]$, respectively.
On the Frobenius term, multiplication by $[u^p]$ is
the scalar action of $u$. Thus
\[
 \psi_{ut}=u\psi_t,\qquad
 \rho_{ut}=(ut)^{-1}r_n(u\psi_t)=t^{-1}r_n\psi_t=\rho_t.
\]
The retractions consequently glue on $W_{\mathrm{reg}}$.
If $j:W_{\mathrm{reg}}\hookrightarrow W$, a homomorphism
to $\OO_{W_{\mathrm{reg}}}$ extends by adjunction to
$j_*\OO_{W_{\mathrm{reg}}}=\OO_W$.
This gives the splitting on $W$.
Finally apply the integral Weil-twist assertion
\cite[Lemma 3.11]{KTT+24} to $T=\lfloor T\rfloor+\Gamma$.
\end{proof}

\begin{corollary}\label{lem:conductor-Witt-vanishing}
Under Lemma~\ref{lem:conductor-Witt-splitting}, assume
$\Supp\{p^iT\}\subseteq C$ for $0\leq i\leq a$.
Then, for every $j$,
\[
 H^j(W,Q_{p^aT,n})=0\quad\Longrightarrow\quad
                            H^j(W,\OO_W(T))=0.
\]
\end{corollary}

\begin{proof}
Set $N_i=n+a-i$. Suppose
$H^j(Q_{p^{i+1}T,N_{i+1}})=0$.
The splitting lemma, and monotonicity in Witt length,
make both $\OO_W(p^{i+1}T)$ and $B_{p^{i+1}T,N_{i+1}}$
direct summands of $Q_{p^{i+1}T,N_{i+1}}$.
Sequence~\eqref{eq:witt-Q-recursion} gives
$H^j(Q_{p^iT,N_i})=0$.
Descending induction starts at $i=a$.
Splitting \eqref{eq:witt-Q-sequence} at $i=0$ proves the claim.
\end{proof}

\subsection{Rational surfaces with elliptic boundary}

Fix a normal projective rational surface $W$ and a smooth
elliptic curve $E$ such that $(W,E)$ is plt and
$K_W+E\sim0$ is Cartier. Adjunction gives
$\operatorname{Diff}_E(0)=0$; hence $W$ is regular along
$E$ and has only rational double points elsewhere.
In particular, restrictions of reflexive sheaves to $E$
are restrictions of invertible sheaves near $E$.

\begin{proposition}\label{lem:rational-elliptic-sheaves}
Suppose
$\FF^{[m]}\simeq\OO_W(-bE)\hookrightarrow\OO_W$,
where $0\leq b<m$ and $\FF$ is rank-one reflexive.
If $H^1(W,\FF)\ne0$, then
\[
 b=0,\qquad\FF^{[2]}\simeq\OO_W,\qquad
 \FF\not\simeq\OO_W,\qquad h^*(\FF)=(0,1,0).
\]
There is at most one such sheaf; denote it by $\varepsilon$.
Its quasi-\'{e}tale double cover is ruled over an elliptic
curve, with two boundary sections mapping isomorphically to $E$.
\end{proposition}

\begin{proof}
\noindent\emph{Step 1. Covers of degree prime to $p$.}
Let $T$ be a rational Weil divisor with
$rT=\divisor(h)$, $(r,p)=1$, and $\{T\}=\alpha E$,
$0\leq\alpha<1$. In the normalized cyclic cover,
$\OO_W(\lfloor lT\rfloor)$ is the character-$l$ summand
of the algebra \cite[Proposition 3.1]{Amb15}.
Choose a connected component $W'$ with stabilizer
$G_0\subseteq\mu_r$, and set $E'=(\pi^{-1}E)_{\mathrm{red}}$.
Tame ramification and adjunction give
\[
 K_{W'}+E'=\pi^*(K_W+E),\qquad E'\longrightarrow E
                          \text{ \'{e}tale on each component}.
\]
By Lemmas~\ref{lem:log-surface-resolution} and
\ref{lem:surface-configurations}, $W'$ is rational or
ruled over an elliptic curve; in the latter case $E'$
has two sections. Thus $H^2(W',\OO)=0$ and $h^1(W',\OO)\leq1$.

In the ruled case, $G_0$ permutes the two sections
transitively. Their stabilizer acts by translations,
and restriction from $W'$ to either section is an
isomorphism on $H^1$. Hence the representation on
$H^1(W',\OO)$ is the nontrivial character of the
quotient of order two; its invariant part is zero
because $H^1(W,\OO)=0$.
For every $j\geq0$, the character $p^j$ is faithful on
$G_0$. It contributes to $H^1$ only if $|G_0|=2$.
Then the section stabilizer and the ramification along
$E$ are trivial, so $\alpha=0$ and
$\OO_W(T)$ has order two. This argument also applies
when the cyclic algebra is disconnected, by induction
of the representation from $G_0$ to $\mu_r$.

\smallskip
\noindent\emph{Step 2. Reduction of an arbitrary order.}
Write $m=p^ar$ with $(r,p)=1$, choose $L$ with
$\FF=\OO_W(L)$, and set
\[
 mL+bE=\divisor(h),\qquad T=L+(b/m)E,\qquad T'=p^aT.
\]
An elliptic curve has quasi-$F$-splitting height at most
two \cite[Remark 2.12]{KTT+24}.
Corollary~\ref{lem:conductor-Witt-vanishing}, with $n=2$, gives
\[
 H^1(W,\FF)\ne0\quad\Longrightarrow\quad H^1(W,Q_{T',2})\ne0.
\]
Set $G_j=\OO_W(\lfloor p^jT'\rfloor)$.
If $H^1(G_j)=H^2(G_j)=0$ for all $j\geq0$, then
$0\to G_j\to F_*G_{j+1}\to B_{p^jT',1}\to0$
and \eqref{eq:witt-Q-recursion} give $H^1(Q_{T',2})=0$.
Step 1 gives $H^2(G_j)=0$ and characterizes every
possible nonzero $H^1(G_j)$.
Consequently $T'$ is integral and $\OO_W(T')$ is an
order-two sheaf defining a quasi-\'{e}tale double cover
$\pi:W'\to W$ with two elliptic boundary sections $E_1,E_2$.

\smallskip
\noindent\emph{Step 3. Elimination of the $p$-primary part.}
Integrality of $T'$ implies $b=rt$, $0\leq t<p^a$.
For $\FF'=(\pi^*\FF)^{**}$,
\[
 (\FF')^{[p^a]}\simeq\OO_{W'}(-t(E_1+E_2)),\qquad
                         \deg(\FF'|_{\mathbb P^1})=-2t/p^a\in\Z.
\]
Since $p\ne2$, these inequalities force $t=0$.
Lemma~\ref{lem:elliptic-ruling-sheaves} gives
$\FF'=f^*N$ for a $p$-primary torsion line bundle on
the elliptic base. Tame trace makes $\FF$ a direct
summand of $\pi_*\FF'$.
A nontrivial $N$ is acyclic, so $H^1(W,\FF)\ne0$
forces $N\simeq\OO$.
The two linearizations of $\OO_{W'}$ descend to
$\OO_W$ and the sign summand $\varepsilon$.
Since $H^1(W,\OO_W)=0$, one has $\FF\simeq\varepsilon$.

\smallskip
\noindent\emph{Step 4. The other cohomology groups and uniqueness.}
A nonzero section of $\FF$ has nonzero constant
$m$-th power. It therefore forces $b=0$ and
$\FF\simeq\OO_W$. Serre duality gives $H^2(W,\FF)=0$,
because the $m$-th reflexive power of
$\sheafHom(\FF,\omega_W)$ is $\OO_W(-(m-b)E)$.
Step 1 identifies the sign summand of $H^1(W',\OO)$
with $k$, proving $h^*(\varepsilon)=(0,1,0)$.
Two distinct order-two sheaves with nonzero $H^1$
would give a connected cover with group $(\Z/2)^2$
and two distinct nonzero character summands in $H^1$.
Lemma~\ref{lem:surface-configurations} bounds that
dimension by one. Hence $\varepsilon$ is unique.
\end{proof}

\begin{corollary}\label{cor:exceptional-restriction}
If $\varepsilon$ exists, then
\[
 \varepsilon|_E\simeq\OO_E,\qquad
 H^1(W,\varepsilon)\xrightarrow{\sim}H^1(E,\OO_E).
\]
For any nontrivial torsion rank-one reflexive sheaf $\FF$,
$\FF|_E\simeq\OO_E$ if and only if $\FF\simeq\varepsilon$.
\end{corollary}

\begin{proof}
On the double cover, restriction from $W'$ to either
section is an isomorphism on $H^1$.
Taking the sign summand of restriction to $E_1\amalg E_2$
gives the first assertion.
Conversely, if $\FF|_E\simeq\OO_E$, the sequence
\[
 0\to\FF(-E)\to\FF\to\OO_E\to0
\]
and $H^0(W,\FF)=0$ give $H^1(W,\FF(-E))\ne0$.
Duality gives $H^1(W,\FF^\vee)\ne0$.
The proposition implies $\FF^\vee\simeq\varepsilon$,
hence $\FF\simeq\varepsilon$.
\end{proof}

\subsection{Contraction of elliptic sections}
\label{subsec:normal-lc-proof}

\begin{lemma}\label{lem:contracted-elliptic-sheaves}
Let $V,E,C_-,C_+$ be as in
Lemma~\ref{lem:elliptic-ruling-sheaves}.
Let $h:V\to W$ contract one or both sections, and let
$C$ be the image of the uncontracted sections.
Suppose $\FF\not\simeq\OO_W$ is rank-one reflexive and
\[
 \FF^{[m]}\simeq\OO_W\Bigl(-\sum m\alpha_i C_i\Bigr),
                         \qquad 0\leq\alpha_i<1.
\]
Then $H^*(W,\FF)=0$. If exactly one section is contracted,
then $h^*(W,\OO_W)=(1,0,0)$.
\end{lemma}

\begin{proof}
\noindent\emph{Step 1. Pullback of the divisorial sheaf.}
Write $\FF=\OO_W(L)$ and set
$\FF_V=\OO_V(\lfloor h^*L\rfloor)$.
The power relation makes $L$ $\Q$-Cartier.
Rounding its pullback gives a power relation for
$\FF_V$ with coefficients in $[0,1)$ on $C_-,C_+$.
For every rational function $s$,
\[
 \divisor(s)+\lfloor h^*L\rfloor\geq0
 \ \Longleftrightarrow\ \divisor(s)+h^*L\geq0.
\]
Effectivity of a $\Q$-Cartier divisor is equivalent to
effectivity of its pullback. Hence $h_*\FF_V=\FF$.
Lemma~\ref{lem:elliptic-ruling-sheaves} gives either
$Rf_*\FF_V=0$ or $Rf_*\FF_V=L$ for a torsion line
bundle $L$ on $E$. If $L\not\simeq\OO_E$, then
$H^*(E,L)=0$. If $L\simeq\OO_E$, the same lemma gives
$\FF_V\simeq\OO_V$ and hence $\FF\simeq\OO_W$,
contrary to the hypothesis. Thus $\FF_V$ is acyclic.
The Leray injection proves $H^1(W,\FF)=0$.
The power relation gives $H^0(W,\FF)=0$; duality gives
$H^2(W,\FF)=0$, using $\omega_W=\OO_W(-C)$.

\smallskip
\noindent\emph{Step 2. The structure sheaf.}
Let $C_0$ be a contracted section and choose an affine
neighbourhood $U$ of $h(C_0)$.
The pullback and restriction maps give a commutative square
\[
\begin{tikzcd}
H^1(W,\OO_W)\ar[r,hook]\ar[d]
 &H^1(V,\OO_V)\ar[d,"\operatorname{res}_{C_0}","\sim"']\\
H^1(U,\OO_U)=0\ar[r]&H^1(C_0,\OO_{C_0}).
\end{tikzcd}
\]
Thus $H^1(W,\OO_W)=0$.
If $C\ne0$, then
$H^2(W,\OO_W)^*=H^0(W,\OO_W(-C))=0$.
\end{proof}

\begin{proof}[Proof of Proposition~\ref{lem:normal-strict-lc-cohomology}]
\noindent\emph{Step 1. The pullback and the restriction map.}
Resolve the non-klt points of $A$ and contract the
coefficient-zero $(-2)$-curves from
Lemma~\ref{lem:log-surface-resolution}.
Write $h:V\to A$ for the resulting morphism and $B$ for
its nonempty reduced exceptional divisor. Then
\[
 K_V+B\sim0,\quad \FF=\OO_A(L),\quad
 \FF_V=\OO_V(\lfloor h^*L\rfloor),\quad h_*\FF_V=\FF,
 \quad\iota:H^1(A,\FF)\hookrightarrow H^1(V,\FF_V).
\]
The proof of the preceding lemma gives $h_*\FF_V=\FF$.
For an elliptic component $E\subset B$, restriction
through an affine neighbourhood of $h(E)$ gives
\begin{equation}\label{eq:exceptional-restriction-zero}
                     \operatorname{res}_E\circ\iota=0.
\end{equation}

\smallskip
\noindent\emph{Step 2. A rational surface with rational boundary.}
Suppose $V$ is rational and $B$ is a rational nodal
cycle. The curve $B$ is $F$-split: the splittings of
its normalizations $\mathbb P^1$ compatible with $0+\infty$
glue at the nodes.
Write $\FF_V=\OO_V(\lfloor T\rfloor)$ with
$\Supp\{T\}\subseteq B$ and $mT$ principal, and set
$m=p^ar$. A normalized degree-$r$ cyclic cover for
$p^aT$ is log crepant with reduced boundary.
After resolving the boundary, its components over
each rational curve are tame covers of $\mathbb P^1$
ramified only at the two nodes; Riemann--Hurwitz
makes them rational. The smooth surface is rational
by Lemma~\ref{lem:surface-configurations}.
Every character summand therefore has zero $H^1,H^2$.
Sequences \eqref{eq:witt-Q-sequence}--\eqref{eq:witt-B-recursion}
give $H^1(Q_{p^aT,1})=0$.
Corollary~\ref{lem:conductor-Witt-vanishing} gives
$H^1(V,\FF_V)=0$.

\smallskip
\noindent\emph{Step 3. An elliptic boundary on a rational surface.}
Proposition~\ref{lem:rational-elliptic-sheaves} leaves
only $\FF_V\simeq\varepsilon$ when $H^1(V,\FF_V)\ne0$.
By Corollary~\ref{cor:exceptional-restriction},
$\operatorname{res}_B$ is an isomorphism.
Equation~\eqref{eq:exceptional-restriction-zero} gives $\iota=0$.

\smallskip
\noindent\emph{Step 4. A surface ruled over an elliptic curve.}
Both sections are exceptional. By
Lemma~\ref{lem:elliptic-ruling-sheaves}, nonzero
$H^1(V,\FF_V)$ requires $\FF_V\simeq\OO_V$.
Restriction to either section is an isomorphism on
$H^1$, and \eqref{eq:exceptional-restriction-zero}
again gives $\iota=0$.
Lemma~\ref{lem:surface-configurations} exhausts the cases.
\end{proof}

\section{Nonnormal surfaces: conductor and cohomology}
\label{app:models}

The field in section is assumed to be algebraically closed of characteristic $p>42$.
We first describe the conductor of a connected projective
nonnormal Gorenstein slc surface $A$ with $\omega_A\simeq\OO_A$.
For rational conductor, the main conclusion is
Proposition~\ref{cor:rational-conductor-picard}: the structure-sheaf
cohomology is that of a sphere or a torus, and the connected
Picard scheme is respectively trivial or $\Gm^2$.
The additional hypothesis used for this conclusion is the
triviality of a Cartier multiple of the full pullback of
the original threefold divisor. Its multiplicities are
retained throughout the proof.

\subsection{Normalization and the conductor}

Write
\[
 \nu:A^\nu=\coprod_{i\in I}V_i\longrightarrow A,
 \qquad C\subset A,\qquad
 \overline C=(\nu^{-1}C)_{\mathrm{red}}.
\]
Here $C$ is the conductor subscheme. Write $C_i$ for its
reduced inverse image on $V_i$, and $\tau$ for the
involution on the normalization of $\overline C$.
Normalized adjunction gives
\begin{equation}\label{eq:normalized-anticanonical}
                  K_{V_i}+C_i\sim0\quad\text{Cartier}.
\end{equation}
The restriction of a generator of $\omega_A$ has opposite
residues on the two branches identified by $\tau$.

\begin{lemma}\label{lem:nonnormal-conductor-types}
Exactly one of the following occurs:
\begin{enumerate}
\item every component of the normalized conductor is
rational, and every $V_i$ is rational;
\item every component of the normalized conductor is
elliptic. The components $V_i$ and their conductor curves
form a chain or a cycle.
\end{enumerate}
In \textup{(ii)}, a component with two conductor curves
is ruled over an elliptic curve with these curves as
sections. A component with one conductor curve is
rational or is obtained by contracting one section of
such a ruling.
\end{lemma}

\begin{proof}
\noindent\emph{Step 1. The possibilities on a component.}
Apply Lemmas~\ref{lem:log-surface-resolution} and
\ref{lem:surface-configurations} to \eqref{eq:normalized-anticanonical}.
Each conductor curve is part of the reduced anticanonical
divisor on the resolution. Its normalized genus is zero
or one. On a fixed $V_i$ the genus is the same for all
conductor curves. In genus zero the resolution is rational.
In genus one it has one elliptic boundary on a rational
surface, or two sections on an elliptic ruling. Exceptional
sections are contracted, leaving one or two conductor curves.

\smallskip
\noindent\emph{Step 2. Connectedness along curves.}
The two curves paired by $\tau$ are isomorphic.
The graph whose vertices are the $V_i$ and whose edges
are paired conductor curves is connected: otherwise $A$
would be a union of two closed subschemes meeting only
in finitely many points, contradicting $S_2$.
The common genus therefore agrees on all components.
This proves the dichotomy.

\smallskip
\noindent\emph{Step 3. The graph in the elliptic case.}
If $\tau$ preserves an elliptic curve $E$, its action
on $H^0(E,\omega_E)$ is $-1$ by the residue condition.
In characteristic greater than three, such an involution
is a translate of inversion. Hence $E/\tau\simeq\mathbb P^1$.
For each such curve add one auxiliary vertex and join it
to the vertex of the component containing $E$.
Every vertex now has degree at most two. The graph is
connected, so it is a chain or a cycle.
In a cycle all vertices correspond to surfaces with two
sections. In a chain the two vertices of degree one are
either surfaces with one conductor curve or the specified
auxiliary vertices. This also defines the two ends used
in Section~\ref{subsec:elliptic-conductor}.
\end{proof}

\subsection{The statement for rational conductor}
\label{subsec:rational-main}

We specify the formal schemes to which the following
result applies. Start with $V,T$ as in
Proposition~\ref{prop:surface-cover} and take the logarithmic
cover of Lemma~\ref{lem:formal-log-cover}.
One may further take a normalized cyclic cover defined by
\begin{equation}\label{eq:additional-tame-algebra}
 \bigoplus_{l=0}^{r-1}\II_{dl}\mathcal U^{-l},
 \qquad (r,p)=1,\qquad
 \mathcal U^r\simeq\II_{dr},\qquad\mathcal U|_A\simeq\OO_A.
\end{equation}
Here $d,r>0$, $drA$ is Cartier, and $\II_{dr}|_A\simeq\OO_A$.
Choose a connected component whenever necessary, and
denote the resulting formal pair again by $(\YY,A)$.
These are precisely the covers used in
Section~\ref{subsec:rational-conductor}.

\begin{proposition}\label{cor:rational-conductor-picard}
Suppose $A$ has rational conductor in the preceding
construction. There exists a closed oriented surface
$\Delta$, homeomorphic to $S^2$ or $S^1\times S^1$, such that
\begin{equation}\label{eq:rational-cohomology}
 H^*(A,\OO_A)\simeq H^*(\Delta,k),\qquad
                 \Pic^0_{A/k}\simeq\Gm^{b_1(\Delta)}.
\end{equation}
The cohomology isomorphism respects cup products and
Frobenius. In particular, Frobenius is bijective on
$H^*(A,\OO_A)$.
\end{proposition}

We compute both invariants on one crepant model
$g:B\to A$. Its components are rational surfaces, its
double curves are smooth rational curves, and
$Rg_*\OO_B\simeq\OO_A$. This last equality preserves
structure-sheaf cohomology and, as proved below, the
connected Picard scheme. Alternating restrictions and
gluing constants then compute the two invariants from
the same dual complex. The threefold hypothesis enters
through $P\cdot C=0$ for each double curve $C$, where $P$
is the full pullback of $T$. Hodge index turns these
equalities into $b_1(\Delta)\leq2$.

\subsection{Crepant birational morphisms}

The following local lemma ensures that cyclic covers near
double curves are covers of simple normal crossing pairs.

\begin{lemma}\label{lem:double-locus-regularity}
Let $(Y,D)$ be a dlt threefold pair with $Y$ klt,
let $K_Y+D$ be Cartier, and let all components of $D$ be  $\Q$-Cartier.
If
\[
                  \widehat\OO_{D,x}\simeq k[[u,v,z]]/(uv),
\]
then $Y$ is regular at $x$ and $(Y,D)$ is snc at $x$.
\end{lemma}

\begin{proof}
Set $R=\widehat\OO_{Y,x}$ and $B=R/\II_D$.

\noindent\emph{Step 1. A local equation for $D$.}
The module $\FF_1=\II_D/\OO_Y(-2D)$ is maximal
Cohen--Macaulay over $B$, by \cite[Theorems 3 and 19(2)]{BK23}.
On the punctured spectrum $U=\Spec B\setminus\{\mathfrak m\}$,
it is invertible: this follows from regularity at smooth
points of $D$ and from the snc property at the generic
point of its double curve.
The two branches of the normalization have local rings
$k[[u,z]]$ and $k[[v,z]]$. Their punctured spectra have
trivial Picard groups. The conductor square gives
\begin{equation}\label{eq:punctured-node-picard}
\Pic(U)=k((z))^*/
 \bigl(k[[z]]^*k[[z]]^*\bigr)\xrightarrow[\ v_z\ ]{\sim}\Z.
\end{equation}
The power relation for $D$ makes $\FF_1|_U$ torsion,
so it is trivial. Extension across the closed point by
$S_2$ gives $\FF_1\simeq\OO_D$.
Lemma~\ref{lem:quotient-duality} makes $D$ Cartier; write
$\II_D=(t)$.

\smallskip
\noindent\emph{Step 2. The equation of the ambient ring.}
If $R$ is not regular, its embedding dimension is four.
It is a hypersurface
$k[[u,v,z,t]]/(f)$ with $f\bmod t=uv$.
The formal splitting lemma for the nondegenerate quadratic
form $uv$ over $k[[z,t]]$ gives
\begin{equation}\label{eq:double-ambient-equation}
                   R=k[[u,v,z,t]]/(uv-th(z,t)).
\end{equation}
The generic snc condition gives $h(z,0)\ne0$.
If $h$ is a unit, the equation has a nonzero $t$-derivative
and $R$ is regular. Assume instead $h\in(z,t)$.

\smallskip
\noindent\emph{Step 3. The two boundary components.}
Their ideals are $(u,t)$ and $(v,t)$.
The two affine charts of the blowup of $(u,t)$ are
\begin{equation}\label{eq:double-blowup-charts}
 t=u\tau,\quad v=\tau h(z,u\tau);
 \qquad u=t\sigma,\quad \sigma v=h(z,t).
\end{equation}
The first chart is regular. The second is a hypersurface
domain whose singular locus is contained in
$\{\sigma=v=h(z,t)=0\}$, of codimension at least two;
it is normal. The exceptional fibre is $\mathbb P^1$,
there is no exceptional divisor, and the invertible
ideal $(u,t)\OO$ has degree one on this fibre.
If $qD_1$ were Cartier for $D_1=(u,t)$, its inverse
line bundle would pull back with degree zero. Off the
exceptional fibre this pullback equals $((u,t)\OO)^q$.
Normality extends the isomorphism across codimension two,
giving the contradiction $0=q$ for their degrees.
Thus $h$ is a unit and $(Y,D)$ is snc.
\end{proof}

\begin{lemma}\label{lem:rational-crepant-modification}
In Section~\ref{subsec:rational-main}, there is a projective
birational morphism of formal pairs
$g:(\mathfrak W,B)\to(\YY,A)$ with the following properties:
\begin{enumerate}
\item $B_i$ is normal rational for every component of $B$;
its singularities are rational double points away from
the double curves, which are smooth rational curves.
The ambient pair is snc along all double curves and triple points.
\item $K_{\mathfrak W}+B=g^*(K_{\YY}+A)$, and
$R(g|_B)_*\OO_B\simeq\OO_A$.
\item The full pullback $P$ of $T$ satisfies
\[
 P=\sum_i r_iB_i,\quad r_i>0,\quad
             NP\text{ Cartier},\quad\OO_{\mathfrak W}(NP)|_B\simeq\OO_B.
\]
\end{enumerate}
\end{lemma}

\begin{proof}
The target $(\YY,A)$ is obtained from the original pair
$(V,T)$ by the following two finite covers:
\[
 (\YY,A)\xrightarrow{\pi}(\YY_0,A_0)
             \xrightarrow{h_0}(\widehat V_T,T).
\]
Here $h_0$ is the cover of Lemma~\ref{lem:formal-log-cover},
and $\pi$ is defined by \eqref{eq:additional-tame-algebra},
with $\II_j=\OO_{\YY_0}(-jA_0)$. If there is no additional
cover, take $r=1$ and $\pi=\operatorname{id}$.
The ramification formula of Lemma~\ref{lem:normalized-cyclic-algebra}
shows that $\pi$ is quasi-\'{e}tale. Thus both maps are
quasi-\'{e}tale, and
\[
 A_0=h_0^*T,\qquad A=\pi^*A_0
\]
are reduced. We take a dlt modification of $(V,T)$ and
pull it back through these covers. Only the second base
change can create singularities along double curves;
there it is a normalized root of a monomial, which has
a log crepant toroidal resolution.

\smallskip
\noindent\emph{Step 1. A modification of the log canonical cover.}
Shrink $V$ to the neighbourhood in
\eqref{eq:geometric-hypotheses}. By
\cite[Corollary 3.6]{HNT19}, there is a projective
$\Q$-factorial dlt modification $f:(W_0,B_0)\to(V,T)$.
Write
\begin{equation}\label{eq:weighted-pullback}
 K_{W_0}+B_0=f^*(K_V+T),\qquad
 P_0:=f^*T=\sum_i\lambda_iB_{0,i},\quad\lambda_i>0.
\end{equation}
For a strict transform $\lambda_i=1$. For an exceptional
component, crepancy gives
$\lambda_i=A_{B_{0,i}}(V,0)>0$, since $V$ is klt.

Complete $W_0$ along $B_0$, and normalize the components of
$\widehat W_0\times_{\widehat V_T}\YY_0$ that dominate
$\YY_0$. Denote the resulting maps by
\[
\begin{tikzcd}[column sep=large]
 \mathfrak W_1\arrow[r,"h_1"]\arrow[d,"g_1"']
      &\widehat W_0\arrow[d,"\widehat f"]\\
 \YY_0\arrow[r,"h_0"']&\widehat V_T.
\end{tikzcd}
\]
Set $P_1=g_1^*A_0=h_1^*P_0$ and $B_1=(P_1)_{\mathrm{red}}$.
The map $g_1$ is projective and birational. Let $q_0$ and
$\mathcal U_0$ be the integer and invertible sheaf used to
define $h_0$ in Lemma~\ref{lem:formal-log-cover}.
On the snc locus of $(W_0,B_0)$, all sheaves in the next
display are restricted to that locus; set
\[
 \mathcal L_1=\OO_{\widehat W_0}(K_{W_0}+B_0)
                         \otimes\widehat f^*\mathcal U_0^{-1},
 \qquad\mathcal L_1^{q_0}\simeq\OO_{\widehat W_0}.
\]
Crepancy identifies its generic cyclic algebra with that
defining $h_1$. Since $\mathcal L_1$ is invertible,
$\bigoplus_{j=0}^{q_0-1}\mathcal L_1^j$ is already finite
\'{e}tale and normal. Uniqueness of normalization identifies
the corresponding component with $h_1$. In particular,
$h_1$ is quasi-\'{e}tale and
$\mathfrak W_1$ is klt. Tame discrepancy comparison gives
a dlt pair with
\[
 K_{\mathfrak W_1}+B_1=g_1^*(K_{\YY_0}+A_0)
                         \quad\text{Cartier}.
\]

Every component of $B_1$ is $\Q$-Cartier. Indeed
$(W_0,B_{0,i})$ is plt, so its reduced inverse image under
$h_1$ is normal by tame discrepancy comparison and plt
adjunction. Its components are disjoint, and their sum is
the pullback of the $\Q$-Cartier divisor $B_{0,i}$.

Adjunction gives Gorenstein dlt surface pairs on the
components of $B_1$. Each component is regular along its
conductor and has rational double points elsewhere.
At a closed point on exactly two components, their conductor
gluing has completed local ring $k[[u,v,w]]/(uv)$;
Lemma~\ref{lem:double-locus-regularity} makes the ambient
pair snc there. Triple points are zero-dimensional lc
strata and hence lie in the snc locus of the dlt pair.
Thus $(\mathfrak W_1,B_1)$ is snc on a neighbourhood of
\[
             \Sigma_1:=\bigcup_{i\ne j}(B_{1,i}\cap B_{1,j}).
\]

\smallskip
\noindent\emph{Step 2. The additional cover and its resolution.}
On an affine formal open $\mathfrak U=\Spf R_0\subset\YY_0$,
choose a generator $e$ of $\mathcal U$. The defining
isomorphism in \eqref{eq:additional-tame-algebra} gives
\[
 e^r\longmapsto s,\qquad
 s\in R_0,\qquad\divisor(s)=drA_0|_{\mathfrak U}.
\]
Thus $s$ generates the ideal $\OO(-drA_0)$; it is not a
unit along $A_0$. Over $K_0=\operatorname{Frac}R_0$, set
$w=e^{-1}$. The cyclic algebra and its generic equation are
\[
 \bigoplus_{l=0}^{r-1}\OO_{\mathfrak U}(-dlA_0)w^l,
 \qquad w^r=s^{-1},\qquad z=w^{-1},\quad z^r=s.
\]
Since $\lfloor l\divisor(s^{-1})/r\rfloor=-ldA_0$,
Lemma~\ref{lem:normalized-cyclic-algebra} identifies this
algebra with the normalization of $R_0$ in
$K_0[z]/(z^r-s)$.

Let $\mathfrak W_2$ be the normalized base change of
$\mathfrak W_1$ by $\pi$, with maps
\[
 \pi_2:\mathfrak W_2\to\mathfrak W_1,\qquad
 g_2:\mathfrak W_2\to\YY,\qquad
 g_1\pi_2=\pi g_2.
\]
Only components dominating $\YY$ are retained.
Uniqueness of integral closure shows that, over an
affine open $\Spf R\subset g_1^{-1}(\mathfrak U)$, its
ring is the integral closure of $R$ in
$\operatorname{Frac}(R)[z]/(z^r-g_1^*s)$.

Near a point of $\Sigma_1$, the ambient space
$\mathfrak W_1$ is regular. Choose local equations
$x_1,\ldots,x_t$ for the components of $B_1$ through
that point, where $t=2$ or $3$. Write
\[
 \begin{gathered}
 B_1=V(x_1\cdots x_t),\qquad
 P_1=\sum_{\alpha=1}^t\rho_\alpha V(x_\alpha),\\
 g_1^*s=u\prod_{\alpha=1}^t x_\alpha^{a_\alpha},
 \qquad a_\alpha=dr\rho_\alpha\in\Z_{>0},\quad u\in R^*.
 \end{gathered}
\]
Consequently $\mathfrak W_2$ is locally the normalization of
\begin{equation}\label{eq:tame-monomial-chart}
 \Spf\!\left(R[z]\Big/
        \bigl(z^r-u x_1^{a_1}\cdots x_t^{a_t}\bigr)\right).
\end{equation}
The equations $x_\alpha=0$ define divisors on
$\mathfrak W_1$, and the exponents are the multiplicities
of $drP_1$. This is a normalized base change, not an
identification of ordinary and reflexive pullbacks.

Adjoining an $r$-th root of $u$ is \'{e}tale.
After this extension, \eqref{eq:tame-monomial-chart}
is a normalized monomial cover and hence toroidal
\cite[Proposition 3.7 and Theorem 3.8]{Amb15}.
Apply the projective toroidal resolution of
\cite[Theorem 3.4.9 and Lemma 3.4.6]{IT14},
including a barycentric subdivision if needed to
separate boundary components. Its saturated log blowups
are compatible on overlaps. Their centres lie over
$\Sigma_1$: a saturated one-dimensional cone is already
regular. Extending their ideals by the unit ideal
outside this locus gives a projective morphism
\[
 \mu:\mathfrak W\to\mathfrak W_2,\qquad
 B=(\mu^{-1}\pi_2^{-1}B_1)_{\mathrm{red}},\qquad g=g_2\mu.
\]
The resulting pair is snc over $\Sigma_1$ and plt
elsewhere. Its boundary components are $\Q$-Cartier:
this is clear on the snc locus; on the plt locus the
components over a given $B_{1,i}$ are disjoint and their
positive weighted sum is $\pi_2^*B_{1,i}$.

Tame ramification and the equality $K+D=0$ for a toric
variety with its full boundary give
\[
 K_{\mathfrak W}+B
   =\mu^*(K_{\mathfrak W_2}+(\pi_2^{-1}B_1)_{\mathrm{red}})
   =g^*(K_{\YY}+A)\quad\text{Cartier}.
\]
The pair is dlt. Since $B$ is $\Q$-Cartier, every
divisorial valuation centred in $B$ satisfies
\[
 A_E(\mathfrak W,0)
      =A_E(\mathfrak W,B)+\ord_E(B)>0.
\]
Away from $B$ the klt locus is unchanged. Thus
$\mathfrak W$ is klt. Adjunction shows that its boundary
components have only rational double points away from
the double curves.

\smallskip
\noindent\emph{Step 3. The full pullback and the rational components.}
Let $\eta=h_0\pi g:\mathfrak W\to\widehat V_T$ and set
$P=\eta^*T=\mu^*\pi_2^*P_1$. The coefficients of $P_1$
are positive. Finite pullback preserves this property,
and a new toric divisor $E$ has coefficient
\[
             \ord_E(P)=\sum_{\alpha=1}^t
                                \rho_\alpha\ord_E(x_\alpha)>0.
\]
Consequently
\[
 P=\sum_i r_iB_i,\quad r_i>0,\qquad
 B=g^{-1}(A)_{\mathrm{red}},\qquad NP\text{ is Cartier}.
\]
The morphism $\eta|_B$ factors through $T$, so the prescribed
trivialization pulls back to
\[
 \OO_{\mathfrak W}(NP)|_B
       \simeq(\eta|_B)^*(\OO_V(NT)|_T)\simeq\OO_B.
\]
This proves (iii). The fibres of $g$ are connected because
$g$ is proper birational and $\YY$ is normal; hence $B$
is connected. Adjunction gives
$\omega_B\simeq(g|_B)^*\omega_A\simeq\OO_B$, and $B$ is
a Gorenstein slc surface.

To prove rationality, let $V_i$ be a component of $A^\nu$,
with conductor $C_i$, and choose $B_j$ dominating $V_i$.
There is a proper birational morphism $b_j:B_j\to V_i$.
Taking residues of compatible rational log canonical
forms in the crepant equality of Step 2 gives
\[
 K_{B_j}+C_{B_j}=b_j^*(K_{V_i}+C_i),\qquad
                         C_{B_j}=(B-B_j)|_{B_j}.
\]
Thus the strict transform of any component of $C_i$ is a
component of $C_{B_j}$. It is a smooth double curve on
$B$, birational to a rational conductor curve on $V_i$,
and hence isomorphic to $\mathbb P^1$.
Lemma~\ref{lem:nonnormal-conductor-types}, applied to the
connected surface $B$, implies that all its normalized
conductor curves and all its normal components are rational.
This completes (i).

\smallskip
\noindent\emph{Step 4. Direct images and formal completion.}
The finite covers above admit local algebraic representatives.
Indeed, trivialize the invertible sheaf defining a cyclic
cover and choose an algebraic generator of its prescribed
power. The remaining factor is
a formal unit $v$. Lift its restriction to the reduced
closed fibre to an algebraic unit $v_0$.
For the relevant root order $q$, which is prime to $p$,
\[
 v/v_0\in1+\mathcal K,\qquad v/v_0=b^q
                         \quad\text{for a unique }b\in1+\mathcal K,
\]
where $\mathcal K$ is the ideal of that closed fibre.
Rescaling the root by $b$ identifies the formal cyclic cover
with the completion of the algebraic one. Normalization
commutes with these excellent completions. The normalized
blowups in Step 2 also commute with completion and with
strict \'{e}tale changes of charts, by the cited
Lemma 3.4.6. It therefore suffices to compute the coherent
direct images on these local algebraic representatives:
completion then gives the formal direct images by the
theorem on formal functions \cite[Tag 02OC]{Sta}.
No global algebraization of the formal cover is required.

Write $g:W\to Y$ for such a representative, retaining
$A,B$ for its boundary divisors, and set
$\mathcal M=\OO_Y(K_Y+A)$. Both $W$ and $Y$ are klt
threefolds. Rationality and duality give
\cite[Theorems 3 and 19]{BK23}
\[
 Rg_*\OO_W\simeq\OO_Y,\qquad Rg_*\omega_W\simeq\omega_Y.
\]
Since $K_W+B=g^*(K_Y+A)$, the projection formula gives
\[
 \begin{aligned}
 Rg_*\OO_W(-B)
   &=Rg_*(\omega_W\otimes g^*\mathcal M^{-1})\\
   &\simeq\omega_Y\otimes\mathcal M^{-1}=\OO_Y(-A).
 \end{aligned}
\]
These identifications agree with the natural inclusions
of the divisorial ideals into the structure sheaves.
Applying $Rg_*$ to
\[
              0\to\OO_W(-B)\to\OO_W\to\OO_B\to0
\]
and comparing with
$0\to\OO_Y(-A)\to\OO_Y\to\OO_A\to0$ therefore yields
$R(g|_B)_*\OO_B\simeq\OO_A$. This proves (ii).
\end{proof}

\subsection{The intersection calculation}

Fix the model $g:(\mathfrak W,B)\to(\YY,A)$ of
Lemma~\ref{lem:rational-crepant-modification}, and retain
the full pullback
\[
 B=\bigcup_{i\in I}B_i,\qquad
 P=\sum_{i\in I}r_iB_i,\qquad
 r_i>0,\qquad \OO_{\mathfrak W}(NP)|_B\simeq\OO_B.
\]
For every double curve $C\subset B$, the last equality
gives $P\cdot C=0$. We first express this equality using
the self-intersections of $C$ on the two components
containing it. We then apply Hodge index on the $B_i$ to
bound the first Betti number of their dual complex.

Let $E$ index the irreducible components of all
$B_i\cap B_j$ with $i\ne j$, and let $F$ index the triple
points. The dual complex $\Delta$ has vertices $I$, edges
$E$, and triangular faces $F$. For $e\in E$ joining
$i,j$, write
\[
 C_e\subset B_i\cap B_j,\qquad
 C_{e,i}=C_e\subset B_i,\qquad
 C_{e,j}=C_e\subset B_j,\qquad
 s_{i,e}=(C_{e,i})^2_{B_i}.
\]
Thus an edge records one irreducible curve, even when
$B_i\cap B_j$ has several components. If a triple point
$P_f$ lies on $C_e$, denote by $k(f,e)$ the index of the
third component of $B$ through $P_f$. The notation
$f\supset e$ below means precisely $P_f\in C_e$.

\begin{lemma}\label{lem:weighted-edge-relation}
The positive multiplicities $r_i$ in the full pullback
of $T$ satisfy, for every edge $e=(i,j)$,
\begin{equation}\label{eq:weighted-edge}
             r_js_{i,e}+r_is_{j,e}
                         +\sum_{f\supset e}r_{k(f,e)}=0.
\end{equation}
\end{lemma}

\begin{proof}
The ambient space is regular near $C_e$. Since
$C_e$ is one component of $B_i\cap B_j$, and distinct
components of this intersection are disjoint, one has
\[
 \deg\OO(B_i)|_{C_e}=s_{j,e},\qquad
 \deg\OO(B_j)|_{C_e}=s_{i,e}.
\]
Every remaining $B_k$ meets $C_e$ transversely at the
triple points with $k(f,e)=k$. Consequently
\[
 0=\frac1N\deg\OO_{\mathfrak W}(NP)|_{C_e}
   =r_is_{j,e}+r_js_{i,e}+\sum_{f\supset e}r_{k(f,e)}.
\]
\end{proof}

\begin{lemma}\label{lem:weighted-hodge}
Suppose $\Delta$ is a closed connected surface and
\eqref{eq:weighted-edge} holds with all $r_i>0$.
Then $b_1(\Delta,\RR)\leq2$.
\end{lemma}

\begin{proof}
Set $v=|I|$. For an oriented edge $e=(i,j)$ let
$z_{ij,e}$ be the coefficient of
$z\in Z^1(\Delta,\RR)$, so $z_{ji,e}=-z_{ij,e}$.
On a face $(i,j,k)$ the cocycle identity is
$z_{ij}+z_{jk}=z_{ik}$; the face specifies the three edges.
Define
\[
 L_i(z)=\sum_{e=(i,j)}r_jz_{ij,e}[C_{e,i}],\qquad
 L:Z^1(\Delta,\RR)\longrightarrow
                   \mathcal V:=\bigoplus_iN^1(B_i)_{\RR}.
\]

\noindent\emph{Step 1. Agreement of intersections along a double curve.}
The boundary is snc near $C_e$. Hence
\[
 \begin{aligned}
 L_i(z)\cdot C_{e,i}
  &=r_jz_{ij,e}s_{i,e}
                 +\sum_{f=(i,j,k)\supset e}r_kz_{ik,f}\\
  &=-r_iz_{ij,e}s_{j,e}
                 +\sum_{f=(i,j,k)\supset e}r_kz_{jk,f}
   =L_j(z)\cdot C_{e,j}.
 \end{aligned}
\]
The second equality uses \eqref{eq:weighted-edge}
and the cocycle identity. Thus
\begin{equation}\label{eq:edge-pairing-equality}
                  L_i(z)\cdot C_{e,i}=L_j(z)\cdot C_{e,j}.
\end{equation}

\smallskip
\noindent\emph{Step 2. The dimension of the image.}
On $\mathcal V$ set $Q(u,w)=\sum_i r_i(u_i\cdot w_i)$.
Hodge index gives $v$ positive eigenvalues.
For $z,w\in Z^1(\Delta,\RR)$,
\[
 Q(L(z),L(w))=
 \sum_{e=(i,j)}r_ir_jz_{ij,e}
    \bigl(L_i(w)\cdot C_{e,i}-L_j(w)\cdot C_{e,j}\bigr)=0.
\]
Choose an ample line bundle on $B$ and let $H_i$
be its numerical class on $B_i$.
Its degrees on the two copies of $C_e$ agree, so
$Q(H,L(z))=0$, while $Q(H,H)>0$.
The form on $H^\perp$ has $v-1$ positive eigenvalues.
The totally isotropic space $\im L\subset H^\perp$
therefore satisfies
\begin{equation}\label{eq:hodge-image-bound}
                            \dim\im L\leq v-1.
\end{equation}

\smallskip
\noindent\emph{Step 3. The dimension of the kernel.}
Fix a face $F$. Suppose $z\in\ker L$ and $z|_F=0$.
If $F=(i,j,k)$ and $F'=(i,j,l)$ share $e=(i,j)$,
intersection with $C_{e,i}$ gives
\[
 0=r_jz_{ij}s_{i,e}+r_kz_{ik}+r_lz_{il}=r_lz_{il}.
\]
Since $r_l>0$, $z_{il}=0$; then the cocycle identity
gives $z_{jl}=0$. Thus $z|_{F'}=0$.
The faces of a connected closed surface are connected
by successive common edges, so $z=0$.
Consequently
\[
 \ker L\hookrightarrow Z^1(F,\RR),\qquad\dim\ker L\leq2.
\]
Finally
\[
 v-1+b_1(\Delta,\RR)=\dim Z^1(\Delta,\RR)
         =\dim\im L+\dim\ker L\leq v+1.
\]
\end{proof}

\subsection{The structure sheaf and the Picard scheme}

\begin{proof}[Proof of Proposition~\ref{cor:rational-conductor-picard}]
Take $g:B\to A$ from
Lemma~\ref{lem:rational-crepant-modification}. Write
$I,E,F$ for the components, double curves, and triple
points of $B$, as in the preceding subsection.

\noindent\emph{Step 1. The dual complex and coherent cohomology.}
For each $B_i$, adjunction gives
\[
 C_i=(B-B_i)|_{B_i},\qquad K_{B_i}+C_i\sim0.
\]
The surface $B_i$ is rational, and $C_i$ is a connected
cycle of smooth rational curves by
Lemmas~\ref{lem:log-surface-resolution}
and~\ref{lem:surface-configurations}. Its cycle is the
link of the vertex $i$ in $\Delta$. Each double curve
contains exactly two triple points. Hence $\Delta$ is
a closed connected topological surface.

The alternating restriction maps give the exact complex
\begin{equation}\label{eq:rational-normalization-resolution}
 0\longrightarrow\OO_B
 \longrightarrow\bigoplus_{i\in I}\OO_{B_i}
 \longrightarrow\bigoplus_{e\in E}\OO_{C_e}
 \longrightarrow\bigoplus_{f\in F}\OO_{P_f}
 \longrightarrow0.
\end{equation}
Exactness is local: at a double or triple point it is
the standard restriction complex for $xy=0$ or $xyz=0$.
Away from these points the first map is an isomorphism.
The three terms on the right have no higher cohomology:
the $B_i$ are rational with rational double points, and
$C_e\simeq\mathbb P^1$. Their global sections form
$C^\bullet(\Delta,k)$. Since $Rg_*\OO_B\simeq\OO_A$,
\begin{equation}\label{eq:rational-cohomology-model}
 H^j(A,\OO_A)\simeq H^j(B,\OO_B)\simeq H^j(\Delta,k).
\end{equation}
Frobenius acts on the cochain complex by coefficientwise
$p$-th powers, so these isomorphisms respect Frobenius.

\smallskip
\noindent\emph{Step 2. The Betti number and cup products.}
Serre duality and $\omega_A\simeq\OO_A$ give
$H^2(\Delta,k)=H^2(A,\OO_A)=k$. Since $p\ne2$, the closed
surface $\Delta$ is orientable. Its integral first homology
is free, and Lemmas~\ref{lem:weighted-edge-relation}
and~\ref{lem:weighted-hodge} give
\[
 q:=h^1(A,\OO_A)=b_1(\Delta,k)=b_1(\Delta,\RR)\leq2.
\]
Thus $\Delta$ is a sphere or a torus, and $q=0$ or $2$.
The cup product on $A$ is the perfect alternating pairing
\[
 H^1(A,\OO_A)\times H^1(A,\OO_A)\longrightarrow
 H^2(A,\OO_A)=k
\]
given by Serre duality. The complex in Step 1 is defined
over $\mathbb F_p$, and cup product preserves
Frobenius-fixed classes. If $q=2$, a symplectic basis
$u,v$ over $\mathbb F_p$ gives $u^2=v^2=0$ and
$u\smile v=w$, a basis of $H^2$. If $q=0$, only degrees
zero and two occur. These are precisely the cohomology
algebras of the torus and the sphere, with Frobenius.

\smallskip
\noindent\emph{Step 3. The Picard scheme of $B$.}
Since $H^1(B_i,\OO_{B_i})=0$, the tangent-space
description of the Picard scheme
\cite[Theorem 5.11]{Kle05} gives $\Pic^0_{B_i/k}=0$.
A line bundle in $\Pic^0_{B/k}$ is therefore obtained,
locally on the parameter scheme, by gluing trivial line
bundles on the $B_i$. An oriented double curve
$e=(i,j)$ carries a gluing constant $\lambda_{ij,e}$,
with $\lambda_{ji,e}=\lambda_{ij,e}^{-1}$. At a triple
point the compatibility condition is
\[
                  \lambda_{ij}\lambda_{jk}\lambda_{ki}=1.
\]
Changing the trivialization on $B_i$ by $a_i$ replaces
$\lambda_{ij,e}$ by $a_j a_i^{-1}\lambda_{ij,e}$.
Gluing is effective: locally the union has equation
$xy=0$ or $xyz=0$, and its ring is the corresponding
fibre product of the component rings
\cite[Tag 0D2J]{Sta}.
Consequently the kernel of restriction to the component
Picard schemes is
\begin{equation}\label{eq:picard-gluing-complex}
 \frac{\ker\bigl(\delta^1:\Gm^E\to\Gm^F\bigr)}
      {\im\bigl(\delta^0:\Gm^I\to\Gm^E\bigr)}
 \simeq\operatorname{Hom}\bigl(H_1(\Delta,\Z),\Gm\bigr)
 \simeq\Gm^q.
\end{equation}
The quotient is taken as an fppf sheaf. Here $\delta^0$
is the ratio $a_j/a_i$, and $\delta^1$ is the displayed
product around a face. The first isomorphism is the
degree-one universal coefficient formula, using
$H_0(\Delta,\Z)=\Z$. The last group is connected, so it
is exactly $\Pic^0_{B/k}$.

\smallskip
\noindent\emph{Step 4. Descent to $A$.}
We prove directly that pullback induces
\begin{equation}\label{eq:picard-crepant-descent}
                       g^*:\Pic^0_{A/k}\xrightarrow{\sim}\Pic^0_{B/k}.
\end{equation}
The equality $g_*\OO_B=\OO_A$ persists after every base
change over $k$. By the projection formula, pullback of
line bundles is fully faithful; hence $g^*$ is a
monomorphism of Picard schemes.

Let $R$ be a local Artinian $k$-algebra with residue field
$k$, and let $\mathcal L$ be a line bundle on $B_R$
whose restriction to $B$ is trivial. For every affine
open $U\subset A$, set $W=g^{-1}(U)$. Leray gives
\[
              H^1(W,\OO_W)=H^0(U,R^1g_*\OO_B)=0.
\]
For a small extension $R\to R'$ with kernel $J$,
the units sequence on $W_R$ is
\[
 1\longrightarrow 1+J\OO_W
 \longrightarrow\OO_{W_R}^{\,*}
 \longrightarrow\OO_{W_{R'}}^{\,*}\longrightarrow1,
 \qquad 1+J\OO_W\simeq J\otimes_k\OO_W.
\]
Its cohomology makes $\Pic(W_R)\to\Pic(W_{R'})$
injective. Induction on the length of $R$ therefore
shows that $\mathcal L|_{W_R}$ is trivial.
On overlaps, its transition functions descend uniquely
to $A_R$, because $g_{R*}\OO_{B_R}=\OO_{A_R}$.
They define a line bundle $\mathcal M$ on $A_R$ with
$g_R^*\mathcal M\simeq\mathcal L$, uniquely up to isomorphism.

Thus $g^*$ induces an isomorphism on the completed local
rings at the identity, and is \'{e}tale there. Translation
makes it \'{e}tale everywhere. An \'{e}tale monomorphism
is an open immersion \cite[Tag 025G]{Sta}; its image is
an open subgroup of the connected group $\Pic^0_{B/k}$,
hence the whole group. This proves
\eqref{eq:picard-crepant-descent}. Combining it with
\eqref{eq:picard-gluing-complex} completes the proof.
\end{proof}

\section{Cohomology classes on finite thickenings}\label{app:witt}

Proposition~\ref{prop:comparison-class} requires a class
whose restriction to $D$ is prescribed and which extends
to $D_{\ell_i+1}$. We construct it by subtracting two Witt
$1$-cochains with the same first coordinate. Their difference
is $V_W(b_i)$ for an ordinary meromorphic $1$-cochain $b_i$.
An estimate for $db_i$ makes $b_i$ a cocycle on the required
thickening; its leading coefficient gives the prescribed
restriction to $D$.

We use Setting~\ref{not:finite-calculation} and the choices of
Lemma~\ref{lem:order-congruences}. Thus
\[
 \begin{gathered}
 1\leq n=\ell_0<R<p,\qquad
       p\ell_{i-1}\leq\ell_i<Rp^i\quad(1\leq i\leq e),\\
 \sigma_i|_D=s^{p^i},\qquad v_i=n\theta_i\quad(0\leq i\leq e).
 \end{gathered}
\]
Fix a $G$-stable indexed affine covering $\mathcal U=(U_a)$
of $\YY$, and use the differential $d$ of
Lemma~\ref{lem:boundary-product}.
Let $\mathcal Q_{\YY}$ be the sheaf of total quotient rings
of $\OO_{\YY}$.
For a meromorphic cochain $q$, write
$\ord_{\JJ}(q)\geq b$ if every entry belongs to
$\JJ^b\subset\mathcal Q_{\YY}$; negative integers $b$ are
allowed. All quotient sheaves and their cohomology below
are considered on the common underlying topological space
of $A$, $D_t$ and $\YY$.

\subsection{Comparison with powers of the first trivialization}

The sections $\sigma_i$ need not be powers of $\sigma_0$.
The following estimate bounds their difference after taking
meromorphic local representatives. This allows all the
cohomology classes to be constructed using a single Witt
lift of $dP_0$ on $D_{p(p-n)}$.

\begin{lemma}\label{lem:meromorphic-cochains}
There are meromorphic zero-cochains $P_i$, of character
$\mu^{np^i}$, representing $\sigma_i^{-n}$ and satisfying
\begin{equation}\label{eq:P-orders}
 \begin{gathered}
 \ord_{\JJ}(P_i)=-np^i,\qquad
 dP_i\in C^1(\mathcal U,\JJ^{\ell_i-np^i}),\\
 [s^{np^i-\ell_i}dP_i]=v_i
       \quad\text{in }H^1(D,\OO_D).
 \end{gathered}
\end{equation}
Set $T_0=0$ and $T_i=P_0^{p^i}-P_i$ for $i\geq1$.
Then
\begin{equation}\label{eq:P-initial-comparison}
 \ord_{\JJ}(T_i)\geq-np^{i-1},\qquad
 dT_i\in C^1(\mathcal U,\OO_{\YY})\quad(i\geq1).
\end{equation}
\end{lemma}

In the last line of \eqref{eq:P-orders}, first reduce $dP_i$
modulo $\JJ^{\ell_i-np^i+1}$ and then divide by the generator
$s^{\ell_i-np^i}$ on $D$. Thus the expression denotes an
ordinary cohomology class in $H^1(D,\OO_D)$.

\begin{proof}
\noindent\emph{Step 1. Equivariant lifts of $\sigma_i^{-n}$.}
Fix $0\leq i\leq e$. We lift $\sigma_i^{-n}$ directly as
a section of $\JJ^{-np^i}|_{D_{\ell_i}}$.
On each affine formal open $U_a$, reduction gives a
surjective map
\[
 \Gamma(U_a,\JJ^{-np^i})\longrightarrow
 \Gamma(U_a\cap D_{\ell_i},
                       \JJ^{-np^i}/\JJ^{\ell_i-np^i}).
\]
Choose $Q_{i,a}$ mapping to $\sigma_i^{-n}|_{U_a\cap D_{\ell_i}}$
and set $Q_i=(Q_{i,a})_a$.
The action on zero-cochains is
\[
 (gQ_i)_a=g(Q_{i,g^{-1}a})\qquad(g\in G),
\]
where $gU_a=U_{ga}$ and the right side uses the natural
action on the fractional ideal $\JJ^{-np^i}$.
Since $|G|$ is invertible in $k$, define
\[
 P_i=\frac1{|G|}\sum_{g\in G}\mu(g)^{-np^i}gQ_i
                  \in C^0(\mathcal U,\JJ^{-np^i}).
\]
The character of $\sigma_i^{-n}$ is $\mu^{np^i}$.
Equivariance of reduction therefore gives
\[
 \begin{gathered}
 gP_i=\mu(g)^{np^i}P_i,\qquad
 P_{i,a}\bmod\JJ^{\ell_i-np^i}
                      =\sigma_i^{-n}|_{U_a\cap D_{\ell_i}},\\
 P_{i,a}|_D=s^{-np^i}|_{U_a\cap D}.
 \end{gathered}
\]
The last section generates $\JJ^{-np^i}|_D$.
Nakayama's lemma implies that $P_{i,a}$ generates
$\JJ^{-np^i}|_{U_a}$, and hence
$\ord_{\JJ}(P_i)=-np^i$.

\smallskip
\noindent\emph{Step 2. The connecting class represented by $dP_i$.}
The reductions of $P_{i,a}$ agree on every overlap, so
\[
 (dP_i)_{ab}=P_{i,b}-P_{i,a}
       \in\Gamma(U_a\cap U_b,\JJ^{\ell_i-np^i}).
\]
Let $\delta^0_{\JJ^{-np^i}}$ be the connecting homomorphism
for restriction from $D_{\ell_i+1}$ to $D_{\ell_i}$.
Its value on $\sigma_i^{-n}$ is, by the definition of the
connecting homomorphism,
\[
 [dP_i\bmod\JJ^{\ell_i-np^i+1}]
   =\delta^0_{\JJ^{-np^i}}(\sigma_i^{-n})
        \in H^1(D,\JJ^{\ell_i-np^i}|_D).
\]
Apply Lemma~\ref{lem:obstruction-rules}(ii) with
$q=p^i$, $\ell=\ell_i$ and $b=n$. It gives
\[
 [s^{np^i-\ell_i}dP_i]
   =s^{np^i-\ell_i}\delta^0_{\JJ^{-np^i}}(\sigma_i^{-n})
   =n\theta_i=v_i.
\]
Together with Step~1, this proves \eqref{eq:P-orders}.

\Needspace{12\baselineskip}
\smallskip
\noindent\emph{Step 3. Comparing consecutive representatives.}
Fix $1\leq i\leq e$ and set $N=p\ell_{i-1}$.
Since $N\leq\ell_i$, Step~2 gives
\[
 \begin{aligned}
 dP_i&\in C^1(\mathcal U,\JJ^{N-np^i}),\\
 d(P_{i-1}^p)=(dP_{i-1})^p
      &\in C^1(\mathcal U,\JJ^{N-np^i}).
 \end{aligned}
\]
Thus $P_i$ and $P_{i-1}^p$ reduce to global generating
sections of $\JJ^{-np^i}|_{D_N}$.
Their local ratios and their common reduction are
\[
 u_a:=\frac{P_{i,a}}{P_{i-1,a}^p}\in\Gamma(U_a,\OO_{\YY}^*),
 \qquad u=(u_a\bmod\JJ^N)_a\in H^0(D_N,\OO_{D_N}^*).
\]
Both numerator and denominator have character $\mu^{np^i}$
and restrict to $s^{-np^i}$ on $D$. Therefore
\[
                         g(u)=u\quad(g\in G),\qquad u|_D=1.
\]

If $u\ne1$, let $h$ be the unique integer with
\[
 u-1\in H^0(D_N,\JJ^h/\JJ^N)
             \setminus H^0(D_N,\JJ^{h+1}/\JJ^N),
                         \qquad 1\leq h<N.
\]
Its image in $H^0(D,\JJ^h|_D)=ks^h$ is $\lambda s^h$
for some $\lambda\in k^*$.
Invariance gives $\mu^h=1$, hence $R\mid h$.
Write $h=p^ja$ with $(a,p)=1$. Since $(R,p)=1$,
\[
                 R\mid a,\qquad Rp^j\leq h<N<Rp^i,
                         \qquad 0\leq j<i.
\]
Set $t=N-h>0$. Under
$\JJ^h/\JJ^N=\JJ^h|_{D_t}$, the section $u-1$
restricts to the generator $\lambda s^h$ on $D$.
It therefore trivializes $\JJ^h|_{D_t}$ by Nakayama's lemma,
as in Lemma~\ref{lem:section-filtration}. Thus
\[
 \mathcal L:=\JJ^{p^j}|_{D_t},\qquad
 \mathcal L|_D\simeq\OO_D,\qquad
                  \mathcal L^{\otimes a}\simeq\OO_{D_t}.
\]
The group $K_t=\ker(\Pic(D_t)\to\Pic(D))$ is $p$-primary:
if $p^b\geq t$, then $(1+\JJ/\JJ^t)^{p^b}=1$, and the
units sequence gives $p^bK_t=0$.
Since $[\mathcal L]\in K_t$ and $(a,p)=1$, one has
$\mathcal L\simeq\OO_{D_t}$.
The definition \eqref{eq:orders-definition} now gives
\[
 t\leq\ell_j\leq\ell_{i-1},\qquad
          h=N-t\geq(p-1)\ell_{i-1}.
\]
Consequently $u_a-1\in\JJ^{(p-1)\ell_{i-1}}$ for every $a$.
If $u=1$, then $u_a-1\in\JJ^N$, so the same inclusion holds.
In both cases
\[
 \begin{aligned}
 P_{i-1,a}^p-P_{i,a}&=P_{i-1,a}^p(1-u_a),\\
 \ord_{\JJ}(P_{i-1}^p-P_i)
    &\geq-np^i+(p-1)\ell_{i-1}\geq-np^{i-1},
 \end{aligned}
\]
where the last inequality uses $\ell_{i-1}\geq np^{i-1}$.
We have proved
\begin{equation}\label{eq:consecutive-pole-bound}
            \ord_{\JJ}(P_{i-1}^p-P_i)\geq-np^{i-1}.
\end{equation}

\smallskip
\noindent\emph{Step 4. Comparison with $P_0^{p^i}$.}
The characteristic-$p$ identity
\[
 T_i=P_0^{p^i}-P_i
       =\sum_{r=1}^i(P_{r-1}^p-P_r)^{p^{i-r}}
\]
is an equality of ordinary meromorphic zero-cochains.
For every $1\leq r\leq i$, \eqref{eq:consecutive-pole-bound}
gives
\[
 \ord_{\JJ}\bigl((P_{r-1}^p-P_r)^{p^{i-r}}\bigr)
                         \geq(-np^{r-1})p^{i-r}=-np^{i-1}.
\]
Hence $\ord_{\JJ}(T_i)\geq-np^{i-1}$.
Finally, $\ell_0=n$ and $\ell_i\geq np^i$ imply
\[
 dP_0\in C^1(\mathcal U,\OO_{\YY}),\qquad
 dP_i\in C^1(\mathcal U,\JJ^{\ell_i-np^i})
                         \subseteq C^1(\mathcal U,\OO_{\YY}).
\]
Therefore
$dT_i=(dP_0)^{p^i}-dP_i\in C^1(\mathcal U,\OO_{\YY})$,
which proves \eqref{eq:P-initial-comparison}.
\end{proof}

\Needspace{12\baselineskip}
\subsection{An ordinary cocycle obtained from two Witt cochains}

For $1\leq i\leq e$, set $j=i-1$.
The leading second coordinate of $d[P_j]$ will give $v_j$,
but its first coordinate is $dP_j$. Adding $d[T_j]$ changes
this first coordinate to $(dP_0)^{p^j}$ without changing the
leading second coordinate. We can therefore subtract the
Frobenius image of one regular Witt lift of $dP_0$.
The difference has first coordinate zero, so it uniquely
determines an ordinary cochain by the Witt sequence
\eqref{eq:witt-short-exact}. The following proof carries
out this construction and the two required order estimates.

\begin{lemma}\label{lem:extended-class}
For $1\leq i\leq e$, set $\rho_i=\ell_{i-1}-np^i$.
There exists
\begin{equation}\label{eq:extended-class}
 h_i\in H^1(D_{\ell_i+1},\JJ^{\rho_i}|_{D_{\ell_i+1}}),
 \qquad s^{-\rho_i}(h_i|_D)=v_{i-1}.
\end{equation}
\end{lemma}

\begin{proof}
\noindent\emph{Step 1. A regular Witt cochain with controlled coboundary.}
Set $L_0=p(p-n)$ and $z=dP_0$.
By \eqref{eq:P-orders}, $z$ is a regular additive $1$-cocycle
of character $\mu^n$. Lemma~\ref{lem:witt-lift} gives a
Witt cocycle on $D_{L_0}$ whose first coordinate is the
specified reduction of $z$. Lift its second coordinate
to a regular $1$-cochain $\eta$ on $\mathcal U$, and set
\[
                    \zeta=(z,\eta)\in C^1(\mathcal U,W_2\OO_{\YY}).
\]
These lifts exist because the intersections of $\mathcal U$
are affine. Since $dz=0$ and $\zeta$ is a cocycle modulo
$\JJ^{L_0}$, there is a unique $2$-cochain $\epsilon$ with
\begin{equation}\label{eq:witt-cochain-error}
 d\zeta=V_W(\epsilon),\qquad
                     \epsilon\in C^2(\mathcal U,\JJ^{L_0}).
\end{equation}
This choice of $\zeta$ will be used for every $i$.

\smallskip
\noindent\emph{Step 2. An ordinary cochain and its coboundary.}
Fix $i$ and set
\[
 j=i-1,\qquad M=L_0p^j,\qquad
 \rho=\ell_j-np^{j+1}=\rho_i.
\]
The inequality $\ell_j<Rp^j<np^{j+1}$ gives $\rho<0$.
The meromorphic Witt cochain $d([P_j]+[T_j])$ and the
regular Witt cochain $F_W^j\zeta$ have the same first
coordinate:
\[
 d(P_j+T_j)=d(P_0^{p^j})=z^{p^j}
                            =\operatorname{pr}_1(F_W^j\zeta).
\]
The kernel of $\operatorname{pr}_1$ on $W_2\mathcal Q_{\YY}$
is $V_W\mathcal Q_{\YY}$. There is therefore a unique
$b_i\in C^1(\mathcal U,\mathcal Q_{\YY})$ such that
\begin{equation}\label{eq:ordinary-comparison-cochain}
             V_W(b_i)=d([P_j]+[T_j])-F_W^j\zeta.
\end{equation}
Taking coboundaries, and using $d^2=0$ and
\eqref{eq:witt-cochain-error}, gives
\[
 V_W(db_i)=-F_W^jV_W(\epsilon)=-V_W(\epsilon^{p^j}).
\]
The injectivity of $V_W$ consequently gives
\begin{equation}\label{eq:ordinary-comparison-coboundary}
                 db_i=-\epsilon^{p^j}\in C^2(\mathcal U,\JJ^M).
\end{equation}

\smallskip
\noindent\emph{Step 3. The pole order and the restriction to $D$.}
We show that
\begin{equation}\label{eq:ordinary-comparison-leading}
 \begin{gathered}
 b_i\in C^1(\mathcal U,\JJ^\rho),\\
 b_{i,ab}\equiv P_{j,a}^{p-1}(P_{j,b}-P_{j,a})
                                      \pmod{\JJ^{\rho+1}}.
 \end{gathered}
\end{equation}
Only the second coordinate of $d[P_j]$ contributes to
this congruence, as the following bounds show.
For meromorphic $x,u$, the Witt addition law gives
\begin{equation}\label{eq:witt-polynomial}
 \operatorname{pr}_2([x+u]-[x])=Q_p(x,u),\qquad
 Q_p(x,u)=\frac{(x+u)^p-x^p-u^p}{p}\in\Z[x,u].
\end{equation}
The polynomial is homogeneous of degree $p$, and
\[
                  Q_p(x,u)\equiv x^{p-1}u\pmod{u^2}.
\]
For $x=P_{j,a}$ and $u=P_{j,b}-P_{j,a}$,
\eqref{eq:P-orders} gives
\[
 \ord_{\JJ}(x^{p-h}u^h)
       \geq-np^{j+1}+h\ell_j
       =\rho+(h-1)\ell_j\qquad(1\leq h<p).
\]
As $\ell_j\geq1$, all terms with $h\geq2$ belong to
$\JJ^{\rho+1}$. Thus
\[
 \operatorname{pr}_2(d[P_j])_{ab}
   \equiv P_{j,a}^{p-1}(P_{j,b}-P_{j,a})
                                      \pmod{\JJ^{\rho+1}}.
\]

For $j\geq1$, $dT_j$ is regular and
$\ord_{\JJ}(T_j)\geq-np^{j-1}$. Applying the same polynomial
to $d[T_j]$ therefore gives
\[
 \ord_{\JJ}\bigl(\operatorname{pr}_2(d[T_j])\bigr)
       \geq-n(p-1)p^{j-1}>\ell_j-np^{j+1}=\rho.
\]
The strict inequality follows from
$\ell_j<Rp^j$ and
$Rp\leq p(p-1)<n(p^2-p+1)$.
For $j=0$, the assertion follows from $T_0=0$.
Both first coordinates $dP_j$ and $dT_j$ are regular,
so the cross term $B_p(dP_j,dT_j)$ is regular.
The second coordinate of $F_W^j\zeta$ is also regular.
Since $\rho<0$, these regular terms belong to
$\JJ^{\rho+1}$. Finally, the two cochains on the right of
\eqref{eq:ordinary-comparison-cochain} have identical first
coordinates, so their difference has second coordinate
equal to the ordinary difference of their second coordinates.
This proves \eqref{eq:ordinary-comparison-leading}.

\smallskip
\noindent\emph{Step 4. The class on $D_{\ell_i+1}$.}
Equations \eqref{eq:ordinary-comparison-coboundary} and
\eqref{eq:ordinary-comparison-leading} define a class
\[
       \widetilde h_i=[b_i]\in H^1(\JJ^\rho/\JJ^M).
\]
The leading term of $P_j$ is $s^{-np^j}$.
Since $-\rho-n(p-1)p^j=np^j-\ell_j$, formula
\eqref{eq:P-orders} yields
\[
 s^{-\rho}(\widetilde h_i|_D)
       =[s^{np^j-\ell_j}dP_j]=v_j.
\]
Moreover,
\[
 M-\rho=p^{i+1}-\ell_{i-1}>\ell_i,
\]
because
$\ell_{i-1}+\ell_i<Rp^{i-1}(p+1)<p^{i+1}$.
As these are integers, $M\geq\rho+\ell_i+1$.
Restrict $\widetilde h_i$ by the quotient map
\[
 \JJ^\rho/\JJ^M\longrightarrow
 \JJ^\rho/\JJ^{\rho+\ell_i+1}
            =\JJ^\rho|_{D_{\ell_i+1}}.
\]
Its image is the class $h_i$ in \eqref{eq:extended-class}.
\end{proof}

\bibliographystyle{amsalpha}
\bibliography{ATPC}

\end{document}